\documentclass[11pt,leqno]{amsart}
\usepackage{a4wide}
\usepackage{amssymb,upgreek}
\usepackage{amsmath}
\usepackage{amsthm}
\usepackage{mathrsfs}
\usepackage{bm,euscript,csquotes,comment}
\usepackage[arrow]{xypic}
\usepackage[usenames,dvipsnames]{xcolor}
\usepackage{hyperref}
\usepackage{mathtools}
\usepackage{tikz}
\usepackage{mathbbol}
\usepackage{endnotes,comment}

\usetikzlibrary{calc,decorations.pathmorphing,shapes,arrows}
\hypersetup{colorlinks,  citecolor=ForestGreen,  linkcolor=Mahogany, urlcolor=black}

\theoremstyle{plain}
\makeatletter
\def\cal@symb#1|#2{\expandafter\def\csname #2#1\endcsname{\mathcal{#1}}}
\def\calsymbols#1#2{\@for\@tmpz:=#2\do{\expandafter\cal@symb\@tmpz|{#1}}}
\def\bb@symb#1|#2{\expandafter\def\csname #2#1\endcsname{\mathbb{#1}}}
\def\bbsymbols#1#2{\@for\@tmpz:=#2\do{\expandafter\bb@symb\@tmpz|{#1}}}
\def\bold@symb#1|#2{\expandafter\def\csname #2#1\endcsname{\mathbf{#1}}}
\def\boldsymbols#1#2{\@for\@tmpz:=#2\do{\expandafter\bold@symb\@tmpz|{#1}}}
\def\scr@symb#1|#2{\expandafter\def\csname #2#1\endcsname{\mathscr{#1}}}
\def\scrsymbols#1#2{\@for\@tmpz:=#2\do{\expandafter\scr@symb\@tmpz|{#1}}}
\def\frak@symb#1|#2{\expandafter\def\csname #2#1\endcsname{\mathfrak{#1}}}
\def\fraksymbols#1#2{\@for\@tmpz:=#2\do{\expandafter\frak@symb\@tmpz|{#1}}}

\def\dmth@p#1|{\expandafter\let\csname#1\endcsname\relax
  \expandafter\DeclareMathOperator\csname#1\endcsname{#1}}
\def\operators#1{\@for\@tmpz:=#1\do{\expandafter\dmth@p\@tmpz|}}
\makeatother
\calsymbols{c}{A,B,C,D,E,F,G,H,I,J,K,L,M,N,O,P,Q,R,S,T,U,V,W,X,Y,Z}
\bbsymbols{b}{A,B,C,D,E,F,G,H,I,J,K,L,M,N,O,P,Q,R,S,T,U,V,W,X,Y,Z}
\boldsymbols{bf}{r,q,A,B,C,D,E,F,G,H,I,J,K,L,M,N,O,P,Q,R,S,T,U,V,W,X,Y,Z}
\scrsymbols{s}{A,B,C,D,E,F,G,H,I,J,K,L,M,N,O,P,Q,R,S,T,U,V,W,X,Y,Z}
\fraksymbols{fr}{a,b,c,d,e,f,g,h,i,j,k,l,m,n,o,p,q,r,s,t,u,v,w,x,y,z,F,M,O,R,X,Z}
\operators{SL,Mod,id,End,op,ind,Hom,Ann,Spec,sing,Ext,Sch,AffSch,LRS,Top,mult,coker,pr,Glue,Aut,odd,even,ob,res,Kdim,ann,tr,deg,GL,QCoh}
\DeclareMathAlphabet{\mathpzc}{OT1}{pzc}{m}{it}

\renewcommand{\Im}{\operatorname{im}}

\newcommand{\qmb}[1]{\quad\mbox{#1}\quad}

\newcommand{\Qp}{\bQ_p}

\newcommand{\hsp}{\hspace{0.1cm}}

\newcommand{\be}{\begin{enumerate}}
\newcommand{\ee}{\end{enumerate}}

\newcommand{\triv}{\mathbb{1}}

\newcommand{\hatOm}{\widehat{\Omega}}
\newcommand{\eps}{\epsilon}
\newcommand{\ok}{\overline{k}}
\newcommand{\Sp}{\mathbf{Sp}}

\DeclareFontFamily{U}{matha}{\hyphenchar\font45}
\DeclareFontShape{U}{matha}{m}{n}{
      <5> <6> <7> <8> <9> <10> gen * matha
      <10.95> matha10 <12> <14.4> <17.28> <20.74> <24.88> matha12
      }{}
\DeclareSymbolFont{matha}{U}{matha}{m}{n}\DeclareMathSymbol{\hash}{0}{matha}{"23}	

\newtheorem{theorem}{Theorem}[subsection]
\newtheorem{corollary}[theorem]{Corollary}
\newtheorem{lemma}[theorem]{Lemma}
\newtheorem{remark}[theorem]{Remark}
\newtheorem{proposition}[theorem]{Proposition}

\newtheorem{definition}[theorem]{Definition}

\newtheorem{example}[theorem]{Example}

\date{\today}

\def\anti{\EuScript J}

\title{\textbf{Stability in the category of smooth mod{-$p$} representations of $\SL_2(\Qp)$}}
\author{Konstantin Ardakov, Peter Schneider}
\date{\today}

\address{Mathematical Institute, Woodstock Road, University of Oxford, Oxford OX2 6GG, UK}
\email{ardakov@maths.ox.ac.uk}
\urladdr{http://people.maths.ox.ac.uk/ardakov/}

\address{ Universit\"at M\"unster,  Mathematisches Institut,  Einsteinstr. 62, 48149 M\"unster, Germany}
\email{pschnei@wwu.de}
\urladdr{https://www.uni-muenster.de/Arithm/schneider/}

\begin{document}
\begin{abstract} Let $p \geq 5$ be a prime number and let $G = \SL_2(\Qp)$. Let $\Xi = \Spec(Z)$ denote the spectrum of the centre $Z$ of the pro-$p$ Iwahori Hecke algebra of $G$ with coefficients in a field $k$ of characteristic $p$. Let $\cR \subset \Xi \times \Xi$ denote the support of the pro-$p$ Iwahori $\Ext$-algebra of $G$, viewed as a $(Z,Z)$-bimodule. We show that the locally ringed space $\Xi/\cR$ is a projective algebraic curve over $\Spec(k)$ with two connected components, and that each connected component is a chain of projective lines. For each Zariski open subset $U$ of $\Xi/\cR$, we construct a stable localising subcategory $\cL_U$ of the category of smooth $k$-linear representations of $G$.
\end{abstract}
\maketitle
\tableofcontents

\section{Introduction}

\subsection{Background}  Let $k$ be a field, let $G$ be a $p$-adic reductive group and let $\Mod_k(G)$ denote the category of smooth $k$-linear representations of $G$. The centre $\frZ(G)$ of the category $\Mod_k(G)$ is called the \emph{Bernstein centre} of $G$. When $k$ is the field of complex numbers, $\frZ(G)$ was studied in detail by Bernstein \cite{Ber} and it plays a fundamental role in the classical local Langlands correspondence. Recently there has been interest in the case where $k$ is a field of characteristic $p$ motivated by considerations from the $p$-adic and mod-$p$ Langlands programmes. However in this case, $\frZ(G)$ turned out to be quite small:
in our previous work \cite{AS23}, we showed that $\frZ(G)$ only depends on the centre $Z(G)$ of $G$, and in particular, that $\frZ(G)$ is isomorphic to the finite dimensional $k$-algebra $k[Z(G)]$ whenever $G$ is assumed to be connected and semisimple.

This situation is reminiscent of the fact that the ring of global regular functions $\cO(X)$ on any projective variety $X$ is also rather small. This observation becomes relevant to the Bernstein centre of $G$ when we recall the work of Gabriel \cite{Gab}, who proved that the structure sheaf of a noetherian scheme $X$ can be reconstructed from the category $\QCoh(X)$ of  quasi-coherent sheaves on $X$, by associating with any open subscheme $U$ of $X$ the localizing subcategory $\cL_U$ of $\QCoh(X)$ consisting of sheaves supported on the complement of $U$ in $X$ and by showing that the ring $\cO(U)$ can be recovered from $\QCoh(X)$ as the centre of the quotient category $\QCoh(X) / \cL_U$. In our recent work \cite{AS24}, we generalised Gabriel's construction to the case of an arbitrary Grothendieck category $\cC$, as follows.

Recall that the localizing subcategory $\cL$ of $\cC$ is said to be \emph{stable} if it is stable under essential extensions. The set $\bfL^{st}(\cC)$ of stable localizing subcategories of $\cC$ forms a partially ordered set under reverse inclusion. For $\cL$ and $\cL_1,\cdots,\cL_n$ in $\bfL^{st}(\cC)$ we call $\{\cL_i\}_{1 \leq i \leq n}$ a \emph{covering} of $\cL$ if $\cL = \bigcap_i \cL_i$. This notion makes $\bfL^{st}(\cC)$ into a Grothendieck site, and we proved in \cite{AS24} Thm. 1.1 that the presheaf $\cL \mapsto Z(\cC/\cL)$ on $\bfL^{st}(\cC)$ is in fact a sheaf. When the category $\cC$ is additionally assumed to be locally noetherian, we showed that there is an order-reversing bijection $U \mapsto \cL_U$ between the so-called \emph{stable subsets} of the injective spectrum $\bf{Sp}(\cC)$ of $\cC$, and $\bfL^{st}(\cC)$. We showed in \cite{AS24} Thm. 1.2 that the corresponding presheaf $U \mapsto Z(\cC/\cL_U)$ on $\bf{Sp}(\cC)$ satisfies the sheaf condition with respect to arbitrary coverings. These results suggest that even though $Z(\Mod_k(G))$ may be small and uninteresting, this only reflects the fact that this centre is the ring of global sections of the sheaf formed by the centres of the quotient categories of $\Mod_k(G)$. Of course, this sheaf is only interesting if one can write down sufficiently many stable localizing subcategories of $\Mod_k(G)$.

\subsection{Main results}
The main goal in this paper is to construct a large family of stable localizing subcategories of $\Mod_k(G)$ in the case where $p \geq 5$ and $G$ is the group $\SL_2(\Qp)$. In order to state our results, it will be convenient to introduce an axiomatic framework as follows. Let $\cA$ be a full abelian subcategory of a locally noetherian Grothendieck category $\cC$. We say that $\cC$ is a \emph{thickening} of $\cA$ if, roughly speaking\footnote{see Definition \ref{def:thickening} for the precise definition}, every noetherian object  in $\cC$ has a finite filtration all of whose subquotients lie in $\cA$, and the inclusion functor $i : \cA \to \cC$ has a left exact right adjoint $r : \cC \to \cA$ which restricts to the identity functor on $\cA$ and which respects non-zero objects in $\cC$. We prove that $r$ induces a bijection $r : \Sp(\cC) \stackrel{\cong}{\longrightarrow} \Sp(\cA)$ between the corresponding injective spectra and that the map $\cL \mapsto \cL \cap \cA$ defines a bijection between the set of localizing subcategories of $\cC$ and the set of localizing subcategories of $\cA$. The bijection $r : \Sp(\cC) \stackrel{\cong}{\longrightarrow} \Sp(\cA)$ respects stable subsets and therefore induces an injective map $\bfr: \bfL^{st}(\cC) \to \bfL^{st}(\cA)$; however $\bfr$ is in general not surjective.

\begin{example} \label{ex:SL2Qp} Let $G = \SL_2(\Qp)$ and let $I$ be the pro-$p$ Iwahori subgroup of $G$. Let $\Mod_k^I(G)$ denote the full subcategory of $\Mod_k(G)$ whose objects are generated by their $I$-fixed vectors and let $r(V) = k[G] \cdot V^I$ for all $V$ in $\Mod_k(G)$. We show in $\S$\ref{sec: ThickEx} below that $\Mod_k(G)$ is a thickening of $\Mod_k^I(G)$, and moreover that the compactly induced representation $\ind_I^G(k)$ is in fact a noetherian projective generator of $\Mod_k^I(G)$.\end{example}
Returning to our axiomatic framework, we assume in addition that $\cA$ has a noetherian projective generator $P$. The functor $\Hom_{\cA}(P,-):\cA \to \Mod(H)$ is then an equivalence of categories with quasi-inverse $P \otimes_H - : \Mod(H) \to \cA$, where $H := \End_{\cA}(P)^{\op}$, and in Thm. \ref{thm:stable-A-L} we show that the image of $\bfr$ consists precisely of those localizing subcategories of $\cA$ that are preserved under the functor $P \otimes_H \Ext^1_{\cC}(P,-) : \cA \to \cA$. To understand this condition better we introduce the graded $(H,H)$-bimodule
\begin{equation}\label{eq:ExtStar}
  \Ext_\cC^*(P,P) := \bigoplus_{i=0}^\infty \Ext_\cC^i(P,P) 
\end{equation}
which is in particular a module for $Z \otimes Z := Z \otimes_\mathbb{Z} Z$ where $Z := Z(H)$ is the centre of $H$. We let $\cJ := \Ann_{Z \otimes Z}(\Ext_\cC^*(P,P))$ be its annihilator ideal  and $\cR := V(\cJ) \subseteq \Spec(Z \otimes Z)$ be the corresponding Zariski closed subset. Let $\pi_1, \pi_2 : \cR \rightarrow \Xi := \Spec(Z)$ denote the restrictions of the two projection maps $\Spec(Z \otimes Z) \rightrightarrows \Xi$. We then have a coequaliser diagram
\begin{equation*}
  \xymatrix{
    \cR \ar@<1ex>[r]^-{\pi_1}  \ar@<-1ex>[r]_-{\pi_2} & \Xi \ar[r]^{\bfq} & \Xi/\frR  }
\end{equation*}
in the category of locally ringed spaces (cf.\ \cite{DG} Prop.\ I.1.1.6). To state our main theorem, we introduce the map $\tau : \Sp(\cC) \to \Xi/\frR$ by the commutativity of the following diagram:
\[ \xymatrix{  \Sp(\cA) \ar[d]_{\cong}&& \Sp(\cC) \ar@{.>}[dd]^\tau\ar[ll]^{r}_{\cong}  \\
\Sp(\Mod(H)) \ar[d] && \\
\Spec(H) \ar[r]_(0.45)\varphi & \Spec(Z) = \Xi \ar[r]_(0.6){\bfq} & \Xi / \frR .}\]
The first vertical arrow on the left comes from the equivalence $\Hom_{\cA}(P,-) : \cA \stackrel{\cong}{\longrightarrow} \Mod(H)$ and the map $\varphi$ is given by $\varphi(P) = P \cap Z$ for all $P \in \Spec(H)$. The second vertical arrow comes from \cite{Gab} V $\S$ 4; it is a bijection if $H$ is finitely generated as a module over its centre. Our first main result is then the following
\begin{theorem}[Prop. \ref{prop:preimoftau}] \label{thm:taucts}   Let $\cC$ be a thickening of $\cA$. Suppose that $\cA$ has a noetherian projective generator $P$ and let $H = \End_{\cA}(P)^{\op}$. Assume furthermore that:
\begin{itemize}
  \item[\textbf{(A1)}]   $H$ is finitely generated as a module over its centre $Z$,
  \item[\textbf{(A2)}]   there is an integer $d$ such that $\Ext_{\cC}^j(P,-)|_{\cA} = 0$ for any $j > d$, and
  \item[\textbf{(A3)}]   $Z \otimes Z/\cJ^{d+1}$ is noetherian.
\end{itemize}
Then $\tau : \Sp(\cC) \to \Xi/\cR$ is continuous with respect to the stable topologies.
\end{theorem}
In this generality, $\Xi/\cR$ is just a locally ringed space. Section $\S \ref{sec:QuotSpace}$ of our paper is devoted to the computation of this space in the situation of Example \ref{ex:SL2Qp} above, relying crucially on the calculations of \cite{OS22} on the structure of the pro-$p$ Iwahori $\Ext$-algebra (\ref{eq:ExtStar}). Our second main result is then the following

\begin{theorem}[Thm. \ref{thm: XiRscheme}, Cor. \ref{cor: ConnCompsOfQuotSpace}]  \label{thm:QuotSpaceIntro}Suppose that $p \geq 5$, $G = \SL_2(\Qp)$, $\cC = \Mod_k(G)$ and $\cA = \Mod_k^I(G)$.
\begin{itemize}
  \item[a)]  $\Xi/\cR$ is a scheme.
  \item[b)] $\Xi/\cR$ has two connected components.
  \item[c)] Each connected component of $\Xi/\cR$ is a chain of projective lines.
\end{itemize}
\end{theorem}

\begin{example} When $p = 13$, the scheme $\Xi/\cR$ looks as follows:
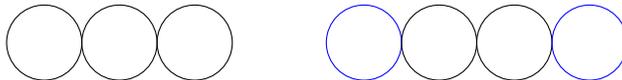
\begin{center}
    \begin{tikzpicture}
    \draw (3.25,-6.5) circle (0.5cm);
    \draw (4.25,-6.5) circle (0.5cm);
    \draw (5.25,-6.5) circle (0.5cm);

    \draw[blue] (7.5,-6.5) circle (0.5cm);
    \draw (8.5,-6.5) circle (0.5cm);
    \draw (9.5,-6.5) circle (0.5cm);
    \draw[blue] (10.5,-6.5) circle (0.5cm);
    \end{tikzpicture}
\end{center}
\end{example}
In Example \ref{ex:p13} below, we sketch how to obtain this gluing from the scheme $\Xi$.

\begin{corollary} \label{cor: tauIntro} Let $p \geq 5$ and let $G = \SL_2(\Qp)$. For every Zariski open subset $U$ of $\Xi/\cR$, $\tau^{-1}(U)$ is a stable open subset of $\bf{Sp}$$(\Mod_k(G))$.
\end{corollary}
\begin{proof} This follows from Thm. \ref{thm:taucts}, once we verify its conditions. Note that $\cC = \Mod_k(G)$ is a thickening of $\cA = \Mod_k^I(G)$, and that $P = \ind_I^G(k)$ is a noetherian projective generator of $\cA$ by Example \ref{ex:SL2Qp}.

\textbf{(A1)}. The pro-$p$ Iwahori-Hecke algebra $H = \End_{\cC}(P)^{\op}$ is finitely generated as a module over its centre by \cite{Vig}; see also \cite{OS18} Cor.\ 3.4 and Remark 3.5.

\textbf{(A2)}. Since $p \geq 5$, the group $I$ has no elements of order $p$, and therefore has finite $p$-cohomological dimension equal to $3$, the dimension of $I$ as a $p$-adic Lie group. Using Frobenius reciprocity, we see that for all $V$ in $\cC$,
\[ \Ext^j_{\cC}(P,V) = \Ext^j_G(\ind_I^G,V) = \Ext^j_I(k,V) = H^j(I,V) = 0\qmb{for all} j > 3.\]

\textbf{(A3)}. In view of Remark \ref{rem: A3klinear} below, we may replace $\textbf{(A3)}$ by the weaker assumption that $Z \otimes_k Z$ is noetherian, because $\cC$ happens to be a $k$-linear category. Prop. \ref{prop: Zcomps} and equation (\ref{eq:Zdecomp}) imply that $Z$ is finitely generated as a $k$-algebra. Hence $Z \otimes_k Z$ is noetherian, by Hilbert's Basis Theorem.  \end{proof}

The paper \cite{DEG} considers, for the group $G = GL(\mathbb{Q}_p)$, the full subcategory $\Mod_{k,\zeta}(G)$ of those representations in $\Mod_k(G)$ which have a fixed central character $\zeta$. They associate with this subcategory a scheme $X$ which is also a chain of projective lines. This scheme is definitively the analogue of a connected component of our quotient space $\Xi/\mathcal{R}$ in this situation. But the idea of \cite{DEG} behind $X$ is completely different: $X$ is viewed as a kind of moduli space of 2-dimensional semisimple Galois representations modulo $p$. From this point of view, the relation between $X$ and the representation theory of $G$ comes from Breuil's semisimple Langlands correspondence modulo $p$. Based on this the paper \cite{DEG} also develops a localisation theory for the category $\Mod_{k,\zeta}(G)$ and shows that the closed points of $X$ parameterize the blocks of the Krull-dimension $0$ subcategory of $\Mod_{k,\zeta}(G)$. We emphasize that in the present paper we work entirely on the representation theoretic side of the full category $\Mod_k(G)$ and our quotient space $\Xi/\mathcal{R}$ arises from the non-vanishing (and the structure) of universal $\Ext$-groups. We therefore like to think that the two papers complement each other in an interesting way.

\textbf{Acknowledgements:} The authors acknowledge support from Deutsche Forschungsgemeinschaft (DFG, German Research Foundation) under SFB 1442, Geometry: Deformations and Rigidity, project-ID 427320536 and  Germany's Excellence Strategy EXC 2044 390685587, Mathematics Münster: Dynamics–Geometry–Structure. The second author also acknowledges support from the Mathematical Institute and Brasenose College, Oxford.

\section{Basic facts about the category $\Mod_k(G)$}
We fix a field $k$ of characteristic $p > 0$. For any locally compact and totally disconnected group $\cG$ we denote by $\Mod_k(\cG)$ the abelian category of smooth $\cG$-representations on $k$-vector spaces. Let $G = SL_2(\mathbb{Q}_p)$. We fix the pro-$p$ Iwahori subgroup $I \subseteq SL_2(\mathbb{Z}_p)$ of all matrices which are upper triangular unipotent mod $p$.

\subsection{Locally noetherian} \label{sec: ThickEx}


Let $\Mod_k^I(G)$ denote the full subcategory of all $V$ in $\Mod_k(G)$ which are generated by their pro-$p$ Iwahori fixed vectors $V^I$. Write $H := \End_{k[G]}(k[G/I])^{\op}$.

\begin{theorem}\label{equivalence}
   The functor
\begin{align*}
  \Mod_k^I(G) & \xrightarrow{\;\simeq\;} \Mod(H) \\
            V & \longmapsto V^I
\end{align*}
is an equivalence of categories with quasi-inverse $M \longmapsto k[G/I] \otimes_{H} M$. Moreover, $k[G/I]$ is projective and faithfully flat as an $H$-module.
\end{theorem}
\begin{proof}
\cite{Koz} and \cite{OS18} Prop.\ 3.25 and its proof.
\end{proof}

\begin{lemma}\label{abelian-sub}
  $\Mod_k^I(G)$ is an abelian subcategory of $\Mod_k(G)$ closed under the formation of subobjects, quotient objects, and arbitrary colimits.
\end{lemma}
\begin{proof}
Clearly $\Mod_k^I(G)$ is closed under the formation of arbitrary direct sums. Let $0 \rightarrow V_0 \rightarrow V_1 \rightarrow V_2 \rightarrow 0$ be an exact sequence in $\Mod_k(G)$ such that $V_1$ lies in $\Mod_k^I(G)$. Obviously, then also $V_2$ lies in $\Mod_k^I(G)$.  For $V_0$ we consider the commutative diagram
\begin{equation*}
  \xymatrix{
     0 \ar[r] & k[G/I] \otimes_H V_0^I \ar[d] \ar[r] & k[G/I] \otimes_H V_1^I \ar[d]_{\cong} \ar[r] & k[G/I] \otimes_H V_2^I \ar[d]_{\cong} & \\
    0 \ar[r] & V_0 \ar[r] & V_1 \ar[r] & V_2 \ar[r] & 0 .  }
\end{equation*}
The upper horizontal row is exact by the left exactness of the functor $(-)^I$ and the fact that $k[G/I]$ is flat as a (right) $H$-module. By the category equivalence in Thm.\ \ref{equivalence} the middle and right perpendicular arrows are isomorphisms. Hence the left one is an isomorphism as well. This shows that $V_0$ lies in $\Mod_k^I(G)$.
\end{proof}

The $k$-algebra $H$ is affine and finitely generated as a module over its centre (\cite{Vig} or, in an explicit form, \cite{OS18} Cor.\ 3.4 and Remark 3.5) and hence noetherian. Therefore the abelian category $\Mod(H)$ is locally noetherian Grothendieck.

\begin{lemma}\label{fin-filt} Let $\cG$ be any locally compact and totally disconnected group with an open pro-$p$ subgroup $J$.
  Any finitely generated $\cG$-representation $V$ in $\Mod_k(\cG)$ has a finite filtration $\{0\} \subset V_1 \subset \ldots \subset V_\ell = V$ by subrepresentations such that each subquotient $V_i /V_{i-1}$, for $0 < i \leq \ell$, is generated by finitely many $J$-fixed vectors.
\end{lemma}
\begin{proof}
It suffices to consider the case where $V$ is generated by a single vector $v$. We can then find an open normal subgroup $N \subseteq J$ such that $N$ fixes $v$. Then $V$ is a quotient of the $\cG$-representation $k[\cG/N]$. Hence it actually suffices to consider the case $V = k[\cG/N]$. In $V$ we then have the $J$-subrepresentation $k[J/N]$. Let $\frm$ denote the augmentation ideal of the ring $k[J/N]$. Since $J/N$ is a finite $p$-group we find an integer $\ell \geq 0$ such that $\frm^\ell = 0$. Observe that $J$ acts trivially on each subquotient $\frm^i / \frm^{i+1}$. We now define $V_i \subseteq V$ as the $\cG$-subrepresentation generated by $\frm^{\ell - i}$. The subquotient $V_i /V_{i-1}$ then is generated by the image of $\frm^{\ell - i} / \frm^{(\ell - i)+1}$, which is contained in $(V_i/V_{i-1})^J$.
\end{proof}

\begin{proposition}\label{loc-noetherian}
   The abelian category $\Mod_k(G)$ is locally noetherian.
\end{proposition}
\begin{proof} By Lemma 1(iv) \cite{Sch}, $\Mod_k(G)$ is a Grothendieck category. We have to show that any finitely generated $G$-representation $V$ in $\Mod_k(G)$ is noetherian. Lemma \ref{fin-filt} reduces us to the case that $V$ is generated by finitely many $I$-fixed vectors. According to Lemma \ref{abelian-sub} any increasing chain of $G$-subrepresentations of $V$ already lies in $\Mod_k^I(G)$. By Thm.\ \ref{equivalence} it therefore corresponds to a chain in the finitely generated $H$-module $V^I$. Hence it must become stationary.
\end{proof}

The exact inclusion functor $\Mod_k^I(G) \subseteq \Mod_k(G)$ is left adjoint to the left exact functor
\begin{align*}
  \Mod_k(G) & \longrightarrow \Mod_k^I(G) \\
          V & \longmapsto V(I) := G\text{-subrepresentation of $V$ generated by $V^I$} (\cong k[G/I] \otimes_{H} V^I) \ .
\end{align*}

\begin{lemma}\label{lem:V(I)}
   The functor $V \mapsto V(I)$
\begin{itemize}
  \item[a)] restricts to the identity functor on $\Mod_k^I(G)$,
  \item[b)] respects non-zero objects,
  \item[c)] respects injective objects, and
  \item[d)] commutes with arbitrary filtered colimits.
\end{itemize}
\end{lemma}
\begin{proof}
a) is obvious. b) holds since $I$ is pro-$p$. c) is a consequence of the functor being right adjoint to a (left) exact functor. For d) let $V = \varinjlim_j V_j$ be a filtered colimit in $\Mod_k(G)$. Since filtered colimits are exact in the Grothendieck category $\Mod_k(G)$ we have the inclusion $\varinjlim_j V_j(I) \subseteq V(I)$. But $V^I = \varinjlim_j V_j^I$. Hence this inclusion is an equality.
\end{proof}

\subsection{Krull dimension}\label{sec: KdimSec}

We briefly recall Gabriel's notion of a Krull dimension for arbitrary Grothendieck categories $\cC$. Gabriel's dimension filtration of $\cC$ is a filtration by localising subcategories $\cC_\alpha$ of $\cC$ indexed by ordinals $\alpha$. His convention is that $\cC_{-1}$ is the subcategory of all zero objects of $\cC$ and $\cC_0$ is the smallest localising subcategory containing all objects of finite length. The $\cC_\alpha$ then are defined successively as follows. If $\alpha = \beta + 1$ then $\cC_\alpha$ is the preimage of $(\cC/\cC_\beta)_0$ under the quotient functor $q_{\cC_\beta} : \cC \rightarrow \cC/\cC_\beta$; if $\alpha$ is a limit ordinal then $\cC_\alpha$ is the smallest localising subcategory containing all $\cC_\beta$ for $\beta < \alpha$. This process terminates as soon as $\cC_/\cC_\alpha$ has no simple objects. But by \cite{Gab} Prop.\ 7 on p.\ 387 the latter implies that $\cC_\alpha$ must be equal to $\cC$. The Krull dimension $\kappa(\cC)$ of $\cC$ can now be defined as the smallest $\alpha$ such that $\cC = \cC_\alpha$. Similarly the Krull dimension $\kappa(Y)$ of an object $Y$ in $\cC$ is defined to be the smallest $\alpha$ such that $Y$ lies in $\cC_\alpha$. Correspondingly the set
\begin{equation*}
  \Sp(\cC) = \dot{\bigcup}_{\alpha < \kappa(\cC)} \Sp_\alpha(\cC)
\end{equation*}
is stratified by the subsets $\Sp_\alpha(\cC)$ of all $[E] \in \Sp(\cC)$ such that $q_{\cC_\alpha}(E)$ contains a simple object (\cite{Gab} p.\ 383). In general it is not possible to read off the Krull dimension $\kappa(E)$ of an $[E] \in \Sp_\alpha(\cC)$ from the $\alpha$.

\begin{remark}\label{rem:stable-Krull}
  Let $\alpha < \kappa(\cC)$ be an ordinal such that $\cC_{\alpha + 1}$ is a stable localising subcategory; then $\kappa(E) = \alpha + 1$ for any $[E] \in \Sp_\alpha(\cC)$.
\end{remark}
\begin{proof}
By our assumption $E$ is $\cC_\alpha$-closed and contains a nonzero subobject $S$ lying in $\cC_{\alpha + 1}$ (cf. \cite{Gab} p.\ 383). Since $E$ is an essential extension of any of its nonzero subobjects (cf.\ \cite{Gab} Prop.\ 11 on p.\ 361) it follows from the stability of $\cC_{\alpha + 1}$ that $E$ lies in $\cC_{\alpha + 1}$.
\end{proof}

\begin{remark}\label{rem:locnoeth-has-Krull}
(\cite{Gab} Prop.\ 7 on p.\ 387) Suppose that $\cC$ is locally noetherian. Then $\cC$ has a Krull dimension and, for any non-limit ordinal $\alpha$, the category $\cC_\alpha/\cC_{\alpha - 1}$ is locally finite.
\end{remark}

For the convenience of the reader we point out the following elementary consequences of the fact that the subcategories $\cC_\alpha$ are localising:
\begin{itemize}
  \item[--] If $0 \rightarrow U_1 \rightarrow U_2 \rightarrow U_3 \rightarrow 0$ is a short exact sequence in $\cC$ then
\begin{equation*}
  \kappa(U_2) = \sup(\kappa(U_1),\kappa(U_3)) \ .
\end{equation*}
  \item[--] If $U$ in $\cC$ is the inductive limit of a family of subobjects $U_i$ then $\kappa(U) = \sup_i \kappa(U_i)$. In particular, for any $U$, we have $\kappa(U) = \sup_i \kappa(U_i)$ where the $U_i$ run over all noetherian subobjects of $U$.
\end{itemize}
In \cite{MCR} Chap.\ 6 the notion of Krull dimension for the category $\cC = \Mod(R)$ of modules over a noetherian ring $R$ is introduced using posets. By loc.\ cit.\ Lemmas 6.2.4 and 6.2.17 it also has the above two properties. Therefore it follows, e.g., from \cite{GR} Prop.\ 2.3 (beware that this reference shifts the Gabriel definition by 1) that these two notions of Krull dimension coincide for noetherian rings $R$.

\begin{proposition}\label{prop:Krull-G}
   The categories $\Mod_k(G)$ and $\Mod(H)$ have Krull dimension one.
\end{proposition}
\begin{proof}
The computation of $Z(H)$ in \cite{OS18} \S 3.2.4 (together with \cite{MCR} Cor.\ 6.4.8) shows that $Z(H)$ has Krull dimension one. By \cite{OS18} Cor.\ 3.4 the centre $Z(H)$ contains a polynomial ring $k[\zeta]$ over which $H$ is a finitely generated free module. Using \cite{MCR} Cor.\ 6.5.3 it then follows that $H$ has Krull dimension one as well. Hence $\Mod_k^I(G) \simeq \Mod(H)$ has Krull dimension one. By the proof of Lemma \ref{abelian-sub} the subcategory $\Mod_k^I(G)$ is closed under the passage to subobjects in $\Mod_k(G)$. Therefore the Krull dimensions of an object in $\Mod_k^I(G)$, when viewed in $\Mod_k^I(G)$ or in $\Mod_k(G)$, coincide.

Since the subcategory $\Mod_k(G)_1$ is closed under extensions it follows from Lemma \ref{fin-filt} that every finitely generated representation lies in $\Mod_k(G)_1$. The assertion about $\Mod_k(G)$ then is immediate from the fact that $\Mod_k(G)_1$ also is closed under inductive limits.
\end{proof}

\section{Thickenings}\label{sec:thickening}

It will be useful to axiomatize the situation in the previous section.

\subsection{Definitions} For this we fix a locally noetherian Grothendieck category $\cC$ together with a full abelian subcategory $\cA$ of $\cC$.

\begin{definition}\label{def:thickening}
   The category $\cC$ is called a \emph{thickening} of $\cA$ if:
\begin{itemize}
  \item[a)] $\cA$ is closed under the formation of subobjects, quotient objects, and arbitrary colimits in $\cC$;
  \item[b)] the inclusion functor $\cA \subseteq \cC$ has a left exact right adjoint functor $r = r_{\cC, \cA} : \cC \rightarrow \cA$ such that:
\begin{itemize}
  \item[-] the functor $r : \cC \xrightarrow{r} \cA \xrightarrow{\subseteq} \cC$ is a subfunctor of the identity functor $\id_\cC$,
  \item[-] $r | \cA = \id_\cA$,
  \item[-] $r$ preserves non-zero objects, and
  \item[-] $r$ commutes with arbitrary filtered colimits;
\end{itemize}
  \item[c)] every noetherian object $V$ in $\cC$ has a finite filtration $0 = V_0 \subseteq V_1 \subseteq \ldots \subseteq V_m = V$ such that all subquotients $V_j/V_{j-1}$, for $1 \leq j \leq m$, lie in $\cA$.
\end{itemize}
\end{definition}

In the following we fix an $\cA$ as in the above definition. Note that $\cA$ is a strictly full subcategory. For any $V$ in $\cC$ the subobject $r(V)$ is the largest subobject of $V$ which lies in $\cA$. The functor $r$ in fact gives rise to a whole sequence of subfunctors $r_j : \cC \rightarrow \cC$, for $j \geq 0$, of the identity functor $\id_\cC$ which are defined inductively as follows. We define
\begin{equation*}
r_0 : \cC \xrightarrow{r} \cA \xrightarrow{\subseteq} \cC \quad \mbox{and} \quad  r_j(V)/r_{j-1}(V) = r_0(V/r_{j-1}(V))  \quad\text{for $j \geq 1$}.
\end{equation*}
Obviously we have, for any $V$ in $\cC$, the increasing filtration by subobjects
\begin{equation*}
  r(V) \subseteq r_1(V) \subseteq \ldots \subseteq r_j(V) \subseteq \ldots \subseteq V \ .
\end{equation*}

\begin{lemma}\label{lem:preradical}
  For any $j \geq 0$ the functor $r_j$ has the following properties:
\begin{itemize}
  \item[a)] $r_j$ is left exact;
  \item[b)] $r_j \circ r_i = r_i \circ r_j = r_j$ for any $i \geq j$;
  \item[c)] $r_j$ preserves filtered unions of subobjects.
\end{itemize}
\end{lemma}
\begin{proof}
In \cite{Ste} VI\S 1 the functor $r_0$ is called a \emph{left exact idempotent preradical} of $\cC$ and the asserted properties of the $r_j$ can be found in the Exercises 1 and 2 of that Chap.\ VI.
\end{proof}

Correspondingly we introduce, for any $j \geq 0$, the strictly full subcategory
\begin{equation*}
  \cA_j := \ \text{all objects $V$ such that $r_j(V) = V$}
\end{equation*}
of $\cC$. We have $\cA_0 = \cA$, and $\cA_j$ is a subcategory of $\cA_{j+1}$. Each functor $r_j$ may be viewed as a functor $r_j : \cC \rightarrow \cA_j$.

\begin{lemma}\label{lem:r-A}
  For any $j \geq 0$ we have:
\begin{itemize}
  \item[a)] $\cA_j$ is an abelian subcategory of $\cC$ closed under the formation of subobjects, quotient objects, and arbitrary colimits in $\cC$;
  \item[b)] $r_j : \cC \rightarrow \cA_j$  is right adjoint to the inclusion functor $\cA_j \subseteq \cC$;
  \item[c)] $r_j | \cA_j = \id_{\cA_j}$;
  \item[d)] $r_j$ preserves non-zero objects;
  \item[e)] $r_j$ commutes with arbitrary filtered colimits;
  \item[f)] any object $V$ in $\cA_j$ has the finite filtration $0 \subseteq r_0(V) \subseteq r_1(V) \subseteq \ldots \subseteq r_j(V) = V$ whose subquotients $r_i(V)/r_{i-1}(V)$ all lie in $\cA$.
\end{itemize}
\end{lemma}
\begin{proof}
a) follows from \cite{Ste} Propositions VI.1.2 and VI.1.7. For b) we observe that for any homomorphism $f : V_0 \rightarrow V_1$ in $\cC$ we have $r_j(f) = f|r_j(V_0)$. c), d), and f) hold by construction. e) We proceed by induction with respect to $j$. For $j = 0$ the claim is part of the Def.\ \ref{def:thickening}. Now assume that the claim holds for some $j$. Let $V = \varinjlim_i V_i$ be a filtered colimit in $\cC$. We then have
\begin{align*}
  r_{j+1}(V)/r_j(V) & = r_0(V/r_j(V)) = r_0(\varinjlim_i V_i / r_j(\varinjlim_i V_i)) = r_0(\varinjlim_i V_i / \varinjlim_i r_j(V_i))  \\
                    & = r_0(\varinjlim_i (V_i / r_j(V_i))) = \varinjlim_i r_0(V_i / r_j(V_i)) = \varinjlim_i (r_{j+1}(V_i) / r_j(V_i))  \\
                    & = \varinjlim_i r_{j+1}(V_i) / \varinjlim_i r_j(V_i) = (\varinjlim_i r_{j+1}(V_i)) / r_j(V) \ ,
\end{align*}
where we use for equations $4$ and $7$ that filtered colimits in $\cC$ are exact. It follows that $r_{j+1}(V) = \varinjlim_i r_{j+1}(V_i)$.
\end{proof}

\begin{lemma}\label{lem:filt}
 Let $0 \rightarrow X \rightarrow Y \rightarrow Z \rightarrow 0$ be a short exact sequence in $\cC$ such that $X$ lies in $\cA_j$ for some $j \geq 0$ and $Z$ lies in $\cA$; then $Y$ lies in $\cA_{j+1}$.
\end{lemma}
\begin{proof}
Since $r_j$ is left exact by Lemma \ref{lem:preradical}.a, we have the commutative exact diagram
\begin{equation*}
  \xymatrix{
             & 0 \ar[d]           & 0 \ar[d]    \\
    0 \ar[r] & X = r_j(X) \ar[d] \ar[r] & X \ar[d] \ar[r] & 0 \ar[d] \ar[r]  & 0 \\
    0 \ar[r] & r_j(Y) \ar[r] & Y \ar[d] \ar[r] & Y/r_j(Y) \ar[r] & 0  \\
    & & Z \ar[d] \ar@{-->}[ur]  \\
    & & 0 ,          }
\end{equation*}
which exhibits $Y/r_j(Y)$ as a quotient object of $Z$. Since $\cA$ is closed under quotient objects it follows that $r_0(Y/r_j(Y)) = Y/r_j(Y)$. On the other hand, by construction we have $r_{j+1}(Y)/r_j(Y) = r_0(Y/r_j(Y))$. It follows that $r_{j+1}(Y) = Y$.
\end{proof}

\begin{lemma}\label{lem:union}
   For any $V$ in $\cC$ we have $V = \bigcup_{j \geq 0} r_j(V)$.
\end{lemma}
\begin{proof}
Write $V$ as a filtered union of noetherian subobjects. Using Lemma \ref{lem:preradical}.c we see that it suffices to prove the assertion for a noetherian $V$. But in this case it follows from Def.\ \ref{def:thickening}.c and an iterated application of Lemma \ref{lem:filt} that $V = r_j(V)$ for sufficiently large $j$.
\end{proof}

\begin{proposition}\label{prop:Posits}
  $\cA_j$, for any $j \geq 0$, is a locally noetherian Grothendieck category.
\end{proposition}
\begin{proof}
By \cite{BP} Lemma 3.4 the category $\cA_j$ is Grothendieck. That it is locally noetherian is clear.
\end{proof}

\begin{remark}\label{rem:Krulldim}
 Let $V$ be an object in $\cA_j$ for some $j \geq 0$; then the Krull dimensions of $V$ viewed in $\cA_j$ and viewed in $\cC$ coincide.
\end{remark}
\begin{proof}
This is immediate from $\cA_j$ being closed under the passage to subobjects in $\cC$.
\end{proof}

\subsection{A spectral sequence}\label{subsec:ss}

The Ext-groups in $\cC$ and $\cA$ are related by the following spectral sequence.

\begin{proposition}\label{prop:ss}
  For objects $A$ in $\cA$ and $V$ in $\cC$ we have the Grothendieck spectral sequence
\begin{equation*}
  \Ext_\cA^i(A, R^j r(V)) \Longrightarrow \Ext_\cC^{i+j}(A,V) \ .
\end{equation*}
\end{proposition}
\begin{proof}
We have the adjunction $\Hom_\cA(A, r(V)) = \Hom_\cC(A,V)$. Since $r$, being right adjoint to a (left) exact functor, preserves injective objects, this equation extends to the asserted composed functor spectral sequence.
\end{proof}

\begin{corollary}\label{cor:ss-proj}
  Suppose that $P$ is a projective object in $\cA$. For any object $V$ in $\cC$ and any $j \geq 0$ we then have
\begin{equation*}
  \Ext_\cC^j(P,V) = \Hom_\cA(P, R^j r(V)) \ .
\end{equation*}
\end{corollary}

\subsection{Injective spectra}\label{subsec:inj-spec}

We first recall some standard notations. For any object $Y$ in a locally noetherian Grothendieck category $\cD$ a choice of injective hull of $Y$ is denoted by $E_\cD(Y)$ or simply $E(Y)$. The injective spectrum $\Sp(\cD)$ is the collection of isomorphism classes $[Y]$ of indecomposable injective objects $Y$ of $\cD$; it is a set, for example, by \cite{Her} p.\ 523.

In the following we keep the setting of a thickening $\cC$ of $\cA$.

\begin{lemma}\label{lem:ess-hull-indec}
  For any $V$ in $\cC$ we have:
\begin{itemize}
  \item[a)] The inclusion $r(V) \subseteq V$ is an essential extension. In particular, if $V$ is injective then $V$ is an injective hull of $r(V)$.
  \item[b)] If $V$ is injective in $\cA$ then $V \cong r(E_\cC(V))$.
  \item[c)] If $V$ is injective in $\cC$ then $V$ is indecomposable if and only if $r(V)$ is indecomposable.
\end{itemize}
\end{lemma}
\begin{proof}
i. Let $U \subseteq V$ be a non-zero subobject. Since $r$ preserves non-zero objects we then have $0 \neq r(U) \subseteq r(V)$ and hence $U \cap r(V) \neq 0$.

ii. We obviously have $V = r(V) \subseteq r(E_\cC(V)) \subseteq E_\cC(V)$. Since $r$, as a right adjoint of a (left) exact functor, preserves injective objects both terms in the essential extension $V = r(V) \subseteq r(E_\cC(V))$ are injective in $\cA$. Hence we must have equality.

iii. Suppose that $r(V) = U_1 \oplus U_2$ with $U_i \neq 0$ (and necessarily lying in $\cA$). Using \cite{Ste} Prop.\ V.2.6 we see that $V$, being an injective hull of $r(V)$ by i., is isomorphic to the direct sum of injective hulls of $U_1$ and $U_2$ and hence is decomposable. On the other hand if $V = V_1 \oplus V_2$ with $V_i \neq 0$ then $r(V) = r(V_1) \oplus r(V_2)$ with $r(V_i) \neq 0$.
\end{proof}

The above lemma implies that the map
\begin{align*}
  \Sp(\cC) & \xrightarrow{\;\simeq\;} \Sp(\cA) \\
             [E] & \longmapsto r([E]) := [r(E)]
\end{align*}
is a bijection with inverse $[U] \longmapsto [E_\cC(U)]$. In fact, because of Lemma \ref{lem:r-A} all of the above remains valid for each functor $r_j$. Hence we have the bijections
\begin{align}\label{f:bij}
  \Sp(\cC) & \xrightarrow{\;\simeq\;} \Sp(\cA_j)  \xrightarrow{\;\simeq\;} \Sp(\cA) \\
             [E] & \longmapsto [r_j(E)] \longmapsto [r(E)] \ .        \nonumber
\end{align}

\subsection{Localizing subcategories}\label{subsec:localising}

We recall that a full subcategory $\cL$ of a locally noetherian Grothendieck category $\cD$ is called \emph{localising} if it is closed under the
formation of subobjects, quotient objects, extensions, and arbitrary direct sums. In particular, it is strictly full and contains the zero object (as the empty direct
sum). It will be technically useful to also recall the following fact. Let $\cD^{noeth}$ denote the full subcategory of all noetherian objects in $\cD$.

\begin{proposition}\label{prop:loc-Serre}
The map
\begin{align*}
  \text{collection of all localising\,} & \xrightarrow{\;\simeq\;} \text{collection of all Serre} \\
  \text{subcategories of $\cD$} &                    \text{\qquad subcategories of $\cD^{noeth}$}  \\
                          \cL & \longmapsto \cL \cap \cD^{noeth}
\end{align*}
is a bijection; its inverse sends a Serre subcategory $\cS$ to the smallest localising subcategory $\langle \cS \rangle$ of $\cD$ which contains $\cS$; equivalently $\langle \cS \rangle$ is the full subcategory of all filtered colimits of objects in $\cS$.
\end{proposition}
\begin{proof}
\cite{Her} Thm.\ 2.8.
\end{proof}

\begin{lemma}\label{lem:AcapL}
   If $\cL$ is a localising subcategory of $\cC$ then $\cA \cap \cL$ is localising in $\cA$.
\end{lemma}
\begin{proof}
Each of the categories $\cA$ and $\cL$ is closed under subobjects, quotient objects, and arbitrary direct sums. Therefore the same holds for $\cA \cap \cL$. It remains to consider extensions. Let $0 \rightarrow X \rightarrow Y \rightarrow Z \rightarrow 0$ be a short exact sequence in $\cA$ such that $X$ and $Z$ lie in $\cL$. Then $Z$ must lie in $\cL$ and hence in $\cA \cap \cL$.
\end{proof}

\begin{proposition}\label{prop:loc-A-C}
The map
\begin{align*}
  \text{collection of all localising\,} & \xrightarrow{\;\simeq\;} \text{collection of all localising} \\
  \text{subcategories of $\cC$} &                    \text{\qquad subcategories of $\cA$}  \\
                          \cL & \longmapsto \cA \cap \cL
\end{align*}
is a bijection; its inverse sends a localising subcategory $\cK$ of $\cA$ to the smallest localising subcategory $\langle \cK \rangle$ of $\cC$ which contains $\cK$.
\end{proposition}
\begin{proof}
We obviously have $\cA^{noeth} = \cA \cap \cC^{noeth}$. Hence Prop.\ \ref{prop:loc-Serre} reduces the asserted bijectivity to the bijectivity of the map
\begin{align*}
  \text{collection of all Serre\,} & \xrightarrow{\;\simeq\;} \text{collection of all Serre} \\
  \text{subcategories of $\cC^{noeth}$} &                    \text{\qquad subcategories of $\cA^{noeth}$}  \\
                          \cS & \longmapsto \cA \cap \cS \ .
\end{align*}
But it follows from Def.\ \ref{def:thickening}.c that any such $\cS$ is the smallest Serre subcategory of $\cC^{noeth}$ which contains $\cA \cap \cS$. On the other hand let $\cT$ be a Serre subcategory of $\cA^{noeth}$. We define $\cS$ to be the full subcategory of $\cC^{noeth}$ whose objects $V$ have a finite filtration $0 = V_0 \subseteq V_1 \subseteq \ldots \subseteq V_m = V$ with $V_j/V_{j-1}$ in $\cT$ for any $1 \leq j \leq m$. It is straightforward to check that $\cS$ is a Serre subcategory of $\cC^{noeth}$ and that $\cA \cap \cS = \cT$.
\end{proof}

The above Prop.\ \ref{prop:loc-A-C}, of course, holds true correspondingly with $\cA$ replaced with $\cA_j$.

For any localising subcategory $\cL$ of a locally noetherian Grothendieck category $\cD$ one defines the subset
\begin{equation*}
  A(\cL) := \{ [E] \in \Sp(\cD) : \Hom_\cD(V,E) = 0\ \text{for any $V \in \ob(\cL)$}\}
\end{equation*}
of $\Sp(\cD)$. These subsets $A(\cL)$ form the closed subsets of a topology on $\Sp(\cD)$ which is called the \emph{Ziegler topology} (cf.\ \cite{Her} Thm.\ 3.4). In fact, by Prop.\ \ref{prop:loc-Serre} and \cite{Her} Thm.\ 3.8, the map
\begin{align}\label{f:loc-classific}
  \text{collection of all localising} & \xrightarrow{\;\simeq\;} \text{set of all Ziegler-closed} \\
  \text{subcategories of $\cD$} &                    \text{\qquad\ subsets of $\Sp(\cD)$}    \nonumber \\
                          \cL & \longmapsto A(\cL)       \nonumber
\end{align}
is an inclusion reversing bijection. This means that the Ziegler closed subsets of $\Sp(\cD)$ classify the Serre subcategories of $\cD^{noeth}$ as well as the localising subcategories of $\cD$. It also implies that $\cL$ can be reconstructed from $A(\cL)$ by
\begin{equation*}
  \ob(\cL) = \{ V \in \ob(\cD) : \Hom_\cD(V,E) = 0 \ \text{for all $[E] \in A(\cL)$}\}.
\end{equation*}

\begin{corollary}\label{cor:Ziegler-homeo}
  The maps $\Sp(\cC) \xrightarrow{\;\simeq\;} \Sp(\cA_j)  \xrightarrow{\;\simeq\;} \Sp(\cA)$ in \eqref{f:bij} are homeomorphisms for the Ziegler topologies.
\end{corollary}
\begin{proof}
Using \cite{Her} Prop.\ 3.2 and Cor.\ 3.5 we see that Def.\ \ref{def:thickening}.c implies that the sets
\begin{align*}
  O_\cC(U) & := \{[E] \in \Sp(\cC) : \Hom_\cC(U,E) \neq 0\} ,  \\
  \text{resp.}\ O_\cA(U) & := \{[r(E)] \in \Sp(\cA) : \Hom_\cA(U,r(E)) \neq 0\},
\end{align*}
for $U \in \ob(\cA)$, form a base for the Ziegler-open subsets of $\Sp(\cC)$, resp.\ $\Sp(\cA)$. But, for such $U$, we have $\Hom_\cC(U,E) = \Hom_\cA(U,r(E))$ since $r$ is right adjoint to the inclusion functor.
\end{proof}

Later on another topology on $\Sp(\cD)$ will be more important for our purposes.

\begin{proposition}\label{prop:qc}
   For a localising subcategory $\cL$ of $\cD$, the following are equivalent:
\begin{itemize}
  \item[a)] there is a noetherian object $C$ in $\cD$ such that $\cL$ is the smallest localising subcategory containing $C$;
  \item[b)] there is a noetherian object $C$ in $\cD$ such that
  \[A(\cL) = \{ [E] \in \Sp(\cD) : \Hom_\cD(C,E) = 0\};\]
  \item[c)] the Ziegler open subset $\Sp(\cD) \setminus A(\cL)$ is quasi-compact.
\end{itemize}
\end{proposition}
\begin{proof}
For the equivalence of a. and b. we consider more generally an arbitrary object of $\cD$, and we let $\langle C \rangle$ denote the smallest localising subcategory containing $C$. Let $E$ be any injective object in $\cD$. If $\Hom_\cD(-,E)$ vanishes on $\langle C \rangle$ then obviously $\Hom_\cD(C,E) = 0$. But the injectivity of $E$ easily implies that the converse holds as well. It follows that
\begin{equation*}
  A(\langle C \rangle) = \{ [E] \in \Sp(\cD) : \Hom_\cD(C,E) = 0\} .
\end{equation*}

For the equivalence of b. and c. see \cite{Her} Cor.\ 3.9.
\end{proof}

The topology on $\Sp(\cD)$ which has as a base of open subsets the complements of quasi-compact Ziegler-open subsets is called the \emph{Gabriel-Zariski topology}.

\begin{corollary}\label{cor:Zariski-homeo}
  The maps $\Sp(\cC) \xrightarrow{\;\simeq\;} \Sp(\cA_j)  \xrightarrow{\;\simeq\;} \Sp(\cA)$ in \eqref{f:bij} are homeomorphisms for the Gabriel-Zariski topologies.
\end{corollary}

\subsection{Stability}\label{subsec:stability}

We recall that a localising subcategory of a locally noetherian Grothendieck category $\cD$ is called \emph{stable} if it is closed under the passage to essential extensions.

\begin{lemma}\label{lem:stable-1}
   For any localising subcategory $\cL$ of $\cD$ the following are equivalent:
\begin{itemize}
  \item[a)] $\cL$ is stable;
  \item[b)] any indecomposable injective object of $\cD$ either lies in $\cL$ or has no non-zero subobject lying in $\cL$.
\end{itemize}
\end{lemma}
\begin{proof}
We argue similarly as in \cite{Gol} Prop.\ 11.3.

$a) \Longrightarrow b)$: Let $E$ be an indecomposable injective object in $\cD$ and let $t_\cL(E)$ denote the largest subobject of $E$ contained in $\cL$. Suppose that $t_\cL(E) \neq 0$. Then $E$ is an injective hull of $t_\cL(E)$ and hence, by stability, is contained in $\cL$.

$b) \Longrightarrow a)$: We write an injective hull $E(V)$ of an object $V$ lying in $\cL$ as a direct sum $E(V) = \oplus_{i \in I} E_i$ of indecomposable injective objects $E_i$. Since $E_i \cap V \neq 0$ lies in $\cL$ for any $i \in I$ we see that all $E_i$ and hence $E(V)$ lie in $\cL$.
\end{proof}

\begin{corollary}\label{cor:stable}
   Let $\cL$ be a stable localising subcategory of $\cD$; then
\begin{equation*}
  A(\cL) = \{[E] \in \Sp(\cD) : E \not\in \ob(\cL)\}  \ .
\end{equation*}
\end{corollary}

The following is a straightforward generalization of \cite{Lou} Prop.\ 4.

\begin{lemma}\label{lem:stable-2}
   For a subset $A \subseteq \Sp(\cD)$ the following are equivalent:
\begin{itemize}
  \item[a)] $A = A(\cL)$ for a stable localising subcategory $\cL$ of $\cD$;
  \item[b)] if $[E] \in \Sp(\cD)$ satisfies $\Hom_\cD(E,E') \neq 0$ for some $[E'] \in A$, then $[E] \in A$.
\end{itemize}
\end{lemma}
\begin{proof}
$a) \Longrightarrow b)$: Let $[E'] \in A(\cL)$ such that $\Hom_\cD(E,E') \neq 0$. Then $E$ does not lie in $\cL$. Since $\cL$ is stable Lemma \ref{lem:stable-1} applies and tells us that $E$ does not have any non-zero subobject lying in $\cL$. Hence $[E] \in A(\cL) = A$.

$b) \Longrightarrow a)$: Let $\cL$ be the localising subcategory of $\cD$ cogenerated by the $E'$ for $[E'] \in A$. This means that $\cL$ is the full subcategory of those objects $V$ in $\cD$ which satisfy $\Hom_\cD(V,E') = 0$ for any $[E'] \in A$. It is immediate that $A \subseteq A(\cL)$. Consider any $[E] \in A(\cL)$. Then $E$ cannot lie in $\cL$. Hence there must exist an $[E'] \in A$ such that $\Hom_\cD(E,E') \neq 0$. It follows from (b) that $[E] \in A$. This shows that $A = A(\cL)$. To establish that $\cL$ is stable we use Lemma \ref{lem:stable-1}. We have just seen that the $E$ which do not lie in $\cL$ must have $[E] \in A(\cL)$. By the very definition of $A(\cL)$ such $E$ do not have a non-zero subobject lying in $\cL$.
\end{proof}

A subset $A \subseteq \Sp(\cD)$ will be called \emph{stable}, resp.\ \emph{stable-open}, if it is of the form $A = A(\cL)$ for some stable localising subcategory $\cL$ of $\cD$, resp.\ if it is stable and open for the Gabriel-Zariski topology. It is clear, for example from Lemma \ref{lem:stable-2}, that arbitrary intersections and unions of stable subsets are stable again. Therefore the stable, resp.\ stable-open, subsets are the open subsets for a topology which we call the stable, resp.\ stable Zariski, topology of $\Sp(\cD)$.

\begin{corollary}\label{cor:stable-cont}
  The bijections $\Sp(\cC) \xrightarrow{\;\simeq\;} \Sp(\cA_j)  \xrightarrow{\;\simeq\;} \Sp(\cA)$ in \eqref{f:bij} respect stable subsets; in particular, the inverse maps are continuous for the stable topologies.
\end{corollary}
\begin{proof}
It suffices to consider the map $\Sp(\cC) \rightarrow \Sp(\cA)$ sending $[E]$ to $[r(E)]$. Let $S \subseteq \Sp(\cC)$ be any stable subset and $[r(E')]$ be any point in the image of $S$. Now suppose that we have $[U] \in \Sp(\cA)$ such that $\Hom_\cA(U,r(E')) \neq 0$. Then $\Hom_\cC(U,E') \neq 0$ and hence $\Hom_\cC(E_\cC(U),E') \neq 0$. Using Lemma \ref{lem:stable-2} we deduce that $[E_\cC(U)] \in S$. But $U \cong r(E_\cC(U))$. It follows that $[U]$ lies in the image of $S$. Using again Lemma \ref{lem:stable-2} we conclude that the image of $S$ is stable.
\end{proof}

\begin{lemma}\label{lem:Rr}
   Suppose that $\cL$ is a stable localising subcategory of $\cC$; then the right derived functors $R^j r$ of $r$, for any $j \geq 0$, map $\cL$ to $\cA \cap \cL$.
\end{lemma}
\begin{proof}
Let $V$ be any object in $\cL$. Since $\cL$ is stable an injective hull $E(V)$ of $V$ in $\cC$ lies already in $\cL$. Hence we find an injective resolution $V \xrightarrow{\simeq} I^\bullet$ of $V$ in $\cC$ all of whose terms lie in $\cL$. The values $R^j r(V)$ of the right derived functors in question are the cohomology objects of the complex $r(I^\bullet)$. On the other hand, Def.\ \ref{def:thickening}.b easily implies that $r$ maps $\cL$ to $\cA \cap \cL$. It follows that the complex $r(I^\bullet)$ and then also its cohomology objects lie in $\cA \cap \cL$.
\end{proof}

For the rest of this section we impose on our thickening $\cC$ of $\cA$ the additional condition that \textbf{$\cA$ has a noetherian projective generator $P$}. Note that $P$ then also is noetherian as an object in $\cC$. We let $H := \End_\cA(P)^{\op} = \End_\cC(P)^{\op}$ denote the opposite ring of the endomorphism ring of the generator $P$. In this situation one has the equivalence of categories
\begin{align}\label{f:cat-equiv}
  \cA & \xrightarrow{\;\simeq\;} \Mod(H) \\
    A & \longmapsto \Hom_\cA(P,A) = \Hom_\cC(P,A)   \nonumber
\end{align}
(cf.\ \cite{Pop} Cor.\ 5.9.5). In fact, the $\Ext$-functors $\Ext_\cC^j(P,-)$ on $\cC$, for $j \geq 0$, can naturally be viewed as functors
\begin{equation*}
  \Ext_\cC^j(P,-) : \cC \longrightarrow \Mod(H) \ .
\end{equation*}
In particular, the natural isomorphism
\begin{equation}\label{f:ss-proj}
  \Ext_\cC^j(P,V) = \Hom_\cA(P, R^j r(V)) \qquad\text{for any $V$ in $\cC$ and any $j \geq 0$}
\end{equation}
from Cor.\ \ref{cor:ss-proj} is an isomorphism of $H$-modules.

\begin{lemma}\label{lesldt}
Let $A$ and $V$ be objects in $\cA$.
\begin{itemize}
\item[a)] There is a natural exact sequence of abelian groups
\begin{equation*}
  0 \rightarrow \Ext^1_{\cA}(A, V) \rightarrow \Ext^1_{\cC}(A,V) \xrightarrow{\varphi_{A,V}} \Hom_H(\Hom_\cA(P,A),\Ext_\cC^1(P,V)) \rightarrow \Ext^2_{\cA}(A, V) \ .
\end{equation*}
\item[b)] If $V$ is an injective object in $\cA$, then $\varphi_{A,V}$ is a natural isomorphism.
\end{itemize}
\end{lemma}
\begin{proof}
a) By Prop.\ \ref{prop:ss} we have the convergent Grothendieck spectral sequence
\begin{equation*}
  \Ext^i_{\cA}(A, R^jr(V)) \Rightarrow \Ext^{i+j}_{\cC}(A,V) \ .
\end{equation*}
The exact sequence of low degree terms is
\begin{equation*}
  0 \rightarrow \Ext^1_{\cA}(A, r(V)) \rightarrow \Ext^1_{\cC}(A,V) \rightarrow \Hom_{\cA}(A, R^1r(V)) \rightarrow \Ext^2_{\cA}(A, r(V)) \ .
\end{equation*}
Since $V$ lies in $\cA$ we have $r(V) = V$. Moreover, using \eqref{f:ss-proj} for the second equality, we obtain
\begin{align*}
  \Hom_{\cA}(A, R^1r(V)) & = \Hom_H(\Hom_\cA(P,A),\Hom_\cA(P, R^1 r(V)))    \\
                         & = \Hom_H(\Hom_\cA(P,A),\Ext_\cC^1(P,V)) \ .
\end{align*}

b) Since $V$ is injective, we have $\Ext^1_\cA(A,V) = \Ext^2_\cA(A,V) = 0$.
\end{proof}

Now let $\cL$ be a localising subcategory of $\cC$. Then $\cL \cap \cA$ is a localising subcategory of $\cA$. We want to investigate how the stability requirements for $\cL$ and $\cA \cap \cL$ are related to each other. It is obvious that if $\cL$ is stable in $\cC$ then $\cA \cap \cL$ is stable in $\cA$. The converse is more subtle. We denote the corresponding localising subcategory of $\Mod(H)$ by
\begin{equation*}
  \cL^H := \ \text{essential image of $\cA \cap \cL$ in $\Mod(H)$ under the functor $\Hom_\cA(P,-)$}.
\end{equation*}

\begin{proposition}\label{TwoStep}
Suppose that $\cA \cap \cL$ is stable in $\cA$ and that $\Ext^1_{\cC}(P,-)$ maps $\cA \cap \cL$ to $\cL^H$. Let $V$ in $\cC$ be a nonzero object such that
\begin{itemize}
\item[a)] $V = r_1(V)$, and
\item[b)] $r(V)$ is the largest subobject of $V$ contained in $\cL$.
\end{itemize}
Then $V = r(V)$.
\end{proposition}
\begin{proof}
We first note that we have $r(V) \neq 0$ by Def.\ \ref{def:thickening}.b.

Let $E_{\cC}(V)$ be an injective hull of $V$ in $\cC$ and let $E = r_1(E_{\cC}(V))$. Then $r(E) = r(E_{\cC}(V))$ is an injective object in $\cA$, and $r(V) \subseteq r(E)$. In fact, we have the inclusions
\begin{equation*}
  \xymatrix{
    r(E)  \ar[rr]^-{=} &  & r(E_\cC(V)) \ar[d]^{\subseteq} \\
    r(V) \ar[u]_{\subseteq} \ar[r]^-{\subseteq} & V  \ar[r]^-{\subseteq} & E_\cC(V) .    }
\end{equation*}
By Lemma \ref{lem:ess-hull-indec}.a the extension $r(V) \subseteq V$ is essential. Hence the extension $r(V) \subseteq r(E)$ is essential as well.
Since $\cA \cap \cL$ is stable in $\cA$ and since $r(V)$ lies $\cA \cap \cL$ by assumption b), we see that $r(E)$ lies in $\cA \cap \cL$.

Now, because $r_1$ is left exact, $V = r_1(V) \subseteq r_1(E_{\cC}(V)) = E$. Next, $(r(E) + V)/r(E) \cong V / (V \cap r(E)) = V / r(V)$ because $r$ is left exact.  Since $r(V)$ is the largest subobject of $V$ contained in $\cL$, we see that $V/r(V)$ is $\cL$-torsion free. Also, $V = r_1(V)$ means that $V / r(V)$ lies in $\cA$. We deduce that $(r(E) + V)/r(E) \cong V / r(V)$ lies in $\cA$ and is $\cL$-torsion free. Hence we need to show that $V \subseteq r(E)$.

Consider any intermediate object $r(E) \subseteq V'' \subseteq V' := r(E) + V \subseteq E$ in $\cC$. We obtain extension classes $e' = [0 \rightarrow r(E) \rightarrow V' \rightarrow V'/r(E) \rightarrow 0] \in \Ext_\cC^1(V'/r(E),r(E))$ and analogously $e'' \in \Ext_\cC^1(V''/r(E),r(E))$. Note that $V'/r(E)$, $V''/r(E)$, and $r(E)$ are objects in $\cA$ with $V'/r(E)$ and $V''/r(E)$ being $\cA \cap \cL$-torsion free and $r(E)$ being injective in $\cA$. Hence Lemma \ref{lesldt} applies. We define the $H$-modules $N'' := \Hom_\cA(P,V''/r(E)) \subseteq N' := \Hom_\cA(P,V'/r(E))$ and obtain the commutative diagram
\begin{equation*}
  \xymatrix{
  \Ext^1_{\cC}(V'/r(E),r(E))\ar[d]_{\subseteq^\ast} \ar[rr]^-{\varphi'}_-{\cong} && \Hom_H(N', \Ext^1_{\cC}(P, r(E)))\ar[d]^{\res}  \\
\Ext^1_{\cC}(V''/r(E),r(E))\ar[rr]_-{\varphi''}^-{\cong} && \Hom_H(N'', \Ext^1_{\cC}(P, r(E))).
}
\end{equation*}
with horizontal isomorphisms. Obviously the left perpendicular arrow maps $e'$ to $e''$. Since $r(E)$ lies in $\cA \cap \cL$ the $H$-module $\Ext^1_{\cC}(P, r(E))$ lies in $\cL^H$ by our assumptions. We have the $H$-linear map $f' := \varphi'(e') : N' \rightarrow \Ext^1_{\cC}(P, r(E))$. Suppose now that $N' \neq 0$ and choose a nonzero element $v \in N'$. Let $J := \Ann_H(f'(v)) \subseteq H$ be the annihilator left ideal. Then $H/J \cong Hf'(v)$ lie in $\cL^H$. We now consider specifically the $H$-submodule $N'' := Jv \subseteq N'$. By the category equivalence there is a unique $V''$ as above such that $N'' = \Hom_\cA(P,V''/r(E))$. We first check that $N'' \neq 0$. Otherwise we would have $J \subseteq \Ann_H(v)$, so that $H/J$ surjects onto $Hv \subseteq N'$, which implies that $Hv$ lies in $\cL^H$. On the other hand the modules $Hv \subseteq N'$ are $\cL^H$-torsion free, which is a contradiction. Hence $N'' \neq 0$. Now
\begin{equation*}
  \res(\varphi'(e'))(N'') = \res(f')(N'') = f'(N'') = f'(Jv) = J f'(v) = 0 \ .
\end{equation*}
The commutativity of the above diagram then implies that the extension class $e'' = 0$, i.e., the short exact sequence $0 \rightarrow r(E) \rightarrow V'' \rightarrow V''/r(E) \rightarrow 0$ in $\cC$ splits. But then because $V''/r(E) \neq 0$ we see $r(E)$ is not essential in $E$, which contradicts Lemma \ref{lem:ess-hull-indec}.a. It follows that $N'$ must be zero, i.e., that $r(E) + V = r(E)$.
\end{proof}

\begin{theorem}\label{thm:stable-A-L}
Suppose that $\cA \cap \cL$ is stable in $\cA$. Then the following are equivalent:
\begin{itemize}
\item[a)] $\cL$ is stable in $\cC$,
\item[b)] $R^jr$ maps $\cL$ to $\cA \cap \cL$ for all $j \geq 0$,
\item[c)] $\Ext^j_{\cC}(P,-)$ maps $\cL$ to $\cL^H$ for all $j \geq 0$,
\item[d)] $\Ext^1_{\cC}(P,-)$ maps $\cL$ to $\cL^H$,
\item[e)] $\Ext^1_{\cC}(P,-)$ maps $\cA \cap \cL$ to $\cL^H$.
\end{itemize}
\end{theorem}
\begin{proof}
a) $\Rightarrow$ b). This is Lemma \ref{lem:Rr}.

b) $\Rightarrow$ c). This follows from \eqref{f:ss-proj} and the definition of $\cL^H$.

c) $\Rightarrow$ d) $\Rightarrow$ e). These are trivial.

e) $\Rightarrow$ a). Let $E$ be an indecomposable injective object in $\cC$ and let $Y$ be its largest subobject in $\cL$. We suppose that $Y \neq 0$ and will aim to show that $Y = E$. Let $E_n := r_n(E)$ for $n \geq 0$ and note that $E = \bigcup\limits_{n=0}^\infty E_n$ by Lemma \ref{lem:union}. Correspondingly $Y = \bigcup\limits_{n=0}^\infty Y_n$ with $Y_n := r_n(Y)$, and note that $Y_n = Y \cap V_n$ for all $n \geq 0$ because each $r_n$ is left exact.

Since $Y$ is the largest subobject of $E$ contained in $\cL$, $E/Y$ is $\cL$-torsion free. Hence for any $n \geq 0$, $(E_n + Y)/Y$ is also $\cL$-torsion free. But $(E_n + Y)/Y \cong E_n / (E_n \cap Y) = E_n / Y_n$. Hence for all $n \geq 0$, $E_n / Y_n$ is $\cL$-torsion free, and $Y_n$ is the largest subobject of $E_n$ contained in $\cL$.

With $Y \neq 0$ also $Y_0 = r(Y)$ is nonzero. Since $E$ is an indecomposable injective in $\cC$, it follows that $Y_0$ is essential in $E$.

Suppose for a contradiction that $Y \neq E$ and let $n$ be minimal such that $Y_n \neq E_n$. Now, $Y_0 \subseteq E_0$ is an essential extension in $\cA$ with $Y_0$ lying in $\cA \cap \cL$; since $\cA \cap \cL$ is stable in $\cA$ we see that $Y_0 = E_0$. Hence $n \geq 1$, and the minimality of $n$ implies that $Y_{n-1} = E_{n-1}$ so that $E_{n-1} \subseteq Y_n$.

Now, consider the short exact sequence
\begin{equation*}
   0 \rightarrow Y_n / E_{n-1} \rightarrow E_n / E_{n-1} \rightarrow E_n / Y_n \rightarrow 0
\end{equation*}
in $\cA$. Note that $Y_n / E_{n-1}$ lies in $\cL$, and that $E_n / Y_n$ is $\cL$-torsion free and nonzero. If $Y_n / E_{n-1}$ is also nonzero, then because $\cA \cap \cL$ is stable in $\cA$, this extension is not essential. This gives us some subobject $W$ of $E_n$ containing $E_{n-1}$, such that $W / E_{n-1}$ is isomorphic to a nonzero subobject of $E_n/Y_n$; thus $W / E_{n-1}$ is also nonzero and $\cL$-torsion free. If $Y_n = E_{n-1}$, then we can take $W = E_n$ to obtain an object with the same properties.

Write $E_{-1} := 0$ and $X := W/E_{n-2} \subseteq E_n / E_{n-2}$, and consider the short exact sequence
\begin{equation*}
  0 \rightarrow E_{n-1} / E_{n-2} \rightarrow X \rightarrow W / E_{n-1} \rightarrow 0 \ .
\end{equation*}
Since the outer terms in this short exact sequence lie in $\cA$, we see that $r_1(X) = X$. Also, $r(X) \subseteq r(E_n / E_{n-2}) = E_{n-1}/E_{n-2} = r( E_{n-1}/E_{n-2} ) \subseteq r(X)$, where the left equality uses Lemma \ref{lem:preradical}.b, shows that
\begin{equation*}
  r(X) = E_{n-1} / E_{n-2} \subsetneqq X \ .
\end{equation*}
Since $W / E_{n-1}$ is $\cL$-torsion free and $E_{n-1} / E_{n-2} = Y_{n-1} / E_{n-2}$ is a subquotient of $Y$, we see that $r(X)$ is the largest subobject of $X$ contained in $\cL$. Applying now Prop.\ \ref{TwoStep} with $X$ leads to a contradiction.
\end{proof}

\subsection{A ring-theoretic case}\label{sec:rings}

In this section, we discuss stability for the category $\cC = \Mod(H)$ of left $H$-modules over a ring $H$ which is assumed to be left noetherian and finitely generated as a module over its centre $Z(H)$. This material will be needed in $\S \ref{sec:Z-scheme}$.

The ring $Z(H)$ is necessarily noetherian by \cite{MCR} Cor.\ 10.1.11(ii). The ring $H$ is also right noetherian, and the abelian categories $\Mod(H)$ and $\Mod(Z(H))$ are locally noetherian Grothendieck categories. Of course, $\Spec(H)$ and $\Spec(Z(H))$ denote the prime ideal spectra of the respective rings equipped with its Zariski topology.

Let $\varrho : \Mod(H) \rightarrow \Mod(Z(H))$ denote the functor of passing to the underlying $Z(H)$-module. By \cite{Gab} Prop.\ 10 on p.\ 428 any localising subcategory of $\Mod(Z(H))$ is stable. We further have the following facts.

\begin{lemma}\phantomsection\label{lem:varrho-stable}\,
\begin{itemize}
  \item[a)]  For any localising subcategory $\cL$ of $\Mod(Z(H))$ the localising subcategory $\varrho^{-1}(\cL)$ of $\Mod(H)$ is stable.
  \item[b)] For any ordinal $\alpha$ we have $\varrho^{-1}(\Mod(Z(H))_\alpha) = \Mod(H)_\alpha$, which therefore is stable.
  \item[c)] For any ordinal $\alpha < \kappa(\Mod(H))$ we have
\begin{equation*}
  \Sp_\alpha(\Mod(H)) = \{[E] \in \Sp(\Mod(H)) : \kappa(E) = \alpha + 1\} .
\end{equation*}
\end{itemize}
\end{lemma}
\begin{proof}
a) \cite{Gab} Prop.\ 12 on p.\ 431. b) The equality is \cite{MCR} Cor.\ 10.1.10, and the stability then follows from a). c) Because of b) this follows from Remark \ref{rem:stable-Krull}.
\end{proof}

Next we have:
\begin{itemize}
  \item[(A)] The continuous map $\varphi : \Spec(H) \rightarrow \Spec(Z(H))$ sending $\frp$ to $\frp \cap Z(H)$ is surjective with finite fibers (cf.\ \cite{Gab} Prop.\ 11 on p.\ 429).
  \item[(B)] The map
\begin{align}\label{f:ass}
  \Sp(\Mod(H)) & \xrightarrow{\;\sim\;} \Spec(H) \\
                   [J] & \longmapsto \text{the unique prime ideal associated to $J$}.     \nonumber
\end{align}
   is a homeomorphism w.r.t.\ the Gabriel-Zariski and the Zariski topology on the left hand and right hand side, respectively (\cite{Ste} Thm.\ VII.2.1 recalling the fact that $H$ is fully left and right bounded by \cite{Ste} VII \S2 Example 4), \cite{Gab} V \S4). Let $\frp \in \Spec(H)$. Fix an injective hull $E(H/\frp)$ of $H/\frp$ in $\Mod(H)$. From \cite{Ste} Prop.\ VII.1.9 we know that $E(H/\frp) \cong \oplus_{i \in I} E_{\frp}$ for a single indecomposable injective $H$-module $E_{\frp}$. Using \cite{Ste} Lemma VII.1.7 we see that the inverse of the bijection \eqref{f:ass} is given by $\frp \longmapsto [E_{\frp}]$.
\end{itemize}

First we describe the stable Zariski topology, where we will always make silently the identification \eqref{f:ass}.  Hence on $\Spec(Z(H))$ the stable Zariski topology is the Zariski topology. From \cite{Gab} V \S 4 we know that the Ziegler-open subsets of $\Spec(H)$ are the possibly infinite unions of Zariski-closed subsets. Hence the quasi-compact Ziegler-open subsets are the Zariski-closed subsets. Therefore the open subsets for the stable Zariski topology are the stable Zariski-open subsets.

\begin{lemma}\label{lem:varrho-stable2}
  $A(\varrho^{-1}(\cL)) = \varphi^{-1}(A(\cL))$ for any localising subcategory $\cL$ of $\Mod(Z(H))$.
\end{lemma}
\begin{proof}
(Recall the notation $A(\cL)$ from $\S$\ref{subsec:localising}.) Using Cor.\ \ref{cor:stable} we compute
\begin{align*}
  A(\varrho^{-1}(\cL)) & = \{\frp \not\in \Spec(H) : E_\frp \in \ob(\varrho^{-1}(\cL)) \} \\
   & = \{\frp \in \Spec(H) : H/\frp \not\in \ob(\varrho^{-1}(\cL)) \}    \\
   & = \{\frp \in \Spec(H) : \varrho(H/\frp) \not\in \ob(\cL) \}    \\
   & \supseteq \{\frp \in \Spec(H) : Z(H)/\varphi(\frp) \not\in \ob(\cL) \}  \\
   & = \varphi^{-1}(A(\cL)) \ .
\end{align*}
For the reverse inclusion we suppose that $Z(H)/\varphi(\frp)$ lies in $\cL$. Being a finitely generated $Z(H)/\varphi(\frp)$-module $\varrho(H / \varphi(\frp) H)$ also must lie in $\ob(\cL)$ and hence $H / \varphi(\frp) H \in \ob(\varrho^{-1}(\cL))$. It follows that $H/\frp$, as a quotient of $H / \varphi(\frp) H$, is contained in $\ob(\varrho^{-1}(\cL))$.
\end{proof}

\begin{lemma}\label{lem:E_P}
    For $\frp, \mathfrak{q} \in \Spec(H)$ we have:
\begin{itemize}
  \item[a)] Let $z \in Z(H)$; if $z \in \frp$, then $z$ acts locally nilpotently on $E_\frp$; if $z \not\in \frp$ then $z$ acts invertibly on $E_\frp$.
  \item[b)] If $\Hom_H(E_\frp,E_\mathfrak{q}) \neq 0$ then $\frp \cap Z(H) \subseteq \mathfrak{q} \cap Z(H)$.
\end{itemize}
\end{lemma}
\begin{proof}
a) Let $L$ be a uniform left ideal in $H/\frp$ so that $E_\frp$ is the injective hull $E(L)$ of $L$.

Suppose first that $z \notin \frp$. Multiplication by $z$ is an injective $H$-linear map $\ell_z : H/\frp \to H/\frp$: if $x + \frp \in H/\frp$ is such that $z(x+\frp) = \frp$ then $(HzH)(HxH) \subseteq \frp$ since $z$ is central in $H$, so $x \in \frp$ since $\frp$ is prime and $z \notin \frp$. Since $L$ is an $H$-submodule of $H/\frp$, $\ell_z : L \to L$ is also injective. Hence it extends to an injective $H$-linear map $\ell_z : E(L) \to E(L)$ by the injectivity of $E(L)$. The image of $\ell_z$ must admit a complement in $E(L)$ since it is itself injective. Since $E(L) \cong E_\frp$ is indecomposable, $\ell_z : E(L) \to E(L)$ is an isomorphism, as claimed.

Now suppose that $z \in \frp$ and let $x \in E_\frp$ be non-zero. Since $L$ is essential in $E_\frp$, we see that $Hx \cap L$ is essential in $Hx$. Since $H$ is noetherian and $z$ is central, the ideal $zH$ of $H$ satisfies the left Artin-Rees property by \cite{MCR} Prop.\ 4.2.6. Since $z \in \frp$ kills $L$, it also kills $Hx \cap L$. So $z^n$ kills $Hx$ for some $n \geq 1$ by the implication (i) $\Rightarrow$ (iii) of \cite{MCR} Thm. 4.2.2.

b) Let $\phi : E_\frp \rightarrow E_\mathfrak{q}$ be a nonzero map. Suppose that there is a $z \in (\frp \cap Z(H)) \setminus \mathfrak{q}$. Then, by a), $z$ acts invertibly on $E_\mathfrak{q}$, but locally nilpotently on $E_\frp$. Consider any $x \in E_\frp$ such that $\phi(x) \neq 0$. We find a $t \geq 1$ such that $z^t x = 0$. Hence $z^t \phi(x) = \phi(z^t x) = \phi(0) = 0$, which is a contradiction.
\end{proof}

For any subset $A \subseteq \Spec(H)$ the full subcategory $\cL_A$ of all modules $U$ in $\Mod(H)$ such that $\Hom_H(U,E_\frp) = 0$ for any $\frp \in A$ is a localising subcategory of $\Mod(H)$. Obviously we have $A \subseteq A(\cL_A)$.

We now describe explicitly all localising subcategories of $\Mod(H)$ which are of the form $\rho^{-1}(\cL)$ as in Lemma \ref{lem:varrho-stable}.a.

Recall that a subset $S$ of a topological space $T$ is called \emph{generalization-closed} if any point $y \in T$ whose closure $\overline{\{y\}}$ contains a point in $S$ also lies in $S$. Equivalently, $S$ is generalization-closed if and only if it is a (possibly infinite) intersection of open subsets of $T$.

Before Lemma \ref{lem:varrho-stable2} we had noted that the Ziegler-open subsets of $\Spec(H)$ are the unions of Zariski-closed subsets. Hence the Ziegler-closed subsets are the generalization-closed subsets. Therefore the inverse of the bijection \eqref{f:loc-classific} is the map
\begin{align*}
  \text{set of all generalization-closed} & \xrightarrow{\;\simeq\;}\ \ \text{collection of all localising} \\
  \text{subsets of $\Spec(Z(H))$} \qquad & \qquad\ \text{subcategories of $\Mod(Z(H))$}   \\
                                 W & \longmapsto \cL_W^Z := \ \text{all $M$ with $\Hom_{Z(H)}(M,E) = 0$}   \\
                                   &  \qquad\qquad\qquad\qquad \text{ for any $[E] \in W$}.
\end{align*}
We conclude from Lemma \ref{lem:varrho-stable2} that
\begin{equation}\label{f:varrho-stable2}
  \varrho^{-1}(\cL_W^Z) = \cL_{\varphi^{-1}(W)} \ .
\end{equation}

For an ideal $I$ of $Z(H)$, recall that $V(I)$ denotes the subset of $\Spec(Z(H))$ consisting of all prime ideals containing $I$.

\begin{proposition}\label{prop:LHW}
    Let $W$ be a generalization-closed subset of $\Spec(Z(H))$. The following are equivalent for an $H$-module $M$:
\begin{itemize}
\item[a)] $M$ lies in $\cL_{\varphi^{-1}(W)}$;
\item[b)] $V(\Ann_{Z(H)}(v)) \cap W = \emptyset$ for all nonzero $v \in M$;
\item[c)] $V(\Ann_{Z(H)}(N)) \cap W = \emptyset$ for all noetherian submodules $N \subseteq M$.
\end{itemize}
\end{proposition}
\begin{proof}
Since $\cL_{\varphi^{-1}(W)}$ is localising we may assume that $M$ is finitely generated. If $v_1, \ldots, v_r$ are generators of the $H$-module $M$ then
\begin{equation*}
  V(\Ann_{Z(H)}(M)) = V\left(\bigcap_{i=1}^r \Ann_{Z(H)}(v_i)\right) = \bigcup_{i=1}^r V(\Ann_{Z(H)}(v_i)) \ .
\end{equation*}
This shows the equivalence of b) and c).

For the equivalence of a) and b) we first note that, because $H$ is noetherian, $M$ contains an essential submodule $N$ of the form $N = N_1 \oplus \cdots \oplus N_r$, where each $N_i$ is a uniform $H$-module (\cite{GW} Cor.\ 5.18 and Prop.\ 5.15). Hence $E(M) = E(N) \cong \bigoplus_{i=1}^n E(N_i)$. By \cite{GW} Lemma 5.26, for each $i = 1,\cdots, r$ there is a unique prime ideal $\frp_i$ of $H$ which is equal to the annihilator of some nonzero $H$-submodule $N_i'$ of $N_i$, and which contains the annihilators of all nonzero $H$-submodules of $N_i$. This $\frp_i$ is the \emph{assassinator} of $N_i$ in the sense of \cite{GW} Def.\ on p.\ 102.

Since $H$ is assumed to be a finitely generated module over its centre, it is an FBN-ring by \cite{GW} Prop.\ 9.1(a). Then we can apply \cite{GW} Prop.\ 9.14 to see that the uniform injective $H$-module $E(N_i)$ is isomorphic to the $H$-injective hull $E(L_i)$ of any uniform left ideal $L_i$ of $H/\frp_i$. In other words (\cite{GW} Lemma 5.14), $E(N_i) \cong E_{\frp_i}$, and
\begin{equation*}
   E(M) \cong E(N) \cong \bigoplus_{i=1}^r E(N_i)  \cong \bigoplus_{i=1}^r E_{\frp_i} \ .
\end{equation*}
We identify $M$ with its image inside $E(M)$. Note that $M$ is not contained in the direct sum $\bigoplus_{i \neq j} E_{\frp_i}$ for any $j$, because otherwise we would have $M \cap E_{\frp_j} = 0$, which contradicts the fact that $M$ is essential in $E(M)$. So, $\Hom_H(M, E_{\frp_j})$ is nonzero for any $j$.

a) $\Rightarrow$ b). Suppose that $M$ lies in $\cL_{\varphi^{-1}(W)}$. From $\Hom_H(M,E_{\frp_j}) \neq 0$ for any $j$ we deduce that $\frp_j \notin \varphi^{-1}(W)$ for all $j$. Now consider any nonzero $v \in M$ and any $\frp \in \varphi^{-1}(V(\Ann_{Z(H)}(v)))$. Because $\frp_i \cap Z(H)$ acts locally nilpotently on $E_{\frp_i}$ by Lemma \ref{lem:E_P}.a, we see that
\begin{equation*}
  ((\frp_1 \cap Z(H)) (\frp_2 \cap Z(H)) \cdots (\frp_r \cap Z(H))^k \cdot v = 0
\end{equation*}
for some $k > 0$. This implies that $\frp_j \cap Z(H) \subseteq \frp \cap Z(H)$ for some $j$. Since $W$ is generalization-closed and $\frp_j \cap Z(H) \notin W$, we see that $\frp \cap Z(H) \notin W$. Hence $V(\Ann_{Z(H)}(v)) \cap W = \emptyset$.

b) $\Rightarrow$ a). Suppose that $\Hom_H(M, E_\mathfrak{q}) \neq 0$ for some $\mathfrak{q} \in \varphi^{-1}(W)$. Then $\Hom_H(E_{\frp_i}, E_\mathfrak{q}) \neq 0$ for some $1 \leq i \leq r$ and so $\frp_i \cap Z(H) \subseteq \mathfrak{q} \cap Z(H)$ by Lemma \ref{lem:E_P}.b. Hence $\frp_i \cap Z(H) \in W$ since $W$ is generalization-closed.

Choose a non-zero element $v \in N_i'$; since $\frp_i$ kills $N_i' \supseteq Hv$ we have $\frp_i \subseteq \Ann_H(Hv)$; on the other hand, $\frp_i$ contains $\Ann_H(Hv)$ as we saw above and hence $\frp_i = \Ann_H(Hv)$. Since $Z(H)$ is central in $H$, we have $\Ann_{Z(H)}(v) = \Ann_{Z(H)}(Hv)$. This ideal of $Z(H)$ is equal to $\Ann_H(Hv) \cap Z(H) = \frp_i \cap Z(H)$. Thus $\Ann_{Z(H)}(v) = \frp_i \cap Z(H)$ lies in $W$ and therefore $V(\Ann_{Z(H)}(v)) \cap W \neq \emptyset$.
\end{proof}

Obviously we have a corresponding result for modules over $Z(H)$.

\begin{remark}\label{rem:LZW}
    Let $W$ be a generalization-closed subset of $\Spec(Z(H))$. The following are equivalent for a $Z(H)$-module $M$:
\begin{itemize}
\item[a)] $M$ lies in $\cL_W^Z$;
\item[b)] $V(\Ann_{Z(H)}(v)) \cap W = \emptyset$ for all nonzero $v \in M$;
\item[c)] $V(\Ann_{Z(H)}(N)) \cap W = \emptyset$ for all noetherian submodules $N \subseteq M$.
\end{itemize}
\end{remark}

\subsection{A stability criterion for thickenings}\label{sec:Z-scheme}

Throughout this section let $\cC$ be a locally noetherian Grothendieck category which is a thickening of its full abelian subcategory $\cA$. Recall from  $\S$\ref{sec:thickening} that we then have the increasing filtration $\cA = \cA_0 \hookrightarrow \cA_1 \hookrightarrow \ldots \hookrightarrow \cA_j \hookrightarrow \ldots$ of $\cC$ by full abelian subcategories, which all are locally noetherian Grothendieck categories. By \eqref{f:bij} the corresponding injective spectra all are in bijection
\begin{align*}
  \Sp(\cC) & \xrightarrow{\;\simeq\;} \Sp(\cA_j)  \xrightarrow{\;\simeq\;} \Sp(\cA) \\
             [E] & \longmapsto r_j([E]) \longmapsto r([E]) \ .        \nonumber
\end{align*}
By Cor.\ \ref{cor:Zariski-homeo} and Cor.\ \ref{cor:stable-cont} the inverses of these bijections are continuous for the stable as well as the stable Zariski topologies. Actually it is straightforward to deduce that, for $i \leq j$, the map
\begin{align*}
 \Sp(\cA_j) &  \xrightarrow{\;\simeq\;} \Sp(\cA_i) \\
             x = [E_x] & \longmapsto r_{j,i}(x) := r_i(x) = [r_i(E_x)]
\end{align*}
is a bijection with inverse $[U] \mapsto [E_{\cA_j}(U)]$, and this inverse is continuous for the stable as well as the stable Zariski topologies.

In this section we will present a technique to actually find stable subsets in $\Sp(\cC)$. For this we assume again from now on that \textbf{$\cA$ has a noetherian projective generator $P$}. Recall that $H = \End_\cA(P)^{\op} = \End_\cC(P)^{\op}$ and let $Z = Z(H)$ denote its centre. By the equivalence of categories \eqref{f:cat-equiv} this centre $Z$ acts naturally on any object in $\cA$, i.e., $Z = Z(\cA)$ is the centre of the category $\cA$.

But we need further assumptions. The first one is:
\begin{itemize}
  \item[\textbf{(A1)}] $H$ is finitely generated as a module over its centre $Z$.
\end{itemize}
The central tool for our investigation will be the (graded) $(H,H)$-bimodule
\begin{equation*}
  \Ext_\cC^*(P,P) := \bigoplus_{i=0}^\infty \Ext_\cC^i(P,P) \ .
\end{equation*}
This, in particular, is a module for $Z \otimes Z := Z \otimes_\mathbb{Z} Z$. We let $\cJ := \Ann_{Z \otimes Z}(\Ext_\cC^*(P,P))$ be its annihilator ideal  and $\cR := V(\cJ) \subseteq \Spec(Z \otimes Z)$ be the corresponding Zariski closed subset. Further let $\pi_i : \cR \rightarrow \Xi := \Spec(Z)$, for $i = 1, 2$, denote the restrictions of the two projection maps $\Spec(Z \otimes Z) \rightrightarrows \Xi$. We then have a coequaliser diagram
\begin{equation*}
  \xymatrix{
    \cR \ar@<1ex>[r]^-{\pi_1}  \ar@<-1ex>[r]_-{\pi_2} & \Xi \ar[r]^{\mathbf{q}} & \mathfrak{Z}  }
\end{equation*}
in the category of locally ringed spaces (cf.\ \cite{DG} Prop.\ I.1.1.6). We briefly recall that $\mathbf{q}$ is a topological quotient map which identifies the points $\pi_1(x)$ and $\pi_2(x)$ for all $x \in \cR$, and that the sheaf $\cO_\mathfrak{Z}$ is given, for any open subset $U \subseteq \Xi$, by the equalizer diagram
\begin{equation*}
  \xymatrix{
    \cO_\mathfrak{Z}(U) \ar[r]^-{\mathbf{q}^\hash = \subseteq} & \cO_\Xi(\mathbf{q}^{-1}(U)) \ar@<1ex>[r]^-{\pi_1^\hash}  \ar@<-1ex>[r]_-{\pi_2^\hash} & \cO_{\cR}((\mathbf{q} \pi_i)^{-1}(U)) \ .  }
\end{equation*}

In section \ref{sec:rings} we considered, for any generalization-closed subset $W \subseteq \Xi$, the localising subcategories
$\cL_W^Z$ of $\Mod(Z)$ and $\cL_W^H := \cL_{\varphi^{-1}(W)}$ of $\Mod(H)$. These satisfy $\varrho^{-1}(\cL_W^Z) = \cL_W^H$, by \eqref{f:varrho-stable2}. Under the category equivalence \eqref{f:cat-equiv} the subcategory $\cL_W^H$ corresponds to the localising subcategory
\begin{equation*}
  \cL_W^\cA := \ \text{all $A$ in $\cA$ such that $\Hom_\cA(P,A)$ lies in $\cL_W^H$}
\end{equation*}
of $\cA$. Finally, by Prop.\ \ref{prop:loc-A-C}, we have the localising subcategory $\cL_W^\cC := \langle \cL_W^\cA \rangle$ of $\cC$, which is the unique localising subcategory of $\cC$ such that $\cA \cap \cL_W^\cC = \cL_W^\cA$.

\begin{theorem}\label{thm:MainStability}
   In addition to \textbf{(A1)} we assume:
\begin{itemize}
  \item[\textbf{(A2)}]   There is an integer $d$ such that $\Ext_{\cC}^j(P,-)|_{\cA} = 0$ for any $j > d$.
  \item[\textbf{(A3)}]   $Z \otimes Z/\cJ^{d+1}$ is noetherian.
\end{itemize}
Then, for any generalization-closed subset $W \subseteq \Xi$ with the property that $\pi_2(\pi_1^{-1}(W)) \subseteq W$, the localising subcategory $\cL_W^{\cC}$ is stable.
\end{theorem}
We begin the proof by recalling the following useful fact.
\begin{lemma} \label{lem: ExtCommutes} Let $M$ be a noetherian object in $\cC$. Then $\Ext^j_{\cC}(M,-)$ commutes with filtered colimits for all $j \geq 0$.
\end{lemma}
\begin{proof} In the Grothendieck category $\cC$ filtered colimits are exact. Moreover, because $\cC$ is locally noetherian and $M$ is a noetherian object in $\cC$,  \cite{Gab} Cor.\ 1 on p.\ 358 tells us that:
\begin{itemize}
  \item[--] the full subcategory of injective objects in $\cC$ is closed under filtered colimits;
  \item[--] the functor $F := \Hom_\cC(M,-)$ from $\cC$ to the category of abelian groups commutes with filtered colimits.
\end{itemize}
In this situation, \cite{KS} Prop.\ 15.3.3 implies that the derived functors $R^jF = \Ext^j_\cC(M,-)$ of $F$ commute with filtered colimits.
\end{proof}
Next, we first need an auxiliary result. Since $Z = Z(\cA)$ the ring $Z \otimes Z$ acts, for any two objects $A_1$ and $A_2$ in $\cA$, naturally on the groups $\Ext^j_\cC(A_1,A_2)$ with the first and second factor in $Z \otimes Z$ acting through endomorphisms in the category $\cC$ on $A_1$ and $A_2$, respectively.

\begin{lemma}\label{lem:J-killing}
  Suppose that \textbf{(A2)} holds true, and let $A$ be an arbitrary object in $\cA$; then, for any $j \geq 0$, the $Z \otimes Z$-module $\Ext^j_\cC(P,A)$ is annihilated by $\cJ^{d+1-j}$.
\end{lemma}
\begin{proof}
By Lemma \ref{lem: ExtCommutes}, the functors $\Ext^j_\cC(P,-)$, for any $j \geq 0$, commute with arbitrary direct sums. It follows that $\cJ$ annihilates $\Ext_\cC^*(P,F)$ for any $F$ which is a possibly infinite direct sums of copies of $P$. Since $P$ is a generator\footnote{The projectivity of $P$ is not needed for this argument.}  of $\cA$, using \cite{Ste} p. 93 Prop 6.2 repeatedly we find an exact sequence in $\cA$ of the form
\begin{equation*}
  \ldots \longrightarrow F_n \longrightarrow \ldots \longrightarrow F_0 \longrightarrow A \longrightarrow 0
\end{equation*}
where each $F_n$ is a direct sum of copies of $P$. We break it up into short exact sequences
\begin{align*}
   & 0 \rightarrow Y_0 \rightarrow F_0 \rightarrow A \rightarrow 0 \\
   & 0 \rightarrow Y_1 \rightarrow F_1 \rightarrow Y_0 \rightarrow 0 \\
   & \cdots                                                           \\
   & 0 \rightarrow Y_n \rightarrow F_n \rightarrow Y_{n-1} \rightarrow 0 \\
   & \cdots
\end{align*}
Applying $\Ext^*_\cC(P,-)$ we obtain exact sequences
\begin{align*}
   & \Ext^i_\cC(P,F_0) \rightarrow \Ext^i_\cC(P,A) \xrightarrow{\delta} \Ext^{i+1}_\cC(P,Y_0) \\
   & \Ext^{i+1}_\cC(P,F_1) \rightarrow \Ext^{i+1}_\cC(P,Y_0) \xrightarrow{\delta} \Ext^{i+2}_\cC(P,Y_1) \\
   & \cdots                                                                                                            \\
   & \Ext^{i+n}_\cC(P,F_n) \rightarrow \Ext^{i+n}_\cC(P,Y_{n-1}) \xrightarrow{\delta} \Ext^{i+n+1}_\cC(P,Y_n) \ , \\
   & \cdots
\end{align*}
where $\delta$ denotes the connecting homomorphisms. These are sequences of $Z\otimes Z$-bimodules. The first term in each sequence is annihilated by $\cJ$. Now choose $n := d - i$. Then $\Ext^{i+n+1}_\cC(P,Y_n) = 0$. The assertion now follows by downward induction along the above sequences.
\end{proof}

\begin{proof}[Proof of Theorem \ref{thm:MainStability}]
Since $\cL_W^H = \varrho^{-1}(\cL_W^Z)$, it follows from Lemma \ref{lem:varrho-stable} that $\cL_W^H$ is stable in $\Mod(H)$. Hence $\cL_W^{\cA} = \cL_W^{\cC} \cap \cA$ is stable in $\cA$. Therefore, by Thm.\ \ref{thm:stable-A-L}, it is enough to show that the functor $\Ext^1_{\cC}(P,-)$ maps $\cL_W^{\cA}$ to $\cL_W^H$. By Lemma \ref{lem: ExtCommutes}, the functor $\Ext^1_{\cC}(P,-)$ commutes with filtered colimits. Hence it suffices to show that for any noetherian object $A$ in $\cL_W^{\cC} \cap \cA$, the $H$-module $M := \Ext^1_{\cC}(P,A)$ lies in $\cL_W^H$. Since the functor $\Ext^1_{\cC}(P,-)$ is also half-exact,  we may assume further, using a prime series for $\Hom_\cA(P,A)$ as in \cite{GW} Prop.\ 3.13, that the annihilator of any nonzero $H$-submodule of $\Hom_\cA(P,A)$ is a prime ideal $\frp$ of $H$. Then necessarily $\frp = \Ann_H(\Hom_\cA(P,A))$, and we write $\frp_Z := \frp \cap Z \in \Xi$.

By Lemma \ref{lem:J-killing} the $Z \otimes Z$-module $M = \Ext^1_{\cC}(P,A)$ is killed by $\cJ^{d+1}$. Letting $\cR_d := \Spec((Z \otimes Z) / \cJ^{d+1})$, we can then regard $M$ as an $\cO(\cR_d) = (Z \otimes Z) / \cJ^{d+1}$-module in a natural way. Next, consider the commutative diagram
\begin{equation}\label{f:pr1pr2}
\xymatrix{
\cR \ar@<0.5ex>[rrd]^{\pi_1}\ar@<-1ex>[rrd]_{\pi_2}\ar[rrrr]^i &&&& \cR_d \ar@<1ex>[lld]^{f_2}\ar@<-0.5ex>[lld]_{f_1}  \\
 && \Xi && }
\end{equation}
where $i$ is the closed embedding, and the $f_j$ are, similarly as before, the restrictions of the two projection maps.

Let $S := \Xi \setminus W$; then our assumption $\pi_2(\pi_1^{-1}(W)) \subseteq W$ is equivalent to $\pi_1^{-1}(W) \subseteq \pi_2^{-1}(W)$, which is the same as $\pi_2^{-1}(S) \subseteq \pi_1^{-1}(S)$. Since $\pi_j = f_j \circ i$, this is equivalent to $i^{-1} f_2^{-1}(S) \subseteq i^{-1} f_1^{-1}(S)$. However, $i$ is a bijection because $\cJ$ is nilpotent modulo $\cJ^{d+1}$, so this is in turn equivalent to $f_2^{-1}(S) \subseteq f_1^{-1}(S)$.

Now, $\frp$ is the annihilator of any nonzero $H$-submodule of $\Hom_\cA(P,A)$, so $\frp_Z = \frp \cap Z = \Ann_Z(v)$ for all nonzero $v \in \Hom_\cA(P,A)$. Since $\Hom_\cA(P,A)$ lies in $\cL_W^H$, this means that $V(\frp_Z) \cap W = \emptyset$ by Prop.\ \ref{prop:LHW}, and hence $V(\frp_Z) \subseteq S = \Xi \setminus W$.

By assumption $\cO(\cR_d) = (Z \otimes Z) / \cJ^{d+1}$ is noetherian. Let $\overline{i}_2 : Z \to (Z \otimes Z) / \cJ^{d+1}$, sending $z$ to $1 \otimes z + \cJ^{d+1}$, be the ring homomorphism which induces $f_2$. Let $Q_1,\cdots, Q_k \subseteq \cO(\cR_d)$ be the minimal prime ideals lying above $\overline{i}_2(\frp_Z) \cO(\cR_d)$; then $Q_1^{a_1}\cdots Q_k^{a_k} \subseteq \overline{i}_2(\frp_Z) \cO(\cR_d)$ for some positive integers $a_1,\cdots, a_k$. Since $Q_j \in f_2^{-1}(V(\frp_Z)) \subseteq f_2^{-1}(S) \subseteq f_1^{-1}(S)$, we see that $\mathfrak{q}_j := f_1(Q_j)$ lies in $S$ for all $j = 1,\cdots, k$. Hence $V(\mathfrak{q}_j) \subseteq S$ for all $j$, because $S$ is specialization-closed.

With $\Hom_\cA(P,A)$ also $A$ is killed by $\frp_Z$. Hence $M = \Ext^1_{\cC}(P, A)$ is killed by $Z \otimes \frp_Z$ as well as $\cJ^{d+1}$. We see that $M$ is killed by $\overline{i}_2(\frp_Z) \cO(\cR_d) \supseteq Q_1^{a_1}\cdots Q_k^{a_k}$. Hence $\mathfrak{q}_1^{a_1} \cdots \mathfrak{q}_k^{a_k} \cdot v = 0$ for any $v \in M$, so $V(\Ann_Z(v)) \subseteq V(\mathfrak{q}_1) \cup \cdots \cup V(\mathfrak{q}_k) \subseteq S$. Therefore the $H$-module $M$ (where $H$ acts on $M$ through its action on $P$) lies in $\cL^H_W$ by Prop.\   \ref{prop:LHW} as required.
\end{proof}
Now we form the following diagram, which defines the map $\tau$:
\[ \xymatrix{  \Sp(\cA) \ar[rr]^{r^{-1}}_{\cong} \ar[d]_{\cong}&& \Sp(\cC) \ar@{.>}[dd]^\tau \\
\Sp(\Mod(H)) \ar[d]_{\cong} && \\
\Spec(H) \ar[r]_(0.45)\varphi & \Spec(Z) = \Xi \ar[r]_(0.55){\bfq} & \frZ = \Xi / \frR }\]
\begin{proposition} \label{prop:preimoftau} Assume \textbf{(A1), (A2), (A3)}.  If $U \subseteq \frZ$ is an intersection of open subsets, then $\tau^{-1}(U)$ is stable.
\end{proposition}
\begin{proof} Put $W = \bfq^{-1}(U)$. Then
\[\pi_2(\pi_1^{-1}(W)) = \pi_2(\pi_1^{-1} \bfq^{-1}(U)) = \pi_2(\pi_2^{-1}\bfq^{-1}(U)) \subseteq \bfq^{-1}(U) = W.\]
Now apply Thm. \ref{thm:MainStability}.
\end{proof}

\begin{remark}\label{rem: A3klinear} Suppose that our category $\cC$ is $k$-linear. Then the $Z \otimes_{\bZ} Z$-action on the $\Ext$-groups $\Ext^j_{\cC}(M,N)$ for $M,N$ in $\cA$, factors through $Z \otimes_k Z$. Therefore the proof of Theorem \ref{thm:MainStability} goes through if we replace $\mathbf{(A3)}$ by the weaker assumption that $(Z \otimes_k Z) / \overline{\cJ}^{d+1}$ is noetherian, where $\overline{\cJ}$ is the image of $\cJ$ under the surjection $Z \otimes_{\bZ} Z \twoheadrightarrow Z \otimes_k Z$.
\end{remark}

We also apply Thm. \ref{thm:MainStability} to give a sufficient criterion for the stability of the Krull-dimension filtration of the locally noetherian category $\cC$. For every ordinal $\alpha$, the set
\[W_\alpha := \{\frp \in \Xi : \kappa(Z/\frp) > \alpha\}\]
is a generalization-stable subset of $\Xi$. We identify the corresponding localising subcategory $\cL^{\cC}_{W_\alpha}$ of $\cC$ in Lemma \ref{lem: LCalpha} below, but first we recall a standard fact about Krull dimension.

\begin{lemma}\label{lem: KdimFormula} Let $M$ be a finitely generated module over a commutative noetherian ring $A$. Then
\[\kappa(M) = \sup_{v \in M} \kappa(Av) = \sup_{v \in M} \quad \sup_{\frp \in V(\ann(v))} \kappa(A / \frp).\]
\end{lemma}
\begin{proof}
 An induction on the number of generators of $M$ together with \cite{MCR} Lemma 6.2.4 shows the first equality, namely $\kappa(M) = \sup_{v \in M} \kappa(Av)$.

Let $I$ be an ideal of $A$. Since $A$ is noetherian,  \cite{MCR} Lemma 6.2.4 implies that $\kappa(A/I) = \kappa(A / I^m)$ for all $m \geq 0$. Since $(\sqrt{I})^m \subseteq I$ for sufficiently large $m$, we have $\kappa(A/I) \geq \kappa(A/\sqrt{I}) = \kappa(A/(\sqrt{I})^m) \geq \kappa(A/I)$, so in fact we have the equality $\kappa(A/I) = \kappa(A/\sqrt{I})$. Since $\sqrt{I}$ is equal to the intersection of the finitely many minimal primes $\frp_1, \cdots, \frp_n$ containing it, there is a natural $A$-linear embedding $A / \sqrt{I} \hookrightarrow \oplus_{i=1}^n A / \frp_i$. So,
\[\kappa(A/I) = \kappa(A/\sqrt{I}) \leq \sup_{1\leq i \leq n} \kappa(A/\frp_i) \leq \sup_{\frp \in V(I)} \kappa(A/\frp) \leq \kappa(A/I).\] Therefore for any $v \in M$, $\kappa(Av) = \kappa(A/\Ann_A(v)) = \sup\limits_{\frp \in V(\ann(v))} \kappa(A/\frp)$.
\end{proof}
Recall ($\S \ref{sec: KdimSec}$) that $\cC_\alpha$ is the full subcategory of $\cC$ consisting of objects $V$ such that $\kappa(V) \leq \alpha$.
\begin{lemma}\label{lem: LCalpha} For any ordinal $\alpha$, we have $\cL^{\cC}_{W_\alpha} = \cC_\alpha$.
\end{lemma}
\begin{proof} By Gabriel's definition of Krull dimension found at \cite{Gab} p. 382, we see that $\kappa(V) = \sup_i \kappa(V_i)$, if we write $V = \varinjlim V_i$ as a filtered colimit of its noetherian subobjects $V_i$. Therefore, it is enough to show that for noetherian $V \in \cC$, we have $V \in \cL^{\cC}_{W_\alpha}$ if and only if $\kappa(V) \leq \alpha$. Since $\cC$ is a thickening of $\cA$, we may further assume that $V \in \cA$. Hence we are reduced to showing that for finitely generated $H$-modules $M$, we have
\[ M \in \cL^H_{W_\alpha}  \quad \Leftrightarrow \quad \kappa(M) \leq \alpha.\]
However since $H$ is a finitely generated $Z$-module, $\kappa(M) = \kappa(\rho(M))$ by \cite{MCR} Cor. 10.1.10, whereas $M \in \cL^H_{W_\alpha}$ if and only if $\rho(M) \in \cL^Z_{W_\alpha}$ by equation (\ref{f:varrho-stable2}). Next, $\rho(M) \in \cL^Z_{W_\alpha}$ if and only if for all $v \in M$ and all $\frp \in \Spec(Z)$ containing $\Ann_Z(v)$, we have $\kappa(Z/\frp) \leq \alpha$. As $\rho(M)$ is a finitely generated $Z$-module, this is equivalent to $\kappa(\rho(M)) \leq \alpha$ by Lemma \ref{lem: KdimFormula}. \end{proof}

We can now establish the stability of the Krull-dimension filtration on $\cC$ under certain fairly restrictive hypotheses on $\Xi$ and $\cR$.

\begin{proposition}\label{prop:KdimStable}  Suppose that $\cC$ is $k$-linear, that $Z$ is a finitely generated $k$-algebra, and that there is an integer $d$ such that $\Ext^j_{\cC}(P,.)|_{\cA} = 0$ for any $j > d$. \textbf{Suppose further that $\pr_2 : \cR \to \Xi$ is quasi-finite}. Then $\cL_{W_\alpha}^{\cC}$ is a stable localising subcategory of $\cC$ for all $\alpha$.
\end{proposition}
\begin{proof} By assumption, the $k$-algebra $Z \otimes_k Z$ is also finitely generated, so it is noetherian by Hilbert's Basis Theorem. By Thm. \ref{thm:MainStability} and Remark \ref{rem: A3klinear}, it is enough to show that
\[\pr_2( \pr_1^{-1}(W_\alpha)) \subseteq W_\alpha.\]
Let $\cS_\alpha := \Xi \backslash W_\alpha$ be the complement of $W_\alpha$ in $\Xi$; then it is enough to show that
\[ \pr_1(\pr_2^{-1}(\cS_\alpha)) \subseteq \cS_\alpha.\]
Let $P \in \pr_2^{-1}(\cS_\alpha)$ and set $\frp_i := \pr_i(P)$ for $i=1,2$. Then $\frp_2 \in \cS_\alpha$, so $\kappa(Z/\frp_2) \leq \alpha$. The morphisms $\pr_1$ and $\pr_2$ induce the following diagram of residue fields:
\[ k(\frp_1) \hookrightarrow k(P) \hookleftarrow k(\frp_2).\]
Since $\pr_2 : \cR \to \Xi$ is quasi-finite, $k(P)$ is a finite dimensional $k(\frp_2)$-vector space. Hence
\[ \tr.\deg_k(k(\frp_1)) \leq \tr.\deg_k(k(P)) = \tr.\deg_k(k(\frp_2)).\]
Since $Z$ is a finitely generated $k$-algebra, we can apply a version of the Noether Normalization Theorem  --- see \cite{Mat} Thm. 5.6 --- to deduce that
\[ \kappa(Z / \frp_i) = \tr.\deg_k(Z/\frp_i) = \tr.\deg_k(k(\frp_i)) \quad\mbox{for}\quad i=1,2.\]
Hence $\kappa(Z/\frp_1) \leq \kappa(Z/\frp_2) \leq \alpha$, which means that $\frp_1 = \pr_1(P) \in \cS_\alpha$. \end{proof}

\section{The quotient space for $\Mod_k(SL_2(\bQ_p))$} \label{sec:QuotSpace}

\subsection{The pro-$p$ Iwahori Hecke algebra and its centre} \label{sec: propIHA}
\subsubsection{The set-up}Let $\frF$ be a finite extension of $\bQ_p$ with ring of integers $\frO$, maximal ideal $\frM$ and residue field $\frf$. We fix a choice of generator $\pi$ of $\frM$. Necessarily $\frf \cong \bF_q$ for some power $q$ of $p$. \textbf{We assume that our ground field $k$ contains $\frf$}.
\subsubsection{Groups}\label{ssec: groups}
Let $G = \SL_2(\frF)$ and $K = \SL_2(\frO)$, and define the \emph{pro-$p$ Iwahori subgroup}
\[I := \begin{pmatrix} 1 + \frM & \frO \\ \frM & 1 + \frM \end{pmatrix} \cap G.\]
Let $T = \begin{pmatrix} \frF^\times & 0 \\ 0 & \frF^\times \end{pmatrix} \cap G$ be the subgroup of diagonal matrices in $G$ and let $T^0 := T \cap K$. The projection onto the $(1,1)$-entry gives group isomorphisms $T \stackrel{\cong}{\longrightarrow} \frF^\times$ and $T^0  \stackrel{\cong}{\longrightarrow} \frO^\times$. Hence $T / T^0$ is an infinite cyclic group. The normaliser $N_G(T)$ of $T$ in $G$ is generated by $T$ and $s_0 := \begin{pmatrix} 0 & 1 \\ -1 & 0 \end{pmatrix}$, and $N_G(T)$ is isomorphic to the semi-direct product $T \rtimes \langle s_0 \rangle$, where $s_0$ acts on $T$ by inversion. This action preserves the subgroup $T_0$, so that $T_0$ is normal in $N_G(T)$. Then we can form the \emph{affine Weyl group} $W := N_G(T) / T_0$ which is isomorphic to the infinite dihedral group, generated by the $T_0$-cosets of $s_0$ and $s_1 := \begin{pmatrix} 0 & - \pi^{-1} \\ \pi & 0 \end{pmatrix}$. The subgroup $T^1 := T \cap I$ is also normal in $N_G(T)$ and this gives the \emph{extended Weyl group} $\widetilde{W} := N_G(T) / T^1$. Let $\Omega := T^0 / T^1$; then $\Omega \cong \frf^\times$ and we have the short exact sequence
\[ 1 \to \Omega \to \widetilde{W} \to W \to 1.\]
Note that $\Omega$ is isomorphic to $\bF_q^\times$ and is therefore a finite cyclic group of order $q-1$.
\subsubsection{The pro-$p$ Iwahori Hecke algebra}
Let $\hatOm := \Hom(\Omega, k^\times)$ be the group of $k$-linear characters of $\Omega$. Since $k$ contains the residue field $\frf$ of $\frF$ by assumption, $\hatOm$ is a cyclic group of order $q-1$, generated by the \emph{identity character}
\[ \id : \Omega \to k^\times, \quad \begin{pmatrix} a & 0 \\ 0 & a^{-1} \end{pmatrix} T^1 \mapsto a + \frM \in \frf^\times \hookrightarrow k^\times.\]
In other words, we have $\hatOm = \{\id^j : 0 \leq j \leq q-2\}$ where $\id^0 = \triv$ is the \emph{trivial character}. The group algebra $k[\Omega]$ contains $q-1$ primitive idempotents
\begin{equation}\label{eq: defelambda} e_\lambda := - \sum\limits_{\omega \in \Omega} \lambda(\omega)^{-1} \omega, \quad\quad \lambda \in \hatOm\end{equation}
that form a basis for $k[\Omega]$ as a $k$-vector space. These idempotents are pairwise orthogonal:
\begin{equation}\label{eq: cpi} e_\lambda e_\mu = \delta_{\lambda \mu} e_{\lambda} \qmb{for all} \lambda,\mu \in \widehat {\Omega}.\end{equation}
We have the compactly induced $G$-representation
\[ \ind_I^G(k) := k[G/I] =  k[G] \underset{k[I]}{\otimes}{} k\]
from the trivial representation $k$ of $I$. It is a smooth $k$-linear representation of $G$, so $k[G/I]$ is an object in $\Mod_k(G)$. The \emph{pro-$p$ Iwahori Hecke algebra} is the opposite ring of the endomorphism ring of this representation:
\[ H := \End_{\Mod_k(G)}(k[G/I])^{\op}.\]
The following facts are known about $H$.
\begin{lemma}\label{lem: Hpres} \
\begin{itemize}
\item[a)] The double cosets $\{\tau_w := IwI : w \in \widetilde{W}\}$ form a $k$-basis for $H$.
\item[b)] The map $k[\Omega] \to H$ given by $\Omega \ni \omega \mapsto \tau_\omega$ is an injection of $k$-algebras.
\item[c)] The elements $\tau_0 := \tau_{s_0}$ and $\tau_1 := \tau_{s_1}$ generate $H$ as a $k$-algebra together with $k[\Omega]$.
\item[d)] We have $\tau_i e_\lambda = e_{\lambda^{-1}} \tau_i$ for $i=0,1$ and any $\lambda \in \hatOm$.
\item[e)] We have the \emph{quadratic relations} $\tau_i^2 = - \tau_i e_\triv$ for $i=0,1$.
\end{itemize}
\end{lemma}
We will never use the notation $\tau_w$ when $w$ is the identity element of $\widetilde{W}$, because this is equal to the identity element in $H$. Therefore writing $\tau_1 = \tau_{s_1}$ should not lead to any confusion.

\subsubsection{The centre of $H$} \label{sec: CentOfH}
The centre $Z = Z(H)$ of this $k$-algebra can be described explicitly. The following statement is well-known, but see also \cite{OS18} $\S$3.2.2.
\begin{lemma}\label{lem: zeta} The element
\[\zeta := (\tau_0 + e_\triv)(\tau_1+e_\triv) + \tau_1\tau_0 = (\tau_1 + e_\triv)(\tau_0 + e_\triv) + \tau_0\tau_1\]
generates a polynomial ring $k[\zeta]$ inside $Z(H)$.
\end{lemma}

\begin{definition}\label{def: Xlambda} For each $\lambda \in \hatOm$, we define the following elements of $H$:
\[ X_\lambda = \left\{ \begin{array}{cll} e_\lambda \tau_0 \tau_1 + e_{\lambda^{-1}} \tau_1 \tau_0 & : & \lambda \neq \triv \\ e_\triv \zeta & : & \lambda = \triv. \end{array} \right.\]
\end{definition}

Again, we omit the proof of the following well-known result.
\begin{lemma} \label{lem: ezeta}  Let $\lambda \in \hatOm$.
\begin{itemize}
\item[a)] $X_\lambda$ is central in $H$.
\item[b)] If $\lambda \neq \lambda^{-1}$, then $(e_\lambda + e_{\lambda^{-1}}) \zeta = X_\lambda + X_{\lambda^{-1}}$.
\item[c)] If $\lambda = \lambda^{-1}$, then $e_\lambda \zeta = X_\lambda$.
\item[d)] We have $\sum\limits_{\alpha \in \hatOm} X_\alpha  = \zeta$.
\end{itemize}
\end{lemma}

We will now describe $Z$ completely as a $k$-algebra with relations and generators. Suppose that $\lambda \in \hatOm$ satisfies $\lambda \neq \lambda^{-1}$; then using Lemma \ref{lem: Hpres}, we see that the idempotent $e_\lambda + e_{\lambda^{-1}}$ is central in $H$. It is known, furthermore, that in this case
\[ X_\lambda \cdot X_{\lambda^{-1}} = 0.\]
Note that precisely two characters $\lambda \in \hatOm$ satisfy $\lambda = \lambda^{-1}$, namely the trivial character $1 = \id^0$ and ``the quadratic character'' $\id^{\frac{q-1}{2}}$. All other elements of $\hatOm$ are not equal to their inverse. This gives us the decomposition of $Z$ as a direct product of non-unital subrings:
\begin{equation}\label{eq:Zdecomp} Z = e_\triv Z \quad\oplus\quad e_{\id^{\frac{q-1}{2}}} Z \quad\oplus\quad \bigoplus_{j=1}^{\frac{q-3}{2}} (e_{\id^j} + e_{\id^{-j}})Z\end{equation}
From $\S 3.2$ of \cite{OS18} we deduce the following
 \begin{proposition}\label{prop: Zcomps} Let $\lambda \in \hatOm$.
 \begin{itemize} \item[a)] Suppose that $\lambda = \lambda^{-1}$. Then the $k$-algebra homomorphism
\[ k[x] \to e_\lambda Z, \quad 1 \mapsto e_\lambda, \quad x \mapsto X_\lambda\]
is an isomorphism.
\item[b)] Suppose that $\lambda \neq \lambda^{-1}$. Then the $k$-algebra homomorphism
\[ b_\lambda : \frac{ k[x, y] } {\langle xy\rangle} \to (e_\lambda + e_{\lambda^{-1}}) Z,   \quad 1 \mapsto e_{\lambda} + e_{\lambda^{-1}}, \quad x \mapsto X_\lambda, \quad y \mapsto X_{\lambda^{-1}} \]
is an isomorphism.\end{itemize}\end{proposition}
\subsection{The normalisation of $\Spec(Z)$}

\subsubsection{Categorical quotients of schemes}

Let $\xymatrix{\cR \ar@<0.6ex>[r]^f\ar@<-0.6ex>[r]_g & X}$ be two morphisms of schemes. Recall that a morphism of schemes $\psi : X \to Y$ is said to be a \emph{categorical quotient of $X$ by $\cR$} if it is a coequaliser of this diagram in the category of schemes.

Given a morphism of locally ringed spaces $f : X \to Y$, we write $|f| : |X| \to |Y|$ for the underlying continuous map of topological spaces.

Let $\xymatrix{\cR \ar@<0.6ex>[r]^f\ar@<-0.6ex>[r]_g & X \ar[r]^\psi & Y}$ be a diagram of schemes such that $\psi f = \psi g$, and let $\{Y_i : i \in I\}$ be an open covering of $Y$. For each $i \in I$, form the fibre products $X_i := X \times_Y Y_i$ and $\cR_i := \cR \times_Y Y_i$, and consider the diagram $\xymatrix{\cR_i \ar@<0.6ex>[r]^{f_i}\ar@<-0.6ex>[r]_{g_i} & X_i \ar[r]^{\psi_i} & Y_i}$ where $f_i,g_i$ and $\psi_i$ are the pullbacks of $f,g$ and $\psi$, respectively.

We denote the category of schemes by $\Sch$, and the category of locally commutatively ringed topological spaces by $\LRS$.

\begin{lemma}\label{lem: catquot} Let $\xymatrix{\cR \ar@<0.6ex>[r]^f\ar@<-0.6ex>[r]_g & X \ar[r]^\psi & Y}$ be a diagram of schemes such that $\psi f = \psi g$, and let $\{Y_i : i \in I\}$ be an open covering of $Y$. Suppose that
\begin{itemize}
\item[a)] $\xymatrix{|\cR| \ar@<0.6ex>[r]^{|f|}\ar@<-0.6ex>[r]_{|g|} & |X| \ar[r]^{|\psi|} & |Y|}$ is a coequaliser diagram of topological spaces,
\item[b)] $Y_i$, $X_i$ and $\cR_i$ is affine for all $i \in I$, and
\item[c)] $\xymatrix{\cO(Y_i) \ar[r]^{\psi_i} & \cO(X_i) \ar@<0.6ex>[r]^{f_i}\ar@<-0.6ex>[r]_{g_i} & \cO(\cR_i)}$ is an equaliser diagram of commutative rings for all $i \in I$.
\end{itemize}
Then $\xymatrix{\cR \ar@<0.6ex>[r]^f\ar@<-0.6ex>[r]_g & X \ar[r]^\psi & Y}$ is a coequaliser in $\LRS$ \footnote{Note that the functor from schemes to LRS does not respect coequalisers (see \cite{LMB} for an example).}, and $\psi$ is a categorical quotient of $\xymatrix{\cR \ar@<0.6ex>[r]^f\ar@<-0.6ex>[r]_g & X}$.
\end{lemma}
\begin{proof}
By the construction of colimits in $\LRS$ (\cite{DG} Prop. I.1.1.6) the first statement is equivalent to the following conditions:
\begin{itemize}
  \item[--] $\xymatrix{|\cR| \ar@<0.6ex>[r]^{|f|}\ar@<-0.6ex>[r]_{|g|} & |X| \ar[r]^{|\psi|} & |Y|}$ is a coequaliser diagram of topological spaces;
  \item[--] $\xymatrix{
    \cO_Y \ar[r] & \psi_* \cO_X \ar@<1ex>[r]^-{f}  \ar@<-1ex>[r]_-{g} & (\psi f)_* \cO_{\cR} = (\psi g)_* \cO_{\cR} }$ is an equaliser diagram of sheaves of commutative rings on $Y$.
\end{itemize}
The second condition can be checked after restriction to the open subschemes $Y_i$. By the second assumption these and there preimages in $X$ and $\cR$ all are affine. Hence these restricted sheaves are determined by there global sections in these affine schemes. Therefore this reduces to the third assumption.

Since $\Sch$ is a full subcategory of $\LRS$, the second statement follows from the first.
\end{proof}

We will now give an example of Lemma \ref{lem: catquot} in action, which will come in useful later. Consider the union of two crossing affine lines:
\[ \cX := \Spec \frac{k[x,y]}{\langle xy \rangle}  \hsp .\]
We will `separate' the two crossing affine lines to form a disjoint union of two affine lines
\[\cX' := \Spec \left( k[x] \times k[y] \right)\]
and then re-glue them together to form $\cX$ as a quotient of $\cX'$ by a certain relation.

To be more precise, note that $\cX'$ is the disjoint union of $\Spec k[x]$ and $\Spec k[y]$. Let $\cS := \Spec k = \{s\}$ be a point, let $a : \cS \to \cX'$ be the inclusion of the origin into the first affine line, and let $b : \cS \to \cX'$ be the inclusion of the origin into the second affine line.

Consider the $k$-algebra homomorphism
\[\varphi : \frac{k[x,y]}{\langle xy \rangle} \to k[x] \times k[y], \quad \quad \varphi\left(\overline{f(x,y)}\right) = ( f(x,0), f(0,y) ) \qmb{for any}  f(x,y) \in k[x,y],\]
and let $\theta := \Spec(\varphi) : \cX' \to \cX$ be the corresponding morphism of affine schemes.
\begin{proposition} \label{prop: nX}$\theta : \cX' \to \cX$ is a categorical quotient of $\xymatrix{\cS \ar@<0.6ex>[r]^{a}\ar@<-0.6ex>[r]_{b} & \cX'}$.
\end{proposition}
\begin{proof} We will apply Lemma \ref{lem: catquot}. The co-morphism $a^\sharp : \cO(\cX') \to \cO(\cS)$ (respectively, $b^\sharp:\cO(\cX') \to \cO(\cS)$) is equal to the composition of the first projection from $k[x] \times k[y]$ onto $k[x]$ (respectively, the second projection $k[x] \times k[y]$ onto $k[y]$) with the evaluation-at-zero map $k[x] \to k$, $f(x) \mapsto f(0)$ (respectively, $k[y] \to k$, $f(y) \mapsto f(0)$). The $k$-algebra homomorphisms $a^\sharp \varphi$ and $b^\sharp \varphi$ both kill $\overline{x}$ and $\overline{y}$; since these two elements generate $\cO(\cX)$ as a $k$-algebra, we see that $a^\sharp \varphi = b^\sharp \varphi$. This gives us a complex of $\cO(\cX)$-modules
\begin{equation}\label{eq: phiab} 0 \to \frac{k[x,y]}{\langle xy \rangle} \quad \stackrel{\varphi }{\longrightarrow} \quad k[x] \times k[y] \quad \stackrel{a^\sharp - b^\sharp}{\longrightarrow} \quad k.\end{equation}
Since $\langle x \rangle \cap \langle y \rangle = \langle xy \rangle$ in $k[x,y]$, the $k$-algebra homomorphism $\varphi$ is injective. If $u = (f(x), g(y)) \in k[x] \times k[y]$ is such that $a^\sharp(u) = b^\sharp(u)$, then $f(0) = g(0)$, and then
\[\varphi\left(\overline{f(x) + g(y)}\right) = (f(x), f(0)) + (g(0), g(y)) = (g(0), f(0)) + (f(x),f(y)) = \varphi\left( \overline{f(0)}\right) + u\]
shows that $\ker (a^\sharp - b^\sharp) = \mathrm{Im}(\varphi)$. So, the complex $(\ref{eq: phiab})$ is in fact exact, and this verifies condition c) of Lemma \ref{lem: catquot}.

We have already observed that $\overline{x},\overline{y} \in \cO(\cX)$ both kill $\cO(\cS) = k$ when we consider $\cO(\cS)$ as an $\cO(\cX)$-module via $a^\sharp \varphi = b^\sharp \varphi$. So it is also killed by $\overline{x} + \overline{y}$. Therefore if we localise the exact sequence (\ref{eq: phiab}) at the element $\overline{x} + \overline{y}\in \cO(\cX)$, then we obtain the isomorphism
\[ \left(\frac{k[x,y]}{\langle xy \rangle}\right)\left[ \frac{1}{\overline{x} + \overline{y}} \right] \quad \stackrel{\cong}{\longrightarrow} \quad k[x, x^{-1}] \times k[y, y^{-1}].\]
Let $O := \langle \overline{x}, \overline{y} \rangle \in \cX$ be the intersection point of the two crossing affine lines. Then the restriction of $\theta = \Spec(\varphi)$ to $\cX' - \{a(s),b(s)\} = \Spec \left(k[x, x^{-1}] \times k[y, y^{-1}]\right)$ is an isomorphism onto $\cX \backslash \{O\}$. Since $\theta$ maps both origins in $\cX'$ to $O \in \cX$, we see that $|\theta|$ is surjective, and that if $p,q \in |\cX'|$ satisfy $|\theta|(p) = |\theta|(q)$ but $p\neq q$, then necessarily $p,q \in \{a(s), b(s)\}$. Thus $|\cX|$ is obtained as a set by identifying $a(s)$ with $b(s)$ in $|\cX'|$.

Finally, both $|\cX'|$ and $|\cX|$ carry the cofinite Zariski topologies, and the quotient topology on $|\cX'| / (a(s) \sim b(s))$ is still the cofinite topology. Hence $\xymatrix{|\cS| \ar@<0.6ex>[r]^{|a|}\ar@<-0.6ex>[r]_{|b|} & |\cX'|} \to |\cX|$ is a coequaliser diagram of topological spaces, as required for condition a) of Lemma \ref{lem: catquot}.
\end{proof}

\subsubsection{The normalisation of $\Spec(Z)$} We return to the notation of $\S \ref{sec: CentOfH}$.

\begin{definition} \hsp \label{def: XiXiprimePhi}
\begin{itemize}
\item[a)] Let $\Xi = \Spec (Z)$.
\item[b)] Let $Z' := k[\Omega][t]$ where $t$ is a formal variable, and let $\Xi' := \Spec Z'$.
\item[c)] Let $\varphi : Z \to Z'$ be the $k$-algebra homomorphism defined by
\[ \varphi(X_\lambda)  = e_\lambda t \qmb{for all} \lambda \in \hatOm.\]
\item[d)] Let $\theta : \Xi' \to \Xi$ be defined by $\theta = \Spec(\varphi)$.
\end{itemize}
\end{definition}

Note that $\Xi'$ is the disjoint union of $|\hatOm| = q-1$ affine lines $\Xi'_\lambda := \Spec k[t_\lambda]$, where $t_\lambda := e_\lambda t \in Z'$. We call $\Xi'_\lambda$ the \emph{$\lambda$-component} of $\Xi'$. We denote the image of the origin $\langle t_\lambda \rangle \in \Xi'_\lambda$ inside $\Xi'$ by $O_\lambda$. Equivalently, $O_\lambda$ is the maximal ideal $\langle 1 - e_{\lambda}, t \rangle$ of $Z'$.

\begin{definition}\label{def: XiSingab} \hsp
\begin{itemize}
\item[a)] Let $\Xi_{\sing} := \{s_1,s_2,\cdots, s_{\frac{q-3}{2}}\}$ be the disjoint union of $\frac{q-3}{2}$ copies of $\Spec k$.
\item[b)] Define $a : \Xi_{\sing} \to \Xi'$ and $b : \Xi_{\sing} \to \Xi'$ to be the closed embeddings, given by
\[a(s_j) =  O_{\id^j} \qmb{and} b(s_j) = O_{\id^{-j}} \qmb{for all} j = 1,\cdots, \frac{q-3}{2}.\]
\end{itemize}
\end{definition}
The notation is intended to indicate that $\Xi_{\sing}$ is a copy of the singular locus of the scheme $\Xi$: indeed, in view of the decomposition (\ref{eq:Zdecomp}) and Prop. \ref{prop: Zcomps}.b, $\Xi$ contains $\frac{q-3}{2}$ copies of the pair of crossing lines $\Spec \frac{k[x,y]}{\langle xy \rangle}$, so the singular locus of $\Xi$ consists of $\frac{q-3}{2}$ closed points.
\begin{proposition}\label{prop: NormalisationAsQuotient} $\theta : \Xi' \to \Xi$ is a categorical quotient of $\xymatrix{\Xi_{\sing} \ar@<0.6ex>[r]^{a}\ar@<-0.6ex>[r]_{b} & \Xi'}$.
\end{proposition}
\begin{proof} The decomposition of $Z = \cO(\Xi)$ as a direct product of non-unital subrings given in (\ref{eq:Zdecomp}) means that $\Xi$ decomposes into the disjoint union of connected components
\begin{equation}\label{eq: XiDecomp} \Xi = \Xi_0 \quad \cup \quad \Xi_{\frac{q-1}{2}} \quad \cup \quad \coprod\limits_{j=1}^{\frac{q-3}{2}} \Xi_{j,-j}\end{equation}
where these connected components are defined by
\[\Xi_0 := \Spec e_\triv Z, \quad \Xi_{\frac{q-1}{2}} := \Spec e_{\id^{\frac{q-1}{2}}} Z, \qmb{and} \Xi_{j,-j} := \Spec (e_{\id^j} + e_{\id^{-j}})Z\]
for all $j = 1,\cdots, \frac{q-3}{2}$. On the other hand, we can write
\begin{equation} \label{eq: XiDashDecomp} \Xi' = \Xi_{\triv} \quad \cup \quad \Xi_{\id^{\frac{q-1}{2}}} \quad \cup \quad \coprod \limits_{j=1}^{\frac{q-3}{2}} (\Xi'_{\id^j} \cup \Xi'_{\id^{-j}}).\end{equation}
Fix $j = 1,\cdots, \frac{q-3}{2}$. The morphism $\theta$ maps $(\Xi'_{\id^j} \cup \Xi'_{\id^{-j}})$ onto $\Xi_{j,-j}$; let $\theta_j$ denote the restriction of $\theta$ to $(\Xi'_{\id^j} \cup \Xi'_{\id^{-j}})$. Then we have the commutative diagram of affine schemes
\begin{equation}\label{eq: alphabetasquare} \xymatrix{ \Xi'_{\id^j} \cup \Xi'_{\id^{-j}}  \ar[d]_{\alpha_j} \ar[r]^(0.6){\theta_j} & \Xi_{j,-j} \ar[d]^{\beta_j}\\ \cX' \ar[r]_{\theta_0} & \cX }\end{equation}
where $\theta_0 : \cX' \to \cX$ is the morphism that was denoted $\theta$ in Prop. \ref{prop: nX}, $\beta_j := \Spec(b_{\id^j})$ is the isomorphism of affine schemes defined by the isomorphism $b_{\id^j}$ from Prop. \ref{prop: Zcomps}.b, and $\alpha_j$ is the isomorphism of affine schemes defined by the $k$-algebra isomorphism
\[ \alpha_j^\sharp : k[x] \times k[y] \quad \longrightarrow k[t_{\id^j}] \times k[t_{\id^{-j}}]\]
given by $\alpha_j^\sharp(x)= t_{\id^j}$ and $\alpha_j^\sharp(y)= t_{\id^{-j}}$. This gives us the commutative diagram of schemes
\begin{equation}\label{eq: jgluing} \xymatrix{ \{s_j\} \ar@<0.6ex>[rrr]^{a_j}\ar@<-0.6ex>[rrr]_{b_j} \ar[d]_\cong &&&  \Xi'_{\id^j} \cup \Xi'_{\id^{-j}}  \ar[d]_{\alpha_j}^{\cong} \ar[rrr]^{\theta_j} &&& \Xi_{j,-j} \ar[d]^{\beta_j}_{\cong}\\ \cS\ar@<0.6ex>[rrr]^{a_0}\ar@<-0.6ex>[rrr]_{b_0} &&& \cX' \ar[rrr]_{\theta_0} &&& \cX }\end{equation}
where $a_0$ and $b_0$ are the morphisms that were denoted $a$ and $b$ in the statement of Prop. \ref{prop: nX}. Now we can apply Prop. \ref{prop: nX} to see that $\theta_j$ is a coequaliser of $a_j, b_j$.

Note that the morphism $\theta$ sends $\Xi'_\triv$ (respectively, $\Xi'_{\id^{\frac{q-1}{2}}}$) isomorphically onto $\Xi_0$ (respectively, $\Xi'_{\frac{q-1}{2}}$). Taking coproducts of the diagrams (\ref{eq: jgluing}) over $j = 1,\cdots, \frac{q-3}{2}$ and bearing in mind the decompositions $(\ref{eq: XiDecomp})$ and $(\ref{eq: XiDashDecomp})$ now gives the result in view of Lemma \ref{lem: CoprodOfCoeq}.
\end{proof}
\subsubsection{The non-singular locus of $\Xi$}Let $V(t)$ denote the set of $q-1$ closed points in $\Xi'$ consisting of the disjoint union of the origins in the $q-1$ affine lines $\Xi'_\lambda = \Spec k[t_\lambda]$.
\begin{lemma}\label{lem: nslocus} The morphism $\theta$ restricts to an isomorphism
\[ \mathring{\theta} : \Xi' - V(t) \quad \stackrel{\cong}{\longrightarrow} \quad \Xi - V(\zeta).\]
\end{lemma}
\begin{proof} Let $\lambda \in \hatOm$ and write $\lambda = \id^j$ for some $j = 0,\cdots, q-2$. Suppose first that $j \neq 0$ and $j \neq \frac{q-1}{2}$. Then the image of $\zeta$ in $\cO(\Xi_{\lambda,\lambda^{-1}})$ under the restriction map $Z = \cO(\Xi) \to \cO(\Xi_{j,-j})$ is $(e_\lambda + e_{\lambda^{-1}}) \zeta$, which is equal to $X_\lambda + X_{\lambda^{-1}}$ by Lemma \ref{lem: ezeta}.b. We see that
$\beta_j(V(\zeta) \cap \Xi_{j,-j})$ is the origin $O$ in the pair of crossing lines $\cX$. Likewise, if $j = 0$ or $j = \frac{q-1}{2}$, Lemma \ref{lem: ezeta}.c implies that $V(\zeta) \cap \Xi_{\frac{q-1}{2}} = V( e_\lambda \zeta ) $ is the origin in the affine line $\Xi_j$. So, $V(\zeta) \subset \Xi$ consists of precisely $\frac{q-3}{2} + 2 = \frac{q+1}{2}$ closed points and contains the singular locus of $\Xi$, whereas $V(t)$ consists of $|\Spec k[\Omega]| = |\hatOm| = q-1$ closed points in $\Xi'$.

The result now follows from the commutative diagram (\ref{eq: alphabetasquare}), together with the observation we made in the proof of Prop. \ref{prop: nX} that $\theta_0$ restricts to an isomorphism between $\cX' \backslash \{a_0(s),b_0(s)\}$ and $\cX \backslash \{O\}$.
\end{proof}
\begin{corollary}\label{cor: Zzeta} There is a $k$-algebra isomorphism
\[ \mathring{\theta}^\sharp : Z_\zeta = \cO(\Xi - V(\zeta)) \quad \stackrel{\cong}{\longrightarrow} \quad \cO(\Xi - V(t)) = k[\Omega][t,t^{-1}]\]
which sends $\zeta$ to $t$.
\end{corollary}
Of course, $\mathring{\theta}^\sharp$ restricts to the basic map $\varphi : Z \to Z'$ on $Z$ from Def. \ref{def: XiXiprimePhi}.c.
\begin{definition}\label{def: epslambda} For each $\lambda \in \hatOm$ we define $\eps_\lambda := (\mathring{\theta}^\sharp)^{-1}(e_\lambda)$. \end{definition}
The family $\{\eps_\lambda : \lambda \in \hatOm\}$ form a complete set of primitive idempotents in $Z_\zeta$.
\begin{lemma}\label{lem: epsilonlambda} We have $\eps_\lambda =  \frac{X_\lambda}{\zeta}$ for all $\lambda \in \hatOm$.
\end{lemma}
\begin{proof} Since $\zeta$ is a unit in $Z_\zeta$, it's enough to show $\eps_\lambda \zeta^2 = X_\lambda \zeta$. Since $\mathring{\theta}^\sharp$ is a ring isomorphism that restricts to $\varphi$ on $Z$, it is enough to show that the equation $e_\lambda \varphi(\zeta)^2 = \varphi(X_\lambda \zeta)$ holds in $k[\Omega][t,t^{-1}]$. However, using Lemma \ref{lem: ezeta}.d together with Def. \ref{def: XiXiprimePhi}.c, we have
\begin{equation}\label{eq: phizeta} \varphi(\zeta) = \varphi\left(\sum\limits_{\alpha \in \hatOm} X_\alpha \right) =  \left( \sum\limits_{\alpha \in \hatOm} e_\alpha t\right) =  t\end{equation}
and therefore $e_\lambda \varphi(\zeta)^2 = e_\lambda t^2 = (e_\lambda t) t = \varphi(X_\lambda) \varphi(\zeta) = \varphi(X_\lambda \zeta)$ as required.
\end{proof}

\begin{corollary}\label{cor: ZzetaDecomp} As a $k[\zeta,\zeta^{-1}]$-module, $Z_\zeta$ decomposes as follows:
\[ Z_\zeta = \bigoplus\limits_{\lambda \in \hatOm} \eps_\lambda k[\zeta, \zeta^{-1}].\]
\end{corollary}
\begin{proof} The isomorphism $\mathring{\theta}^\sharp : Z_\zeta \stackrel{\cong}{\longrightarrow} k[\Omega][t, t^{-1}]$ given by Lemma \ref{lem: nslocus} sends $\zeta$ to $t$ by equation (\ref{eq: phizeta}), and it sends $\eps_\lambda \in Z_\zeta$ to $e_\lambda \in k[\Omega][t,t^{-1}]$ for each $\lambda \in \hatOm$. Since $\{e_\lambda : \lambda \in \hatOm\}$ forms a $k[t,t^{-1}]$-module basis for $k[\Omega][t,t^{-1}]$, it follows that $\{\eps_\lambda : \lambda \in \hatOm\}$ forms a $k[\zeta,\zeta^{-1}]$-module basis for $Z_\zeta$.\end{proof}

\subsection{Bimodule annihilators}
\subsubsection{Some generalities about annihilators and bimodules}  \label{sec: GenBimod}
\begin{definition} Let $A$ be a $k$-algebra, and let $\phi : A \to A$ be a $k$-algebra automorphism. Let $M$ be an $A$-bimodule. We write $M_\phi$ to denote the $(A,A)$-bimodule $M$, with action given by
\[ a \cdot m \cdot b = a \hsp m \hsp \phi(b) \qmb{for all} a,b \in A, m \in M\]
and we call $M_\phi$ the \emph{right $\phi$-twist of $M$}.
\end{definition}
Thus, the left $A$-action on $M_\phi$ is the usual one, whereas the right $A$-action is twisted by $\phi$.

\begin{lemma}\label{lem: TwistedAnn}  Let $A$ be a $k$-algebra, let $M$ be an $(A,A)$-bimodule and let $\phi : A \to A$ be a $k$-algebra automorphism. Then
\[ \Ann_{A \otimes_k A}(M_\phi) = (1 \otimes \phi^{-1})(\Ann_{A \otimes_k A}(M)).\]
\end{lemma}
\begin{proof} Let $x = \sum_{i=1}^n a_i \otimes b_i \in A \otimes_k A$. Then $x \cdot M_\phi = 0$ if and only if $\sum_{i=1}^n a_i m \phi(b_i) = 0$ for all $m \in M$. This is equivalent to $(1 \otimes \phi)(x) \cdot M = 0$ and the result follows.
\end{proof}

\begin{lemma}\label{lem: HZbimodann} Let $A$ be a $k$-algebra with centre $Z$ and let $\mult_Z : Z \otimes_k Z \to Z$ be the multiplication map. Then
\[\Ann_{Z \otimes_k Z}(A) = \Ann_{Z\otimes_kZ}(Z) = \Ann_{Z \otimes_k Z}(1) = \ker \mult_Z.\]
\end{lemma}
\begin{proof} For an element $x = \sum_{i=1}^n a_i\otimes b_i \in Z \otimes_k Z$, we have $x \cdot 1 = \sum\limits_{i=1}^n a_ib_i = \mult_Z(x)$. This shows the last equality. Since $A \supseteq Z \ni 1$, we have $\Ann_{Z \otimes_k Z}(A) \subseteq \Ann_{Z \otimes_k Z}(Z) \subseteq \Ann_{Z \otimes_k Z}(1) = \ker \mult_Z$ by the above. Let $x = \sum_{i=1}^n a_i \otimes b_i \in \ker \mult_Z$ and let $c \in A$. Then because $a_i, b_i$ are central in $A$, we have  $x \cdot c =  \sum_{i=1}^n a_i c b_i = \left(\sum_{i=1}^n a_i b_i\right) c = \mult_Z(x)c = 0$. Hence $\ker \mult_Z \subseteq \Ann_{Z \otimes_kZ}(A)$.\end{proof}

\begin{lemma}\label{lem: BiAnnA} Let $A$ be a commutative $k$-algebra and let $\mult_A : A \otimes_k A \to A$ be the multiplication map.
\begin{itemize}
\item[a)] $\ker \mult_A$ is generated as an ideal in $A \otimes_k A$ by $\{ a \otimes 1 - 1 \otimes a : a \in A\}$.
\item[b)] Suppose that $a_1,\cdots, a_m$ generate $A$ as a $k$-algebra. Then $\{a_i \otimes 1 - 1 \otimes a_i : i=1,\cdots, m\}$ generates $\ker \mult_A$ as an ideal in $A \otimes_kA$.
\end{itemize}
\end{lemma}
\begin{proof} a) Certainly $\mult_A$ kills $a \otimes 1 - 1 \otimes a$ for all $a \in A$. Conversely, if $x = \sum_{i=1}^n a_i \otimes b_i \in A \otimes A$ is such that $\mult_A(x) = \sum_{i=1}^n a_i b_i = 0$, then
\[x = \sum_{i=1}^n (a_i \otimes b_i - 1 \otimes a_i b_i) = \sum_{i=1}^n (a_i \otimes 1 - 1 \otimes a_i) (1 \otimes b_i)\]
lies in the ideal in $A \otimes A$ generated by $\{ a \otimes 1 - 1 \otimes a : a \in A\}$.

b) Write $\partial(a) = a \otimes 1 - 1 \otimes a$ for $a \in A$. Then for all $a,b \in A$ we have
\[ \partial(ab) = ab \otimes 1 - 1 \otimes ab = (a \otimes 1 - 1 \otimes a)(b \otimes 1) + (1 \otimes a)(b \otimes 1 - 1 \otimes b) = \partial(a) (b \otimes 1) + (1 \otimes a) \partial(b).\]
Hence the ideal generated by $\{\partial(a_i) : i=1,\cdots, m\}$ contains $\partial(a_1^{k_1}\cdots a_m^{k_m})$ for all non-negative integers $k_1,\cdots, k_m$. Since $a \mapsto \partial(a)$ is $k$-linear, and since $A = k[a_1,\cdots,a_m]$ by assumption, this ideal contains $\partial(a)$ for all $a \in A$. Hence it contains $\ker \mult_A$ by part a). The reverse inclusion is clear because $\mult_A(\partial(a_i)) = 0$ for all $i$.
\end{proof}

\subsubsection{The bimodule $\left(\frac{H_\zeta}{H_\zeta \tau_0} \oplus \frac{H_\zeta}{H_\zeta \tau_1}\right)[\kappa,z, \mu]$} \label{sec: BimoduleB}
Recall the canonical involution $\iota : H \to H$ from \cite{OS22}, \S 2.4.3: it is a self-inverse $k$-algebra automorphism of $H$ that fixes the subalgebra $k[\Omega]$ pointwise, and satisfies
\[ \iota(\tau_0) = -e_\triv - \tau_0 \qmb{and} \iota(\tau_1) = -e_\triv - \tau_1.\]

\begin{lemma}\label{lem: iotafixesZ} $\iota$ fixes the centre $Z$ of $H$ pointwise.
\end{lemma}
\begin{proof} The decomposition (\ref{eq:Zdecomp}) together with Prop. \ref{prop: Zcomps} implies that the central elements $\{X_\lambda : \lambda \in \hatOm\}$ generate $Z$ as a $k$-algebra. Since $\iota$ is a $k[\Omega]$-linear involution, it therefore suffices to check that $\iota$ fixes each $X_\lambda$. Now if $\lambda \neq \triv$, then since $e_\triv e_\lambda = e_\triv e_{\lambda^{-1}} = 0$ by (\ref{eq: cpi}), we compute
\[\iota(X_\lambda) = \iota( e_\lambda \tau_0\tau_1 + e_{\lambda^{-1}}\tau_1\tau_0 ) = e_\lambda (-e_\triv - \tau_0)(-e_\triv - \tau_1) + e_{\lambda^{-1}}(-e_\triv - \tau_1)(-e_\triv - \tau_0) = X_\lambda.\]
On the other hand, since $\zeta = (\tau_0 + e_\triv)(\tau_1 + e_\triv) + \tau_1\tau_0$ and since $\iota$ fixes $e_\triv \in k[\Omega]$, we have
\[\iota(X_\triv) = e_\triv \iota(\tau_0 + e_{\triv}) \iota(\tau_1 + e_\triv) + \iota(\tau_1)\iota(\tau_0)  = e_\triv \tau_0\tau_1 + (-e_\triv-\tau_1)(-e_\triv-\tau_0) = e_\triv \zeta = X_\triv,\]
where we used the second formula for $\zeta$ from Lemma \ref{lem: zeta}.
\end{proof}

Let $\kappa : H \to R$ be a  homomorphism of $k$-algebras, let $z \in Z(R)$ and let $\mu \in \hatOm$. Then in \cite{OS22}, \S 2.4.7, Ollivier and Schneider explain how to use these parameters to endow the left $R$-module $R \oplus R$ with the structure of a right $H$-module, in fact forming an $(R,H)$-bimodule, that they denote $(R \oplus R)[\kappa, z, \mu]$.
The left $R$-action on $(R \oplus R)[\kappa, z, \mu]$ is the obvious one, but $H$ acts from the right as follows:
\[ (r_1,r_2) h := (r_1,r_2) \kappa_2(h) \qmb{for all} r_1,r_2 \in R, h \in H.\]
Here $\kappa_2 : H \to M_2(H_\zeta)$ is the explicit $k$-algebra homomorphism from \cite{OS22} \S 2.4.7.  In this generality, we will explicitly compute the right action of $Z$ on this bimodule.
\begin{lemma}\label{lem: k2Xalpha} We have $\kappa_2(X_\alpha) = z^2 \begin{pmatrix} \kappa( e_{\alpha^{-1}\mu} \tau_0\tau_1 ) & 0 \\ 0 & \kappa( e_{\alpha\mu^{-1}} \tau_1\tau_0 ) \end{pmatrix}$ for all $\alpha \in \hatOm$.
\end{lemma}
\begin{proof}Recall from \cite{OS22}\S 2.4.7 that $\kappa_2(\tau_\omega) = M_\omega = \begin{pmatrix} \mu^{-1}(\omega) \kappa(\tau_\omega) & 0 \\ 0 & \mu(\omega)\kappa(\tau_\omega)\end{pmatrix}$. Then
\begin{equation}\label{eq: k2ealpha}\begin{array}{lll} \kappa_2(e_\alpha) &=& - \sum\limits_{\omega \in \Omega} \alpha^{-1}(\omega) \kappa_2(\omega)  = - \sum\limits_{\omega \in \Omega} \alpha^{-1}(\omega) M_\omega \\
&=&\begin{pmatrix}  - \sum\limits_{\omega \in \Omega} \alpha^{-1}(\omega)\mu^{-1}(\omega) \kappa(\tau_\omega)  & 0 \\ 0 & - \sum\limits_{\omega \in \Omega} \alpha^{-1}(\omega) \mu(\omega) \kappa(\tau_\omega) \end{pmatrix} \\
&=& \begin{pmatrix} \kappa(e_{\alpha \mu}) & 0 \\ 0 & \kappa(e_{\alpha\mu^{-1}}) \end{pmatrix} \end{array}\end{equation}
using (\ref{eq: defelambda}). On the other hand, we have
\[ \kappa_2(\tau_0) = M_0 = \begin{pmatrix} -\kappa(e_\mu) & 0 \\ z \kappa(\tau_1) & 0 \end{pmatrix} \qmb{and} \kappa_2(\tau_1) = M_1 = \begin{pmatrix} 0 & z \kappa(\tau_0)  \\0 & - \kappa(e_{\mu^{-1}}) \end{pmatrix}.\]
Then using (\ref{eq: k2ealpha}) together with (\ref{eq: cpi}), we compute
\begin{eqnarray*} \kappa_2(e_\alpha \tau_0\tau_1) &=& \begin{pmatrix}\kappa(e_{\alpha \mu}) & 0 \\0 & \kappa(e_{\alpha\mu^{-1}}) \end{pmatrix} \begin{pmatrix} -\kappa(e_\mu)& 0\\ z \kappa(\tau_1) & 0 \end{pmatrix} \begin{pmatrix}0 & z \kappa(\tau_0)\\0 & -\kappa(e_{\mu^{-1}})\end{pmatrix} \\
&=& \begin{pmatrix} 0& -z \kappa(e_{\alpha \mu} e_\mu \tau_0)\\ 0& z^2 \kappa(e_{\alpha\mu^{-1}} \tau_1\tau_0) \end{pmatrix}= \begin{pmatrix} 0& -\delta_{\alpha\triv}z \kappa(e_\mu \tau_0) \\0 &z^2 \kappa(e_{\alpha\mu^{-1}}\tau_1\tau_0) \end{pmatrix}.\end{eqnarray*}
 Similarly,
\begin{eqnarray*} \kappa_2(e_{\alpha^{-1}} \tau_1\tau_0) &=& \begin{pmatrix}\kappa(e_{\alpha^{-1} \mu}) & 0 \\0 & \kappa(e_{\alpha^{-1}\mu^{-1}}) \end{pmatrix} \begin{pmatrix}0 & z \kappa(\tau_0)\\0 & -\kappa(e_{\mu^{-1}})\end{pmatrix} \begin{pmatrix} -\kappa(e_\mu)& 0\\ z \kappa(\tau_1) & 0 \end{pmatrix}\\
&=& \begin{pmatrix} z^2 \kappa(e_{\alpha^{-1} \mu} \tau_0\tau_1)& 0 \\- z\kappa(e_{\alpha^{-1}\mu^{-1}} e_{\mu^{-1}} \tau_1) & 0 \end{pmatrix}= \begin{pmatrix} z^2\kappa(e_{\alpha^{-1} \mu} \tau_0\tau_1)& 0 \\ -z \delta_{\alpha\triv} \kappa(e_{\mu^{-1}}\tau_1) &0 \end{pmatrix}.\end{eqnarray*}
Suppose first that $\alpha \neq \triv$ so that $X_\alpha = e_\alpha \tau_0\tau_1 + e_{\alpha^{-1}} \tau_1\tau_0$. Then $\delta_{\alpha\triv} = 0$, and adding the last two displayed equations together gives the result in this case.

Suppose now that $\alpha = \triv$ so that $X_\triv = e_\triv \zeta = e_\triv \tau_0\tau_1 + e_\triv \tau_1\tau_0 + e_\triv (\tau_0 + \tau_1 + 1)$. Adding the above two equations in this case gives
\begin{equation}\label{eq: exyeyx} \kappa_2(e_\triv \tau_0\tau_1 + e_\triv \tau_1\tau_0) = z^2 \begin{pmatrix} \kappa( e_{\mu} \tau_0\tau_1 ) & 0 \\ 0 & \kappa( e_{\mu^{-1}} \tau_1 \tau_0 ) \end{pmatrix} + \begin{pmatrix} 0 & -z \kappa(e_\mu \tau_0) \\ -z \kappa(e_{\mu^{-1}}\tau_1) & 0 \end{pmatrix}.\end{equation}
Since $\kappa_2(e_\triv) = \begin{pmatrix} \kappa(e_\mu) & 0 \\ 0 & \kappa(e_{\mu^{-1}}) \end{pmatrix}$ and $\kappa_2(\tau_0+\tau_1+1) = \begin{pmatrix} \kappa(1-e_\mu) & z \kappa(\tau_0) \\ z \kappa(\tau_1) & \kappa(1-e_{\mu^{-1}}) \end{pmatrix}$, we have
\begin{equation}\label{eq: e1x+y+1} \kappa_2( e_\triv (\tau_0+\tau_1+1)) = \begin{pmatrix} 0 & z \kappa(e_\mu \tau_0) \\ z \kappa(e_{\mu^{-1}}\tau_1) & 0 \end{pmatrix}.\end{equation}
Adding together $(\ref{eq: exyeyx})$ and $(\ref{eq: e1x+y+1})$ finishes the calculation.
\end{proof}

\begin{corollary}\label{cor: kappa2zeta} $\kappa_2(\zeta) = z^2 \begin{pmatrix} \kappa(\tau_0\tau_1) & 0 \\ 0 & \kappa(\tau_1\tau_0) \end{pmatrix}$.\end{corollary}
\begin{proof} Note that $\sum\limits_{\alpha \in \hatOm} e_{\alpha^{-1}\mu} = \sum\limits_{\lambda \in \hatOm} e_\lambda = 1$. Now apply Lemma \ref{lem: k2Xalpha} and Lemma \ref{lem: ezeta}.d.
\end{proof}
Now we specialise slightly, and take $R$ to be the central localization $H_\zeta$ of $H$ at $\zeta$. Let $\kappa$ be the composition of $\iota : H \to H$ with the inclusion $H \hookrightarrow H_\zeta$ and set
\[z := -\tau_{\omega_{-1}} \zeta^{-1} \in H_\zeta \qmb{and} \mu := \id^2.\]
\begin{definition}\label{def: bimodB} Let $B := \left(\frac{H_\zeta}{H_\zeta \tau_0} \oplus \frac{H_\zeta}{H_\zeta \tau_1}\right)[\kappa, z, \mu]$, an $(H_\zeta,H)$-bimodule. \end{definition}
We will now calculate the right action of $Z$ on this bimodule.
\begin{proposition}\label{prop: epstwist} \hsp\begin{itemize}
\item[a)] We have $v \zeta = \zeta^{-1} v$ for all $v \in B$.
\item[b)] Let $\mathbf{1} := (\overline{1},\overline{1})\in B$. Then $\mathbf{1}   \eps_\alpha = \eps_{\alpha^{-1} \mu} \mathbf{1}$ for all $\alpha \in \hatOm$.
\end{itemize}
\end{proposition}
\begin{proof} a) We have $z = -\tau_{\omega_{-1}}\zeta^{-1}$ so $z^2 = \zeta^{-2}$ because $\zeta$ is central and $(\omega_{-1})^2 = \omega_1 = 1$. Note that $\kappa(\tau_0\tau_1) = \iota(\tau_0\tau_1) = (\tau_0 + e_\triv)(\tau_1 + e_\triv) = \zeta - \tau_1\tau_0 \equiv \zeta \mod H\tau_0$, and similarly $\kappa(\tau_1\tau_0) = \zeta - \tau_0\tau_1 \equiv \zeta \mod H\tau_1$ in view of Lemma \ref{lem: zeta}. Writing $v = (\overline{v_1},\overline{v_2})$, we apply Cor. \ref{cor: kappa2zeta}:
\[ (\overline{v_1},\overline{v_2})\zeta = (\overline{v_1},\overline{v_2})\kappa_2(\zeta) = (\overline{v_1},\overline{v_2}) \zeta^{-2} \begin{pmatrix} \zeta - \tau_1\tau_0 & 0 \\ 0 & \zeta - \tau_0\tau_1 \end{pmatrix} = (\overline{v_1\zeta^{-1}}, \overline{v_2\zeta^{-1}}) = \zeta^{-1}(\overline{v_1},\overline{v_2}).\]
b) Recall that $\eps_\alpha = X_\alpha/\zeta$ by Lemma \ref{lem: epsilonlambda}. We will show that
\begin{equation}\label{eq: zeta211Xalpha} \zeta^2 \mathbf{1}   X_\alpha = X_{\alpha^{-1} \mu} \mathbf{1}.\end{equation}
Applying Lemma \ref{lem: k2Xalpha}, we calculate
\begin{equation}\label{eq: 11Xalpha} (\overline{1},\overline{1})\kappa_2(X_\alpha)= (\overline{1},\overline{1})z^2 \begin{pmatrix} \kappa( e_{\alpha^{-1}\mu} \tau_0\tau_1 ) & 0 \\ 0 & \kappa( e_{\alpha\mu^{-1}} \tau_1 \tau_0 ) \end{pmatrix}=\zeta^{-2} \left(\overline{e_{\alpha^{-1}\mu} \iota(\tau_0\tau_1)}, \overline{e_{\alpha\mu^{-1}} \iota(\tau_1\tau_0)}\right).\end{equation}
Next,
\begin{equation}\label{eq: iotaxy} \begin{array}{lll} \iota(\tau_0\tau_1) &=& (-\tau_0-e_\triv)(-\tau_1-e_\triv) = \tau_0\tau_1 + e_\triv(\tau_0 + \tau_1 + 1), \qmb{and}\\\iota(\tau_1\tau_0) &=& (-\tau_1-e_\triv)(-\tau_0-e_\triv) = \tau_1\tau_0 + e_\triv(\tau_0 + \tau_1 + 1).\end{array}\end{equation}
Suppose first that $\alpha \neq \mu$. Then since $e_{\alpha^{-1}\mu}e_\triv = e_{\alpha \mu^{-1}}e_\triv = 0$ by (\ref{eq: cpi}),
\[\mathbf{1}   X_\alpha = (\overline{1},\overline{1})\kappa_2(X_\alpha)=  \zeta^{-2}\left( \overline{e_{\alpha^{-1}\mu} \tau_0\tau_1}, \overline{e_{\alpha\mu^{-1}} \tau_1\tau_0}\right).\]
On the other hand, using Def. \ref{def: Xlambda}, we have the congruences
\[X_{\alpha^{-1} \mu} = e_{\alpha^{-1}\mu} \tau_0\tau_1 + e_{\alpha\mu^{-1}} \tau_1\tau_0 \equiv \left\{\begin{array}{ll} e_{\alpha^{-1}\mu} \tau_0\tau_1 & \mod H\tau_0\\ e_{\alpha\mu^{-1}} \tau_1\tau_0 & \mod H\tau_1. \end{array} \right.\]
Putting these together shows that
\[X_{\alpha^{-1} \mu} \mathbf{1} = \left( \overline{ X_{\alpha^{-1} \mu}}, \overline{X_{\alpha^{-1} \mu} }\right) = \left( \overline{e_{\alpha^{-1}\mu} \tau_0\tau_1}, \overline{e_{\alpha\mu^{-1}} \tau_1\tau_0}\right) = \zeta^2 \mathbf{1} X_\alpha\]
as required. Suppose now that $\alpha = \mu$. Then applying (\ref{eq: 11Xalpha}) and (\ref{eq: iotaxy}) again, we have
\[ \zeta^2\mathbf{1}  X_\mu = (\overline{e_\triv\iota(\tau_0\tau_1)}, \overline{e_\triv\iota(\tau_1\tau_0)}) = (\overline{e_\triv \tau_0\tau_1 + e_\triv \tau_1 + e_\triv}, \overline{e_\triv \tau_1\tau_0 + e_\triv \tau_0 + e_\triv}),\]
whereas since $X_{\alpha^{-1}\mu} = X_\triv = e_\triv \zeta = e_\triv(\tau_0\tau_1 + \tau_1\tau_0 + \tau_0 + \tau_1 + 1)$ in this case, we see that
\[ X_\triv\mathbf{1} = (\overline{X_\triv}, \overline{X_\triv}) = (\overline{e_\triv \tau_0\tau_1 + e_\triv \tau_1 + e_\triv}, \overline{e_\triv \tau_1\tau_0 + e_\triv \tau_0 + e_\triv}).\]
Having shown that (\ref{eq: zeta211Xalpha}) holds, we use part a) to obtain
\[\mathbf{1}   \eps_\alpha = \left(\mathbf{1}   \zeta^{-1} \right)  X_\alpha= \zeta\mathbf{1}   X_\alpha = \zeta^{-1}\left(\zeta^2 \mathbf{1}   X_\alpha\right) = \zeta^{-1} X_{\alpha^{-1}\mu} \mathbf{1} = \eps_{\alpha^{-1}\mu} \mathbf{1}. \qedhere \]
\end{proof}

\subsubsection{The annihilator of $B$ in $Z_\zeta \otimes Z_\zeta$}
Recall from Cor. \ref{cor: Zzeta} that $Z_\zeta$ is isomorphic to $k[\Omega][t,t^{-1}]$. Noting that $k[\Omega]$ is isomorphic to the direct product of $|\hatOm| = q-1$ copies of the ground field $k$, the following statement is clear:

\begin{lemma}  There is a unique $k$-algebra automorphism $\phi : Z_\zeta \to Z_\zeta$ such that
\[ \phi(\zeta) = \zeta^{-1} \qmb{and} \phi(e_\alpha) = e_{\mu/\alpha} \qmb{for all} \alpha \in \hatOm.\]
\end{lemma}
Note that this $\phi$ is in fact an involution.

\begin{proposition}\label{lem: TwistedSubBimodule} The map $\psi : (H_\zeta)_\phi \to B$ given by $\psi(h) = h \mathbf{1}$ for all $h \in H_\zeta$ is an injective homomorphism of $(H_\zeta,Z_\zeta)$-bimodules.
\end{proposition}
\begin{proof} The map $\psi$ is clearly left $H_\zeta$-linear. Using Prop. \ref{prop: epstwist}, we have
\[ \psi(h) \zeta = h \mathbf{1} \zeta = h \zeta^{-1} \mathbf{1} = \psi( h\phi(\zeta) ) \qmb{and} \psi(h) \eps_\alpha = h \mathbf{1} \eps_\alpha = h \eps_{\mu/\alpha} \mathbf{1} = \psi(h \phi(\eps_\alpha))\]
for any $\alpha \in \hatOm$. After Cor. \ref{cor: Zzeta}, we know that $Z_\zeta$ is generated as a $k$-algebra by the $\eps_\alpha$'s, $\zeta$ and $\zeta^{-1}$. It follows that $\psi$ is right $Z_\zeta$-linear.

We have $\ker \psi = H_\zeta \tau_0 \cap H_\zeta \tau_1$ by definition of $B$. This is equal to $(H \tau_0 \cap H \tau_1)_\zeta$. But $H \tau_0 \cap H \tau_1 = 0$ by Lemma \ref{lem: HxHy} below.
\end{proof}
\begin{lemma}\label{lem: HxHy} The sum of left ideals $H\tau_0 + H\tau_1$ is direct.\end{lemma}
\begin{proof} Recall that $\{\tau_w : w \in \widetilde{W}\}$  forms a basis for $H$ as a $k$-vector space. Now, $\widetilde{W}$ contains $\Omega$ as a normal subgroup and the images of $s_0$ and $s_1$ in $\widetilde{W} / \Omega$ generate a copy of the infinite dihedral group. It follows that every element of $\widetilde{W}$ has a unique representation in the form $\omega w$ where $\omega \in \Omega$ and $w$ is a word in $s_0$ and $s_1$. Therefore the words in $\tau_0$ and $\tau_1$ form a basis for $H$ as a $k[\Omega]$-module. Every non-empty such word ends in exactly one of either $\tau_0$ or $\tau_1$, which gives the result.\end{proof}
Now we study $\coker \psi$.
\begin{lemma}\label{lem: HmodF1H} The left $H$-module $H / (H\tau_0 + H\tau_1)$ is killed by $\zeta(\zeta - 1)$.
\end{lemma}
\begin{proof} The fact that $\{\tau_w : w \in \widetilde{W}\}$ spans $H$ implies that $H = k[\Omega] \oplus (H\tau_0 + H\tau_1)$ as a $k$-vector space. Therefore it is enough to show that $\zeta(\zeta-1)e_\lambda \in H\tau_0 + H\tau_1$ for all $\lambda \in \hatOm$. However $\zeta = \tau_0\tau_1 + \tau_1\tau_0 + e_\triv \tau_1 + e_\triv \tau_0 + e_\triv \equiv e_\triv \mod H\tau_0 + H\tau_1$, so already $\zeta e_\lambda \equiv 0 \mod H\tau_0 + H \tau_1$ if $\lambda \neq \triv$, whereas if $\lambda = \triv$ then $\zeta(\zeta-1) e_\triv =\zeta(\zeta e_\triv - e_\triv) \equiv \zeta( e_\triv^2 - e_\triv) = 0 \mod H\tau_0 + H\tau_1$.
\end{proof}

\begin{corollary}\label{cor: uv} There exist elements $u,v \in H$ such that $\zeta(\zeta - 1) = u\tau_0 + v\tau_1$.\end{corollary}

\begin{proposition}\label{prop: CRT} We have $(\zeta - 1) \cdot \coker \psi = 0$.
\end{proposition}
\begin{proof} Let $(\overline{a}, \overline{b}) \in B$ for some $a,b \in H_\zeta$. Define $c:= av\tau_1 + bu\tau_0$, where $u,v \in H$ given by Cor. \ref{cor: uv}. Then
\[ c = a\left(\zeta(\zeta-1) - u\tau_0\right) + bu\tau_0 \equiv a \zeta(\zeta-1) \mod H_\zeta \tau_0\]
and similarly $c \equiv b \zeta(\zeta -1) \mod H_\zeta \tau_1$. Hence $\zeta(\zeta - 1)(\overline{a}, \overline{b}) = (\overline{c}, \overline{c}) = c\mathbf{1} = \psi(c)$, so $\coker \psi$ is killed by $\zeta(\zeta-1)$. Since $\zeta$ is invertible in $H_\zeta$, $\coker \psi$ is already killed by $\zeta - 1$. \end{proof}

\begin{lemma}\label{lem: HHxitf} $H/H \tau_0$ and $H/H\tau_1$ are $k[\zeta]$-torsionfree.
\end{lemma}
\begin{proof} Let $H_0$ be the $k$-subalgebra of $H$ generated by $k[\Omega]$ and $\tau_0$. We have
\[ H_0 = k[\Omega] \oplus k[\Omega] \tau_0 = k[\Omega] \oplus \tau_0 k[\Omega].\]
Next, by \cite{OS18} Cor. 3.4, $H$ is free of rank $2$ as a $H_0 \otimes_k k[\zeta]$-module, with basis $\{1, \tau_1\}$. If $H_0[\zeta]$ denotes the subalgebra of $H$ generated by $H_0$ and $\zeta$, then this means that the multiplication map $H_0 \otimes_k k[\zeta] \to H_0[\zeta]$ is a $k$-algebra isomorphism, and also that $H$ is free of rank $2$ as a left $H_0[\zeta]$-module, with basis $\{1, \tau_1\}$ :
\[ H = H_0[\zeta] \oplus H_0[\zeta] \tau_1.\]
Let $A := k[\Omega][\zeta] \subset H_0[\zeta]$; clearly $A$ is isomorphic as a $k$-algebra to the polynomial ring in one variable over $k[\Omega]$. Combining the two displayed equations above, we obtain the decomposition
\[ H = A \oplus A \tau_0 \oplus A \tau_1 \oplus A \tau_0\tau_1\]
of $H$ as a left $A$-module. Now $\tau_1^2 = - e_\triv \tau_1$ by Lemma \ref{lem: Hpres}, and $e_\triv$ is central in $H$. This implies that right-multiplication by $\tau_1$ sends $H$ into $A \tau_1 \oplus A \tau_0 \tau_1$. Since this $A$-module is clearly contained in $H \tau_1$, we conclude that $H \tau_1 = A \tau_1 \oplus A \tau_0 \tau_1$, and therefore
\[ H / H \tau_1 \cong A \oplus A \tau_0\]
as a left $A$-module. Since $A$ is itself free as a $k[\zeta]$-module, we see that $H / H \tau_1$ is free of rank $2(q-1)$ as a $k[\zeta]$-module. It is in particular $k[\zeta]$-torsionfree.

The argument for $H / H \tau_0$ is the same, after switching $\tau_0$ with $\tau_1$.
\end{proof}

\begin{corollary}\label{cor: Bk[zeta]tf} The left $H$-module $B$ is $k[\zeta]$-torsionfree. \end{corollary}
\begin{proof} As a left $k[\zeta]$-module, $B$ is isomorphic to $k[\zeta, \zeta^{-1}] \otimes_{k[\zeta]} \left(\frac{H}{H \tau_0}\oplus \frac{H}{ H \tau_1}\right)$. The expression in the brackets is $k[\zeta]$-torsionfree by Lemma \ref{lem: HHxitf}, and the result follows easily.
\end{proof}

\begin{theorem}\label{thm: eauBea}Let $R = Z_\zeta \otimes Z_\zeta$. Then we have $\Ann_R(B) = (1 \otimes \phi)(\ker \mult_{Z_\zeta})$.\end{theorem}
\begin{proof} Using Prop. \ref{lem: TwistedSubBimodule}, we see that
\[\Ann_R(B) \subseteq \Ann_R( (H_\zeta)_\phi).\]
Suppose that $r \in R$ kills $(H_\zeta)_\phi$; then $r$ kills $\Im(\psi)$. By Prop. \ref{prop: CRT}, $(\zeta - 1)B \subseteq \Im(\psi)$. Hence $r \cdot \left( (\zeta - 1) \otimes 1 \right)$ kills $B$, so $\zeta - 1$ kills $r \cdot B$. But $B$ is $k[\zeta]$-torsionfree by Cor. \ref{cor: Bk[zeta]tf}, so $r \cdot B = 0$. This shows that
\[\Ann_R(B) = \Ann_R( (H_\zeta)_\phi).\]
Now we apply Lemma \ref{lem: TwistedAnn} to the $(Z_\zeta,Z_\zeta)$-bimodule $H_\zeta$ to see that
\[ \Ann_R((H_\zeta)_\phi) = (1 \otimes \phi^{-1})(\Ann_R(H_\zeta)).\]
Now $Z_\zeta$ is the centre of $H_\zeta$, so $\Ann_R(H_\zeta) = \Ann_R(Z_\zeta)$ by Lemma \ref{lem: HZbimodann}. Since $\phi = \phi^{-1}$,
\[ \Ann_R(B) = \Ann_R( (H_\zeta)_\phi) = (1 \otimes \phi)(\Ann_R(Z_\zeta)).\]
The result now follows from Lemma \ref{lem: HZbimodann} applied with $A = Z_\zeta$.
\end{proof}

Write $\delta(a) := a\otimes 1 - 1 \otimes \phi(a)$ for all $a \in Z_\zeta$. The following \emph{finite} set of ideal generators for $\Ann_R(B)$ will be useful to us later.

\begin{corollary} \label{cor: AnnRBgens} We have $\Ann_R(B) = R\delta(\zeta) + \sum\limits_{\alpha \in \hatOm} R \delta(\eps_\alpha)$.
\end{corollary}
\begin{proof} By Cor. \ref{cor: Zzeta}, there is a $k$-algebra isomorphism $Z_\zeta \stackrel{\cong}{\longrightarrow}  k[\Omega][t, t^{-1}]$ that sends $\eps_\alpha$ to $e_\alpha$ and $\zeta$ to $t$. Therefore $\{\eps_\alpha : \alpha \in \hatOm\} \cup \{\zeta, \zeta^{-1}\}$ forms a set of $k$-algebra generators for $Z_\zeta$. Writing $\partial(a) = a \otimes 1 - 1 \otimes a$ for all $a \in Z_\zeta$, it follows from Lemma \ref{lem: BiAnnA}.b that
\begin{equation}\label{eq: kermultZzeta} \ker \mult_{Z_\zeta} = R \partial(\zeta) + R \partial(\zeta^{-1}) + \sum\limits_{\alpha \in \hatOm} R \partial(\eps_\alpha).\end{equation}
Note that $\partial(\zeta^{-1}) = \zeta^{-1} \otimes 1 + 1 \otimes \zeta^{-1} = (\zeta^{-1} \otimes \zeta^{-1})\partial(\zeta)$, so $R \partial(\zeta^{-1}) = R \partial(\zeta)$. Also, note that $(1 \otimes \phi) \partial = \delta$. Therefore, applying $1 \otimes \phi$ to (\ref{eq: kermultZzeta}) gives
\[ (1 \otimes \phi)\left(\ker \mult_{Z_\zeta}\right) = R \delta(\zeta) + \sum\limits_{\alpha \in \hatOm} R \delta(\eps_\alpha).\]
The result now follows from Thm. \ref{thm: eauBea}.
\end{proof}


\subsection{The bimodule annihilator of $\Ext^\ast_G(X,X)$} \label{sec: BiModAnnExt}
Our goal will be to compute the ideal
\[ \Ann_{Z \otimes_kZ} E^\ast, \qmb{where} E^\ast =  \bigoplus_{n\geq 0} E^n, \qmb{and} E^n = \Ext^n_G(X,X).\]
For brevity, we will write for each $n \geq 0$
\[ J_n := \Ann_{Z \otimes_kZ} \left(E^n\right).\]
First, we will use the Poincar\'e duality isomorphisms from \cite{OS19} to relate $J_n$ with $J_{d-n}$, where $d$ is the dimension of $G$ as a $p$-adic Lie group. For this, we need to look at the anti-involution $\anti : E^\ast \to E^\ast$ that was introduced in \cite{OS19}, $\S 6$: it is an anti-automorphism of the graded algebra $E^\ast$ by \cite{OS19}, Prop. 6.1, and therefore in particular restricts to a $k$-algebra anti-automorphism of $H = E^0$.
\begin{lemma} \label{lem: antiInvResZ} $\anti$ fixes $Z$ pointwise.
\end{lemma}
\begin{proof} By \cite{OS19} Prop. 6.1, we know that $\anti(\tau_g) = \tau_{g^{-1}}$ for all $g \in G$, where $\tau_g$ denotes the characteristic function of the double coset $IgI$ --- see \cite{OS19}, $\S$2.2. It follows therefore that $\anti$ restricts to the standard anti-involution on the group algebra $k[\Omega] \subset H$ that sends the group elements $\tau_\omega$ to $\tau_{\omega^{-1}}$, for each $\omega \in \Omega$. Equation (\ref{eq: defelambda}) now implies that
\[ \anti(e_\lambda) = e_{\lambda^{-1}} \qmb{for all} \lambda \in \hatOm.\]
Next, we have $s_i^{-1} = -s_i$ for $i=0,1$ inside $G$. Writing $\eps = \tau_{\omega_{-1}}$, it follows that
\[ \anti(\tau_i) = \tau_i \eps, \quad i=0,1.\]
With these two pieces of information in hand, we can now compute the effect of $\anti$ on the algebra generators $X_\lambda$ of $Z$ from Def. \ref{def: Xlambda}, for each $\lambda \in \hatOm$. When $\lambda \neq \triv$, we have
\begin{eqnarray*} \anti(X_\lambda) &=& \anti(e_\lambda \tau_0 \tau_1 + e_{\lambda^{-1}}\tau_1\tau_0)  \\
&=& \anti(\tau_1) \anti(\tau_0) \anti(e_\lambda) + \anti(\tau_0)\anti(\tau_1)\anti(e_{\lambda^{-1}}) \\
&=& \tau_1\eps \hsp \tau_0 \eps \hsp e_{\lambda^{-1}} + \tau_0 \eps \hsp \tau_1 \eps \hsp e_\lambda \\
&=& \tau_1 \tau_0 e_{\lambda^{-1}} + \tau_0 \tau_1 e_\lambda\\
&=& e_{\lambda^{-1}} \tau_1 \tau_0 + e_\lambda \tau_0 \tau_1 \\
&=& X_\lambda. \end{eqnarray*}
We have used here the fact that $\eps$ is a central involution in $H$, as well as the relations
\[ \tau_i e_\lambda = e_{\lambda^{-1}} \tau_i, \quad i=0,1\]
from Lemma \ref{lem: Hpres}. Next, recall from Lemma \ref{lem: zeta} that $\zeta = (\tau_0 + e_\triv)(\tau_1+e_\triv) + \tau_1\tau_0$. When $\lambda = \triv$, the calculation goes as follows:
\begin{eqnarray*} \anti(X_\triv) &=& \anti( e_\triv\zeta ) = \anti(\zeta) e_\triv \\
&=& ( (\tau_0 \eps) (\tau_1 \eps) + (\tau_1 \eps + e_\triv) (\tau_0 \eps + e_\triv)) e_\triv \\
&=& (\tau_0\tau_1 + \tau_1 \tau_0 + e_\triv \tau_0 \eps + \tau_1 e_\triv \eps + e_\triv) e_\triv \\
&=& \zeta e_\triv = X_\triv.
\end{eqnarray*}
We have used here the fact that $e_\triv \eps = \triv(\eps) e_\triv = e_\triv$. We have shown that $\anti$ fixes each $X_\lambda$. Since these generate $Z$ as a $k$-algebra in view of Prop. \ref{prop: Zcomps}, $\anti|_Z$ is the identity map.\end{proof}

For our next result, we need to introduce the $k$-algebra involution  $\sigma:  Z \otimes_k Z \to Z \otimes_k Z$, given by $\sigma(z_1 \otimes z_2) = z_2 \otimes z_1$ for $z_1,z_2 \in Z$.

\begin{lemma} \label{lem: JdPoinc} For each $n = 0,\cdots, d$, we have $J_{d-n} = \sigma(J_n)$.
\end{lemma}
\begin{proof} For a $k$-vector space $V$, let $V^\vee$ denote the full $k$-linear dual $V^\vee = \Hom_k(V,k)$. When $V$ happens to be an $(H,H)$-bimodule, then $V^\vee$ is also an $(H,H)$-bimodule. Following \cite{OS19}, $\S$7.1, this bimodule structure is given by the following rule:
\[ (a \cdot f \cdot b)(v) = f(b \hsp v \hsp a) \qmb{for all} v \in V, a,b \in H, f \in V^\vee.\]
We are only interested in $(Z,Z)$-bimodules here. Regarding every $(Z,Z)$-bimodule $V$ as a left $A := Z \otimes_kZ$-module in the standard way, we can write this action of $A$ on $V^\vee$ as follows:
\[ (x \cdot f)(v) = f(\sigma(x)v) \qmb{for all} v \in V, x \in A, f \in V^\vee.\]
It follows immediately that for every $(Z,Z)$-bimodule $V$ we have
\begin{equation}\label{eq: AnnVdual} \sigma^{-1}(\Ann_AV) \subseteq \Ann_A(V^\vee).\end{equation}
Now, \cite{OS19} Prop. 7.18 gives us an injective homomorphism of $(H,H)$-bimodules
\[ \Delta^n : E^n \hookrightarrow ({}^{\anti} (E^{d-n})^\anti)^\vee.\]
Since $\anti$ fixes $Z$ pointwise by Lemma \ref{lem: antiInvResZ}, we may write this as
\[\Delta^n : E^n \hookrightarrow (E^{d-n})^\vee,\]
an injective homomorphism of $(Z,Z)$-bimodules. Using (\ref{eq: AnnVdual}), we obtain
\[ J_{d-n} = \Ann_A(E^{d-n}) \subseteq \sigma(\Ann_A (E^{d-n})^\vee) \subseteq \sigma(\Ann_A(E^n)) = \sigma(J_n).\]
Replacing $n$ with $d-n$ gives $J_n \subseteq \sigma(J_{d-n})$, so $J_{d-n} \subseteq \sigma(J_n) \subseteq \sigma^2(J_{d-n}) = J_{d-n}$. \end{proof}

Recall the multiplication map $\mult_Z : Z \otimes_k Z \to Z$ from $\S \ref{sec: GenBimod}$.

\begin{lemma} \label{lem: AnnE0}\hsp
\begin{itemize} \item[a)] We have $J_0 = \ker \mult_Z$.
\item[b)]  $J_0  = \sigma(J_0)$.
\end{itemize}
\end{lemma}
\begin{proof} Since $\Ext^0_G(X,X) = \Hom_G(X,X) = H$ and since $Z = Z(H)$,  Lemma \ref{lem: HZbimodann} implies that $J_0 = \ker \mult_Z$. The second statement follows from Lemma \ref{lem: BiAnnA}.a.
\end{proof}

\textbf{Until the end of $\S \ref{sec: BiModAnnExt}$, we assume that $\frF = \bQ_p$, $p \neq 2,3$ and $\pi = p$.} This assumption allows us to apply the results of \cite{OS22} to study $J_1$. Recall first that by  \cite{OS22} Prop. 6.10(1), there is a certain short exact sequence of $(H,H)$-bimodules
\begin{equation}\label{eq: E^1bimodses} 0 \to \ker(f_1) \oplus \ker(g_1) \to E^1 \to C \to 0.\end{equation}
Here $C$ is a certain $(H,H)$-bimodule with $\dim_k C = 4$.
\begin{lemma} \label{lem: Annkerg1} We have $\Ann_{Z \otimes_k Z}(\ker g_1)) = J_0$.
\end{lemma}
\begin{proof} By \cite{OS22} Prop. 6.3, $\ker(g_1) \cong F^1H$ as an $(H,H)$-bimodule, where $(F^nH)_{n \geq 0}$ is the descending filtration on $H$ defined at \cite{OS22}, \S 2.2.4. Writing $A = Z\otimes_kZ$, we then have
\[ \Ann_A(\ker(g_1)) = \Ann_A(F^1H) \supseteq \Ann_A(H).\]
For the reverse inclusion, note that $F^1H = H \tau_0 + H \tau_1$, so $\zeta(\zeta - 1)$ kills $H/F^1H$ from the left by Lemma \ref{lem: HmodF1H}. Since $H$ is $k[\zeta]$-torsionfree and since $\zeta$ is central in $H$, left-multiplication by $\zeta(\zeta-1)$ gives an injective homomorphism of $(H,H)$-bimodules $H \hookrightarrow F^1H$. Hence
\[ \Ann_A(\ker(g_1)) = \Ann_A(F^1H) = \Ann_A(H) = \Ann_A(E^0) = J_0. \qedhere\]
\end{proof}
On p. 38 of \cite{OS22}, Ollivier and Schneider consider a certain twist of $(H_\zeta \oplus H_\zeta)[\kappa, z, \mu]$ that they denote $(H_\zeta \oplus H_\zeta)^\pm$. The left action of $H$ on $(H_\zeta \oplus H_\zeta)$ is unchanged, but the right $H$-action is twisted by pre-composing the previous right $H$-action by $\iota$. They show that the left $H_\zeta$-submodule $H_\zeta \tau_0 \oplus H_\zeta \tau_1$ of this direct sum is stable under this new right $H$-action, which allows them to pass to the quotient to form the $(H_\zeta,H)$-bimodule
\[ \left(\frac{H_\zeta}{H_\zeta \tau_0} \oplus \frac{H_\zeta}{H_\zeta \tau_1}\right)^\pm = \frac{(H_\zeta  \oplus H_\zeta )^\pm}{(H_\zeta \tau_0 \oplus H_\zeta \tau_1)^{\pm}}.\]
Then they prove the following Theorem --- see \cite{OS22} Prop. 3.28 and Thm. 6.8:
\begin{theorem}\label{thm: MainOS21} There is an $(H_\zeta,H)$-bimodule isomorphism
\[ \left(\frac{H_\zeta}{H_\zeta \tau_0} \oplus \frac{H_\zeta}{H_\zeta \tau_1}\right)^\pm \stackrel{\cong}{\longrightarrow} \ker(f_1)\]
\end{theorem}
We are only interested in the $(Z_\zeta,Z_\zeta)$-bimodule structure on $\ker(f_1)$. Recall the $(H_\zeta,H)$-bimodule $B$ from Def. \ref{def: bimodB}.
\begin{corollary}\label{cor: rightZaction} There is an $(H_\zeta,Z_\zeta)$-bimodule isomorphism $B \stackrel{\cong}{\longrightarrow} \ker(f_1)$.
\end{corollary}
\begin{proof} Note that $\zeta$ acts invertibly on $\ker(f_1)$ from both sides by its definition: indeed the left $\zeta$-actions and the right $\zeta$-actions are mutually inverse. Therefore $\ker(f_1)$ is in fact a $(H_\zeta,H_\zeta)$-bimodule. Since $\iota$ fixes $Z$ by Lemma \ref{lem: iotafixesZ}, precomposing with $\iota$ makes no difference to the right action of $Z$. Now we can apply Thm. \ref{thm: MainOS21}.
\end{proof}

\begin{theorem} \label{thm: AnnE1}Suppose that $\frF = \bQ_p$, $p \neq 2,3$ and $\pi = p$.
\begin{itemize} \item[a)] We have $J_1 = J_0 \cap \Ann_{Z \otimes_k Z}(B)$.
\item[b)]  $J_1 = \sigma(J_1)$.
\end{itemize}
\end{theorem}
\begin{proof} a) Again, write $A = Z \otimes_k Z$. Using (\ref{eq: E^1bimodses}), we see that
\[\Ann_A(E^1) \subseteq \Ann_A(\ker (f_1) \oplus \ker (g_1)).\]
For the reverse inclusion, suppose that $x \in \Ann_A(\ker (f_1) \oplus \ker (g_1))$. Since the $(H,H)$-bimodule $C$ appearing in (\ref{eq: E^1bimodses}) satisfies $\dim_k C < \infty$, there is a non-zero $g(\zeta) \in k[\zeta]$ such that $g(\zeta)E^1 \subseteq \ker (f_1) \oplus \ker (g_1)$. Hence $x \cdot (g(\zeta)E^1) = 0$, so $g(\zeta)(x \cdot E^1) = 0$. Since $E^1$ is $k[\zeta]$-torsionfree by \cite{OS22} Lemma 5.1, this forces $x \cdot E^1 = 0$. Hence $x \in \Ann_A(E^1)$ and
\[\Ann_A(E^1) = \Ann_A(\ker (f_1) \oplus \ker (g_1)) = \Ann_A(\ker(f_1)) \hsp \cap \hsp \Ann_A(\ker(g_1)).\]
Using Cor. \ref{cor: rightZaction} and Lemma \ref{lem: Annkerg1}, we see that
\begin{equation}\label{eq: J1decomp} J_1 = \Ann_A(E^1) = \Ann_A(B) \cap J_0.\end{equation}
b) We know that $B$ is in fact a $(Z_\zeta,Z_\zeta)$-bimodule, so that
\[\Ann_A(B) = A \cap \Ann_{Z_\zeta \otimes_k Z_\zeta}(B).\]
In view of formula (\ref{eq: J1decomp}) and Lemma \ref{lem: AnnE0}.b, it is enough to show that the obvious extension of $\sigma$ to $Z_\zeta \otimes Z_\zeta$ preserves $\Ann_{Z_\zeta \otimes_k Z_\zeta}(B)$.

Recall from Thm. \ref{thm: eauBea} that this ideal is generated by $\{\delta(a) : a \in Z_\zeta\}$, where we define $\delta(a) := a \otimes 1 - 1 \otimes \phi(a)$ for each $a \in A$.  Since $\phi^2(a) = a$ for all $a \in Z_\zeta$, we calculate
\[\sigma(\delta(a)) = \sigma\left(a \otimes 1 - 1 \otimes \phi(a)\right) = -(\phi(a) \otimes 1 - 1 \otimes a) = - \delta(\phi(a)).\]
Therefore $\sigma$ preserves $\Ann_{Z_\zeta \otimes_k Z_\zeta}(B)$ as required.
\end{proof}
Putting everything together, we obtain

\begin{corollary} \label{cor: AnnEstar}  Suppose that $\frF = \bQ_p$, $p \neq 2,3$ and $\pi = p$. Then
\[\Ann_{Z \otimes_kZ}(E^\ast) \hsp = \hsp \ker \mult_Z \hsp \cap \hsp \Ann_{Z \otimes_kZ}(B).\]
\end{corollary}
\begin{proof} Here $d = 3$, so $\Ann_{Z \otimes_kZ}(E^\ast) = J_0 \cap J_1 \cap J_2 \cap J_3$. However $J_2 = \sigma(J_1)$ and $J_3 = \sigma(J_0)$ by Lemma \ref{lem: JdPoinc}. Using Lemma \ref{lem: AnnE0}.b and Thm. \ref{thm: AnnE1}.b, we obtain
\[\Ann_{Z \otimes_kZ}(E^\ast) = J_0 \cap J_1.\]
Now use Thm. \ref{thm: AnnE1}.a and Lemma \ref{lem: AnnE0}.a  to conclude. \end{proof}

\subsection{The computation of the quotient space}
From now on, we assume that $\frF = \bQ_p$, $p \neq 2,3$ and $\pi = p$. Our goal is to compute the quotient locally ringed space $\Xi / \cR$, where
\[\cR := V(\Ann_{Z \otimes_k Z}(E^\ast)) \subset \Xi \times \Xi.\]

\subsubsection{Passing to the normalisation} Cor. \ref{cor: AnnEstar} tells us that
\[\cR = V\left( \ker \mult_Z \hsp \cap \hsp \Ann_{Z \otimes_k Z}(B)\right).\]
Recall from Def. \ref{def: XiXiprimePhi} that $\theta : \Xi' \to \Xi$ is the normalisation of $\Xi$, where $\theta = \Spec(\varphi)$ and $\varphi : Z \to Z'$ is the map sending $X_\lambda \in Z$ to $e_\lambda t \in Z' = k[\Omega][t]$. Because the localisation map $Z \to Z_\zeta$ factors through $Z'$, we see that the $(Z_\zeta,Z_\zeta)$-bimodule $B$ from Def. \ref{def: bimodB} can be viewed as a $(Z',Z')$-bimodule via restriction along $Z' \hookrightarrow Z_\zeta$.
\begin{definition} \label{defn: DefOfRdash} We define $\cR' := V\left(\ker \mult_{Z'} \hsp \cap \hsp \Ann_{Z' \otimes_k Z'}(B)\right) \subset \Xi' \times \Xi'.$
\end{definition}
\begin{lemma} \label{lem: kermultZdash} Write $A' = Z' \otimes_k Z'$. Then we have
\[\ker \mult_{Z'} = A'(t \otimes 1 - 1 \otimes t) + \sum\limits_{\alpha \in \hatOm} A'(e_\alpha \otimes 1 - 1 \otimes e_\alpha).\]
\end{lemma}
\begin{proof} The $k$-algebra $Z' = k[\Omega][t]$ is generated by $\{t\} \cup \{e_\alpha : \alpha \in \hatOm\}$. Use Lemma \ref{lem: BiAnnA}.b.
\end{proof}
It turns out that $\Ann_{Z' \otimes_k Z'}(B)$ is slightly easier to compute than $\Ann_{Z \otimes_k Z}(B)$.
\begin{proposition} \label{prop: gensofJdash} Write $A' = Z' \otimes_k Z'$. Then we have
\[ \Ann_{A'}(B) = I := A'(t \otimes t - 1 \otimes 1) + \sum\limits_{\alpha \in \hatOm} A'(e_\alpha \otimes 1 - 1 \otimes e_{\mu/\alpha}).\]
\end{proposition}
\begin{proof} By Cor. \ref{cor: Zzeta}, there is a $k$-algebra isomorphism $Z'_t \cong Z_\zeta$. Under this isomorphism, $t \in Z'$ maps to $\zeta \in Z_\zeta$ and $e_\alpha \in Z'$ maps to $\eps_\alpha = X_\alpha/\zeta \in Z_\zeta$, for each $\alpha \in \hatOm$. This isomorphism also induces $k$-algebra isomorphisms $A'_{t \otimes t} \cong (Z'_t) \otimes_k (Z'_t) \cong (Z_\zeta) \otimes_k (Z_\zeta) \cong R$.  Regarding $B$ as an $A'_{t \otimes t}$-module via this isomorphism, we can apply Cor. \ref{cor: AnnRBgens} to obtain
\[ \Ann_{A'_{t \otimes t}}(B) = A'_{t \otimes t} (t \otimes 1 - 1 \otimes t^{-1}) + \sum\limits_{\alpha \in \hatOm} A'_{t \otimes t}(e_\alpha \otimes 1 - 1 \otimes e_{\mu/\alpha}) = A'_{t \otimes t} \cdot I.\]
Using the definition of $I$, we see that $A'/I$ is isomorphic to $S[t, t^{-1}]$ as a $k$-algebra, where $S := k[\Omega]^{\otimes 2}  / \left\langle e_\alpha \otimes 1 - 1 \otimes e_{\mu/\alpha} : \alpha \in\hatOm\right\rangle$. This ring is $t$-torsionfree. Therefore $A'/I$ is $t \otimes t$-torsionfree, so $A' \cap (I \cdot A'_{t \otimes t}) = I$. Hence $\Ann_{A'}(B) = A' \cap \Ann_{A'_{t \otimes t}}(B) = A' \cap (I \cdot A'_{t \otimes t}) = I.$
\end{proof}
Recall the coequaliser diagram $\xymatrix{\Xi_{\sing} \ar@<0.6ex>[r]^{a}\ar@<-0.6ex>[r]_{b} & \Xi' \ar[r]^\theta & \Xi}$ from Prop. \ref{prop: NormalisationAsQuotient}.
\begin{proposition} \label{prop: XiRXi'R'} Let $q' : \Xi' \to \Xi'/\cR'$ be a coequaliser of  in $\LRS$ of $\xymatrix{\cR' \ar@<0.6ex>[r]^{\pr_1^{\cR'}}\ar@<-0.6ex>[r]_{\pr_2^{\cR'}} & \Xi'}$, and let $s : \Xi'/\cR' \to (\Xi'/\cR')/\Xi_{\sing}$ be a coequaliser in $\LRS$ of $\xymatrix{\Xi_{\sing} \ar@<0.6ex>[r]^{q'a}\ar@<-0.6ex>[r]_{q'b} & \Xi'/\cR'}$. Then
\[\Xi / \cR \quad \cong \quad (\Xi'/\cR')/\Xi_{\sing}.\]
\end{proposition}
\begin{proof}The situation is summarised in the following diagram:
\begin{equation} \begin{split} \xymatrix{ & & \Xi_{\sing} \ar@<0.6ex>[dd]^b\ar@<-0.6ex>[dd]_a  \ar@<0.6ex>[ddrr]^{q'b}\ar@<-0.6ex>[ddrr]_{q'a} & & \\
 & & & & & \\
\cR' \ar@<0.6ex>[rr]^{\pr_1^{\cR'}}\ar@<-0.6ex>[rr]_{\pr_2^{\cR'}} && \Xi' \ar[rr]_{q'} \ar[d]_\theta & & \Xi'/\cR' \ar[dr]^s & \\
\cR \ar@<0.6ex>[rr]^{\pr_1^{\cR}}\ar@<-0.6ex>[rr]_{\pr_2^{\cR}} && \Xi \ar[rr]_q  & & \Xi/\cR & (\Xi'/\cR')/\Xi_{\sing}. } \end{split} \end{equation}
The result now follows from Lemma \ref{lem: TwoCoequalisers}, provided we can find an epimorphism $\theta' : \cR' \to \cR$ in $\LRS$ making the diagram commute, in the sense that $\theta' \pr_i^{\cR} = \pr_i^{\cR'} \theta$ holds for $i=1,2$.

Write $A = Z \otimes_k Z$ and $A' = Z' \otimes_k Z'$; then by Cor. \ref{cor: AnnEstar} we have
\[\cO(\cR) = \frac{A}{\ker \mult_Z \hsp \cap \hsp \Ann_A(B)}  \qmb{and} \cO(\cR') = \frac{A'}{\ker \mult_{Z'} \hsp \cap \hsp \Ann_{A'}(B)}.\]
The map $\varphi \otimes \varphi : A \to A'$ is an injective $k$-algebra homomorphism; furthermore we have
\[ \ker \mult_Z = (\varphi \otimes \varphi)^{-1}(\ker \mult_{Z'}) \qmb{and} \Ann_A(B) = (\varphi \otimes \varphi)^{-1}(\Ann_{A'}(B)).\]
Therefore $\varphi \otimes \varphi$ descends to a natural \emph{injective} $k$-algebra homomorphism
\begin{equation}\label{eq: thetadashhash} \overline{\varphi \otimes \varphi} : \cO(\cR) \hookrightarrow \cO(\cR').\end{equation}
We define $\theta' := \Spec(\overline{\varphi \otimes \varphi}) : \cR' \to \cR$ to be the corresponding morphism of affine schemes. The following diagram of commutative rings
\[ \xymatrix{
\cO(\cR') && \ar@<0.6ex>[ll]^(0.4){\pr_2^\sharp} \ar@<-0.6ex>[ll]_(0.4){\pr_1^\sharp} Z' \\
\cO(\cR) \ar[u]^{\overline{\varphi\otimes\varphi}} &&  \ar@<0.6ex>[ll]^(0.4){\pr_2^\sharp}\ar@<-0.6ex>[ll]_(0.4){\pr_1^\sharp} Z \ar[u]^{\varphi}
}\]
 is readily checked to be commutative: for example, we have
 \[\pr_1^\sharp(\varphi(a)) = \overline{\varphi(a) \otimes 1} = \overline{\varphi \otimes \varphi}( \overline{a \otimes 1} ) = \overline{\varphi \otimes \varphi}( \pr_1^\sharp(a)) \qmb{for any} a \in Z.\]
 Because the schemes $\Xi,\Xi',\cR,\cR'$ are all affine, it follows that $\theta' \pr_i^{\cR} = \pr_i^{\cR'} \theta$ holds for $i=1,2$. It remains to show that $\theta'$ is an epimorphism in $\LRS$, and this follows from Lemma \ref{lem: epiLRS} below, once we check its two conditions.

a) We note that $Z' = k[\Omega][t]$ is a finitely generated $Z$-module via $\varphi : Z \to Z'$: indeed, $\Omega \subset Z'$ is a finite generating set as a $k[t]$-module, and $k[t] = k[\varphi(\zeta)] \subseteq \varphi(Z)$. Hence $A'$ is a finitely generated $A$-module, so $\cO(\cR')$ is a finitely generated $\cO(\cR)$-module. Hence $|\theta'|$ is surjective by \cite{AM} Prop. 5.1 and Thm. 5.10.

b) The map $(\theta')^\sharp(\cR) = \overline{\varphi \otimes \varphi}$ is injective, as we saw above in equation (\ref{eq: thetadashhash}).
\end{proof}

\begin{lemma} \label{lem: epiLRS} \footnote{Because $|\cdot| : \LRS \to \Top$ is a left adjoint, it must preserve epimorphisms, so condition a) is necessary for Lemma \ref{lem: epiLRS} to hold. The discussion \cite{KStm} gives an explicit example where Lemma \ref{lem: epiLRS} fails if only condition b) is assumed to hold.}
Let $f : Y \to X$ be a morphism of affine schemes. Suppose that
\begin{itemize} \item[a)] $|f| : |Y| \to |X|$ is surjective,
\item[b)] $f^\sharp(X) : \cO(X) \to \cO(Y)$ is injective.
\end{itemize}
Then $f : Y \to X$ is an epimorphism in $\LRS$.
\end{lemma}
\begin{proof} Suppose that $uf = vf$ for some morphisms $u,v : X \to Z$ in $\LRS$. Then $|u| |f| = |v| |f|$ implies that $|u| = |v|$ because $|f|$ is surjective. Hence it remains to show that $u^\sharp = v^\sharp$. Now, $|u|_\ast(f^\sharp) \circ u^\sharp = (uf)^\sharp = (vf)^\sharp = |v|_\ast(f^\sharp) \circ v^\sharp$, so it is enough to show that $f^\sharp : \cO_X \to f_\ast \cO_Y$ is injective. This can be checked on basic open subsets $D(g)$ of $X$, $g \in \cO(X)$. But since $X, Y$ are affine, the map $f^\sharp(D(g)) : \cO(D(g)) \to \cO(f^{-1}D(g))$ is the localisation of the injective map $f^\sharp(X) : \cO(X) \to \cO(Y)$ at $g$ and is therefore also injective.
\end{proof}
\subsubsection{Calculating $\Xi'/\cR'$} For each $\alpha \in \hatOm$, we let $\Xi_\alpha := V(1 - e_\alpha) \subset \Xi'$ be the closed subscheme cut out by the idempotent $1 - e_\alpha \in \cO(\Xi') = k[\Omega][t]$. Let $t_\alpha$  be the image of $t \in \cO(\Xi')$ in $\cO(\Xi'_\alpha)$; then $\cO(\Xi'_\alpha) = k[t_\alpha]$ and $\Xi' = \coprod_{\alpha \in \hatOm} \Xi'_\alpha$. It follows that
\[ \Xi' \times \Xi' = \coprod_{\alpha,\beta \in \hatOm} \Xi'_\alpha \times \Xi'_\beta.\]
We write $\cR'_{\alpha,\beta} := \cR' \cap (\Xi'_\alpha \times \Xi'_\beta)$ for all $\alpha, \beta \in \hatOm$, so that
\[ \cR' = \coprod_{\alpha,\beta \in \hatOm} \cR'_{\alpha,\beta}.\]
For each $\alpha, \beta \in \hatOm$, will identify $\cO(\Xi'_\alpha \times \Xi'_\beta) = k[t_\alpha] \otimes_k k[t_\beta]$ with the polynomial algebra $k[x,y]$, via $x \mapsto t_\alpha \otimes 1$ and $y \mapsto 1 \otimes t_\beta$.  With this notation in hand, we can now calculate the defining equations for each $\cR'_{\alpha,\beta}$.
\begin{proposition} \label{prop: RdashAlphaBeta} For every $\alpha, \beta \in \hatOm$ we have
\[ \cR'_{\alpha,\beta} = V( (x-y)^{\delta_{\alpha,\beta}} (xy-1)^{\delta_{\alpha, \mu/\beta}} ).\]
\end{proposition}
\begin{proof} Consider the canonical projection $\pi_{\alpha,\beta} : A' = \cO(\Xi' \times \Xi') \twoheadrightarrow \cO(\Xi'_\alpha \times \Xi'_\beta) = k[x,y]$ with kernel $\ker \pi_{\alpha,\beta} = \langle (1 - e_\alpha) \otimes 1, 1 \otimes (1 - e_\beta) \rangle$. Then we for all $\gamma,\delta \in \hatOm$ we have
\begin{equation}\label{eq: piabcd} \begin{array}{lll}\pi_{\alpha,\beta}( e_{\gamma} \otimes e_{\delta} ) &=& \delta_{\alpha,\gamma} \delta_{\beta,\delta}, \\
\pi_{\alpha,\beta}(e_\gamma \otimes 1) &=& \delta_{\alpha,\gamma}, \\
\pi_{\alpha,\beta}(1 \otimes e_\delta) &=& \delta_{\beta,\delta}. \end{array} \end{equation}
Using Lemma \ref{lem: kermultZdash} together with equations $(\ref{eq: piabcd})$, we have
\begin{eqnarray*} \pi_{\alpha,\beta}(\ker \mult_{Z'}) &=& \langle x - y \rangle + \sum\limits_{\lambda \in \hatOm} \langle \pi_{\alpha,\beta}(e_\lambda \otimes 1) - \pi_{\alpha,\beta}(1 \otimes e_\lambda) \rangle \\
&=& \langle x - y \rangle + \sum\limits_{\lambda \in \hatOm} \langle \delta_{\alpha,\lambda} - \delta_{\beta,\lambda} \rangle \\
&=& \langle (x - y)^{\delta_{\alpha,\beta}} \rangle.\end{eqnarray*}
Similarly, using Prop. \ref{prop: gensofJdash} together with equations $(\ref{eq: piabcd})$, we have
\begin{eqnarray*} \pi_{\alpha,\beta}(\Ann_{A'}(B)) &=& \langle xy-1 \rangle + \sum\limits_{\lambda \in \hatOm} \langle \pi_{\alpha,\beta}(e_\lambda \otimes 1) - \pi_{\alpha,\beta}(1 \otimes e_{\mu/\lambda}) \rangle \\
&=& \langle xy - 1 \rangle + \sum\limits_{\lambda \in \hatOm} \langle \delta_{\alpha,\lambda} - \delta_{\beta,\mu/\lambda} \rangle \\
&=& \langle (xy - 1)^{\delta_{\alpha,\mu/\beta}} \rangle.\end{eqnarray*}
By the elementary Lemma \ref{lem: piIJ} below, we then have
\begin{eqnarray*}  \pi_{\alpha,\beta}(\ker \mult_{Z'} \hsp \cap \hsp \Ann_{A'}(B)) &=& \pi_{\alpha,\beta}(\ker \mult_{Z'}) \hsp \cap \hsp \pi_{\alpha,\beta}(\Ann_{A'}(B)) \\
&=& \langle (x-y)^{\delta_{\alpha,\beta}} \rangle \hsp \cap \hsp \langle (xy-1)^{\delta_{\alpha, \mu/\beta}}\rangle \\
&=& \langle (x-y)^{\delta_{\alpha,\beta}} (xy-1)^{\delta_{\alpha, \mu/\beta}}\rangle.
\end{eqnarray*}
Recalling that $\cR'_{\alpha,\beta} = \cR' \cap (\Xi'_\alpha \times \Xi'_\beta)$, and using Def. \ref{defn: DefOfRdash}, we have
\[\cR'_{\alpha,\beta} = V(\pi_{\alpha,\beta}(\ker \mult_{Z'} \hsp \cap \hsp \Ann_{A'}(B))) \subset \Xi'_\alpha \times \Xi'_\beta.\]
The result follows. \end{proof}
\begin{lemma} \label{lem: piIJ} Let $A$ be a ring, let $e,f \in A$ be two central idempotents and consider the canonical projection $\pi : A \twoheadrightarrow A / \langle 1-e,1-f\rangle$. Then for every pair of ideals $I,J$ of $A$ we have $\pi(I \cap J) = \pi(I) \cap \pi(J)$.
\end{lemma}
\begin{proof} Let $a \in \pi(I) \cap \pi(J)$. Then we can find $x \in I$ and $y \in J$ such that $a = \pi(x) = \pi(y)$. Now, $x = (1-e+e)(1-f+f)x$ implies that $\pi(x) = \pi(efx)$ and similarly $\pi(y) = \pi(efy)$. Hence $\pi(efx) = \pi(x) = \pi(y) = \pi(efy)$, so $efx - efy \in \ker \pi$. But $ef(1-e) = ef(1-f) = 0$, so $ef \ker \pi = 0$. Hence $ef(efx - efy) = 0$, so $efx = e^2f^2x = e^2f^2y = efy$. Since $efx \in I$ and $efy \in J$, we see that $efx = efy \in I \cap J$. Hence $a = \pi(x) = \pi(efx) \in \pi(I \cap J)$, so $\pi(I) \cap \pi(J) \subseteq \pi(I \cap J)$. The reverse inclusion is clear. \end{proof}
We introduce an equivalence relation $\sim$ on $\hatOm$ by setting $\alpha \sim \beta$ if and only if $\beta \in \{\alpha, \mu/\alpha\}$. For an equivalence class $\gamma \in \hatOm/\sim$, we write
\begin{equation}\label{eq: Xi'R'} \Xi'_\gamma := \coprod_{\alpha \in \gamma} \Xi'_\alpha \qmb{and} \cR'_\gamma := \coprod_{\alpha,\beta \in \gamma} \cR'_{\alpha,\beta}.\end{equation}
We use the projection maps $\pr_1^{\Xi'_\gamma}, \pr_2^{\Xi'_\gamma} : \Xi'_\gamma \times \Xi'_\gamma \to \Xi'_\gamma$ to define the map
\[ f_i^\gamma := \pr_i^{\Xi'_\gamma}|_{\cR'_\gamma} : \cR'_\gamma \to \Xi'_\gamma \qmb{for} i=1,2. \]
Then for each $\gamma \in \hatOm/\sim$, we may form the coequaliser diagram in $\LRS$
\begin{equation}\label{eq: Xi'ModR'} \xymatrix{\cR'_\gamma \ar@<0.6ex>[rr]^{f_1^\gamma}\ar@<-0.6ex>[rr]_{f_2^\gamma} && \Xi'_\gamma \ar[r] & \Xi'_\gamma / \cR'_\gamma.}\end{equation}
\begin{lemma} \label{lem: Xi'R'decomp} We have $\Xi' = \coprod\limits_{\gamma \in \hatOm/\sim} \Xi'_\gamma$ and $\cR' = \coprod\limits_{\gamma \in \hatOm/\sim}\cR'_\gamma$.
\end{lemma}
\begin{proof} Only the second statement needs proof. For this, we have
\[ \cR' = \coprod\limits_{\alpha,\beta \in \hatOm} \cR'_{\alpha,\beta} = \coprod\limits_{\gamma \in \hatOm/\sim}\coprod\limits_{\alpha \in \gamma} \coprod\limits_{\beta \in \hatOm} \cR'_{\alpha,\beta}.\]
But by Prop. \ref{prop: RdashAlphaBeta}, we have $\cR'_{\alpha,\beta} = \emptyset$ if $\beta \notin \{\alpha, \mu/\alpha\}$. Hence $\cR' = \coprod\limits_{\gamma \in \hatOm/\sim} \cR'_\gamma$.
\end{proof}
\begin{corollary} \label{cor: Xi'R'} We have $\Xi'/\cR' \cong \coprod\limits_{\gamma \in \hatOm/\sim} \Xi'_\gamma / \cR'_\gamma$ in $\LRS$.
\end{corollary}
\begin{proof} We have $\Xi' = \coprod\limits_{\gamma \in \hatOm/\sim} \Xi'_\gamma$ and $\cR' = \coprod\limits_{\gamma \in \hatOm/\sim}\cR'_\gamma$ by Lemma \ref{lem: Xi'R'decomp}. Note also that $\pr_i^{\Xi'} = \coprod\limits_{\gamma \in \hatOm/\sim} f_i^\gamma$ for $i = 1,2$. Now the result follows from Lemma \ref{lem: CoprodOfCoeq}.
\end{proof}

Recall \cite{Sta} \href{https://stacks.math.columbia.edu/tag/01JA}{Section 01JA} that a \emph{gluing datum in $\LRS$} consists of the following data:
\begin{itemize}
\item an index set $I$,
\item a locally ringed space $X_i$ for each $i \in I$,
\item an open subspace $\iota_{i,j} : U_{i,j} \hookrightarrow X_i$ for each $i,j \in I$,
\item an isomorphism $\varphi_{i,j} : U_{i,j} \stackrel{\cong}{\longrightarrow} U_{j,i}$ in $\LRS$ for all $i,j \in I$,
\end{itemize}
such that
\be \item $U_{i,i} = X_i$ for all $i \in I$, and
\item $\varphi_{k,j} \circ \varphi_{j,i}|_{U_{i,j} \cap U_{i,k}} = \varphi_{k,i}|_{U_{i,j} \cap U_{i,k}}$ holds for all $i,j,k \in I$.
\ee
Recall from \cite{DG} Prop. I.1.1.6 that $\LRS$ admits all colimits.
\begin{definition} \label{def:glueing} Let $\left(I, \{U_i\}_{i \in I}, \{\iota_{i,j}\}_{i,j \in I}, \{\varphi_{i,j}\}_{i,j\in I}\right)$ be a gluing datum in $\LRS$. Form the coproduct $U := \coprod\limits_{i \in I} U_i$ and let $\ell_i : U_i \hookrightarrow U$ be the canonical inclusions for each $i \in I$. Define
\begin{equation}\label{eq: gluecoeqdiag} \xymatrix{ \coprod\limits_{i,j} U_{i,j}   \ar@<0.6ex>[rr]^u\ar@<-0.6ex>[rr]_v && \coprod\limits_{i \in I} U_i}\end{equation}
by setting $u = (u_{i,j})_{i,j \in I}$ and $v = (v_{i,j})_{i,j \in I}$, where
\begin{equation} \label{eq: uvelliota} u_{i,j} = \ell_i \circ \iota_{i,j} \qmb{and} v_{i,j} = \ell_j \circ \iota_{j,i} \circ \varphi_{i,j}\qmb{for all} i,j \in I.\end{equation}
If $q : U \to X$ is a coequaliser of this diagram in $\LRS$, then we call $X$ the \emph{gluing of the $X_i$'s with respect to the gluing datum}.
\end{definition}

\begin{remark} It is shown in \cite{Sta}, \href{https://stacks.math.columbia.edu/tag/01JA}{Lemma 01JA} that the glued locally ringed space $X$ admits an open covering $\{U_i : i \in I\}$ such that $U_i \cong X_i$ for all $i \in I$. It follows immediately that $X$ is a scheme whenever each $U_i$ is a scheme. \end{remark}


\begin{theorem} \label{thm: gluingpair} Suppose that $\gamma = \{\alpha,\beta\}$, where $\beta = \mu/\alpha$ and $\alpha \neq \beta$. Then $\Xi'_\gamma/\cR'_\gamma$ is isomorphic to the projective line $\bP^1$.
\end{theorem}
\begin{proof} We will first recall the well-known gluing datum that is used in the construction of the projective line $\bP^1$: see, for example \cite{Sta}, \href{https://stacks.math.columbia.edu/tag/01JE}{Example 01JE}. The indexing set is $I := \{\alpha,\beta\}$. The spaces to be glued are $X_\alpha := \Xi'_\alpha = \Spec k[t_\alpha]$ and $X_ \beta := \Xi'_\beta = \Spec k[t_\beta]$. The $U_{i,j}$'s are as follows:
\[U_{\alpha,\alpha} := \Xi'_\alpha, \quad U_{\beta,\beta} = \Xi'_\beta, \quad U_{\alpha,\beta} = \Spec k[t_\alpha, t_\alpha^{-1}] \qmb{and} U_{\beta, \alpha} = \Spec k[t_\beta, t_\beta^{-1}]\]
The inclusions $\iota_{\alpha,\beta}$ are made clear by the notation, and the gluing isomorphisms $\varphi_{i,j}$ are as follows: $\varphi_{\alpha,\alpha} = \id_{U_\alpha}$, $\varphi_{\beta,\beta} = \id_{U_\beta}$, $\varphi_{\beta,\alpha} = \varphi_{\alpha,\beta}^{-1}$, and $\varphi_{\alpha,\beta} : U_{\alpha,\beta} \stackrel{\cong}{\longrightarrow} U_{\beta,\alpha}$ is determined by the corresponding map $\varphi_{\alpha,\beta}^\sharp$ on coordinate rings, which is given by $\varphi_{\alpha,\beta}^\sharp(t_\beta) = t_{\alpha}^{-1}$. Hence we have the coequaliser diagram
\[ U_\gamma := \xymatrix{ \coprod\limits_{i,j \in \{\alpha,\beta\}} U_{i,j} \ar@<0.6ex>[rr]^u\ar@<-0.6ex>[rr]_v && \Xi'_\alpha\coprod \Xi'_\beta \ar[rr]^q && \bP^1}.\]
Recall that $\Xi'_\gamma/\cR'_\gamma$ is defined by the coequaliser diagram (\ref{eq: Xi'ModR'}). We see that to show that $\Xi'_\gamma /\cR'_\gamma \cong \bP^1$ as locally ringed spaces, it will be enough find an isomorphism of schemes
\[ \tau : \cR'_\gamma \stackrel{\cong}{\longrightarrow} U_\gamma \]
such that the following two diagrams of affine schemes are commutative:
\[ \xymatrix{ \cR'_\gamma \ar[rr]\ar[d]^\cong_\tau && \Xi'_\gamma\times\Xi'_\gamma \ar[d]^{\pr_1} \\ U_\gamma \ar[rr]_u && \Xi'_\gamma } \qmb{and} \xymatrix{ \cR'_\gamma \ar[rr]\ar[d]^\cong_\tau && \Xi'_\gamma\times\Xi'_\gamma \ar[d]^{\pr_2} \\ U_\gamma \ar[rr]_v && \Xi'_\gamma. }\]
Here, both of the top horizontal arrows are the closed embedding of $\cR'_\gamma$ into $\Xi'_\gamma \times \Xi'_\gamma$. Now, the relation $\cR'_\gamma$ also decomposes as a disjoint union
\[\cR'_\gamma =\coprod\limits_{i,j \in \{\alpha,\beta\}} \cR'_{i,j} \quad \subset \quad \Xi'_\gamma \times\Xi'_\gamma = \coprod\limits_{i,j \in \{\alpha,\beta\}} \Xi'_i \times \Xi'_j\]
so it will be enough to work componentwise, and for each $i,j \in \gamma$ to find an isomorphism
\[ \tau_{ij} : \cR'_{i,j} \stackrel{\cong}{\longrightarrow} U_{i,j} \]
such that the following two diagrams of affine schemes are commutative:
\begin{equation}\label{eq: tauij} \xymatrix{ \cR'_{i,j} \ar[rr]\ar[d]^\cong_{\tau_{i,j}} && \Xi'_i \times\Xi'_j \ar[d]^{\pr_1} \\ U_{i,j} \ar[rr]_{u_{i,j}} && \Xi'_i } \qmb{and} \xymatrix{ \cR'_{i,j} \ar[rr]\ar[d]^\cong_{\tau_{i,j}} && \Xi'_i\times\Xi'_j \ar[d]^{\pr_2} \\ U_{i,j} \ar[rr]_{v_{i,j}} && \Xi'_j. }\end{equation}
Because $\alpha \neq \beta$, the connected components $\cR'_{i,j}$ of $\cR'_\gamma$ are given by Prop. \ref{prop: RdashAlphaBeta} as follows:
\begin{equation}\label{eq: Rij} \begin{array}{lll}\cR'_{\alpha,\alpha} &=& V(\langle t_\alpha \otimes 1 - 1 \otimes t_\alpha \rangle), \\
\cR'_{\beta,\beta} &=& V(\langle t_\beta \otimes 1 - 1 \otimes t_\beta \rangle), \\
\cR'_{\alpha,\beta} &=& V(\langle t_\alpha \otimes t_\beta - 1 \otimes 1\rangle), \qmb{and} \\
\cR'_{\beta,\alpha} &=& V(\langle t_\beta \otimes t_\alpha - 1 \otimes 1\rangle).\end{array} \end{equation}

Suppose first that $j = i$. In this case $u_{i,i} = \id_{\Xi'_i}$, so we define $\tau_{i,i} := \pr_1|_{\cR'_{i,j}}$; then the first diagram in (\ref{eq: tauij}) commutes by definition. For the second diagram, note that $v_{i,i} = u_{i,i} \circ \varphi_{i,i}$ is also equal to $\id_{\Xi'_i}$, so $v_{i,i} \circ \tau_{i,i}  = \pr_1|_{\cR'_{i,i}} =  \pr_2|_{\cR'_{i,i}}$ in view of the first two equations in $(\ref{eq: Rij})$. Hence the second diagram in (\ref{eq: tauij}) commutes as well.

Suppose now that $j \neq i$. Looking at the last two equations in $(\ref{eq: Rij})$, we see that the element $\pr_1^\sharp(t_i)|_{\cR'_{i,j}} = \overline{t_i \otimes 1}$ is a unit in $\cO(\cR'_{i,j})$. Therefore the $k$-algebra homomorphism $(\pr_1|_{\cR'_{i,j}})^\sharp : \cO(\Xi'_i) \to \cO(\cR'_{i,j})$ extends to the localisation $\cO(U_{i,j})$ of $\cO(\Xi'_i)$, which means that $\pr_1|_{\cR'_{i,j}}$ factors through the Zariski open subset $U_{i,j}$ of $\Xi'_i$. In other words, there exists a morphism $\tau_{i,j} : \cR'_{i,j} \to U_{i,j}$, making the first diagram in $(\ref{eq: tauij})$ commutative. Looking at the last two equations in $(\ref{eq: Rij})$ again, we see that the corresponding map on coordinate rings
\[ \tau_{i,j}^\sharp : k[t_i, t_i^{-1}] \to \frac{k[t_i] \otimes k[t_j]}{\langle t_i \otimes t_j - 1 \otimes 1 \rangle}\]
sends $t_i$ to $\overline{t_i \otimes 1}$ and $t_i^{-1}$ to $\overline{t_i \otimes 1}^{-1} = \overline{1 \otimes t_j}$, and that it is an isomorphism. Therefore $\tau_{i,j}$ is also an isomorphism. It remains to check that the second diagram in $(\ref{eq: tauij})$ is commutative. Since all schemes involved are affine, it will be enough to check on coordinate rings.

Suppose that $(i,j) = (\alpha,\beta)$. Since $v_{\alpha,\beta} =  u_{\beta,\alpha} \circ \varphi_{\alpha,\beta}$ by (\ref{eq: uvelliota}), this diagram is
\[\xymatrix{ \frac{ k[t_\alpha] \otimes k[t_\beta] }{ \langle t_\alpha \otimes t_\beta - 1 \otimes 1 \rangle } &&&& k[t_\alpha] \otimes k[t_\beta] \ar[llll]\\
k[t_\alpha, t_\alpha^{-1}] \ar[u]^{\tau_{\alpha,\beta}^\sharp} && k[t_\beta,t_\beta^{-1}] \ar[ll]^{\varphi_{\alpha,\beta}^\sharp}&& k[t_\beta]\ar[ll]\ar[u]_{\pr_2^\sharp} }\]
and it is commutative because $\tau_{\alpha,\beta}^\sharp(\varphi_{\alpha,\beta}^\sharp(t_\beta)) = \tau_{\alpha,\beta}^\sharp(t_\alpha^{-1}) = \overline{1 \otimes t_\beta} = \overline{\pr_2^\sharp(t_\beta)}$.

Similarly, when $(i,j) = (\beta,\alpha)$, because $v_{\beta,\alpha} = u_{\alpha,\beta} \circ \varphi_{\beta,\alpha}$ by (\ref{eq: uvelliota}), this diagram is
\[\xymatrix{ \frac{ k[t_\beta] \otimes k[t_\alpha] }{ \langle t_\beta \otimes t_\alpha - 1 \otimes 1 \rangle } &&&& k[t_\beta] \otimes k[t_\alpha] \ar[llll]\\
k[t_\beta, t_\beta^{-1}] \ar[u]^{\tau_{\beta,\alpha}^\sharp} && k[t_\alpha,t_\alpha^{-1}] \ar[ll]^{\varphi_{\beta,\alpha}^\sharp}&& k[t_\alpha]\ar[ll]\ar[u]_{\pr_2^\sharp} }\]
and it is commutative because $\tau_{\beta,\alpha}^\sharp(\varphi_{\beta,\alpha}^\sharp(t_\alpha)) = \tau_{\beta,\alpha}^\sharp(t_\beta^{-1}) = \overline{1 \otimes t_\alpha} = \overline{\pr_2^\sharp(t_\alpha)}$.  \end{proof}
Next, we will study the following diagram of $k$-schemes:
\begin{equation}\label{eq: RAPdiagram} \xymatrix{\bR \ar@<0.6ex>[r]^{\pr_1}\ar@<-0.6ex>[r]_{\pr_2} & \bA^1 \ar[r]^\psi & \bP^1}\end{equation}
where $\bA^1 = \Spec k[x]$ is an affine line, $\bR = \Delta \cup \bH = V((x-y)(xy-1)) \subseteq \bA^2 = \Spec k[x,y]$ is the union of the hyperbola $\bH = V(xy-1)$ and the diagonal $\Delta = V(x-y)$, $\pr_1, \pr_2$ are the projection morphisms whose respective comorphisms are determined by $\pr_1^\sharp(x) = x$ and $\pr_2^\sharp(x) = y$, and the morphism $\psi : \bA^1 \to \bP^1$, viewed as a natural transformation between the corresponding functors of points, is given by the rule
\[ \psi(a) = (a^2 + 1 : a) \qmb{for all} a \in \bA^1.\]
\begin{lemma} \label{lem: psipri} We have $\psi \circ \pr_1 = \psi \circ \pr_2$.
\end{lemma}
\begin{proof} Let $(a,b) \in \bR$. Then $\psi(\pr_1(a,b)) = \psi(a) = (a^2+1:a)$ and $\psi(\pr_2(a,b)) = \psi(b) = (b^2 + 1 : b)$. However $(a,b) \in \bR$ implies that $a = b$ or $ab = 1$, and in either case we have $(a^2+1)b  = a(b^2+1)$. Hence $\psi(\pr_1(a,b)) = \psi(\pr_2(a,b))$. \end{proof}

We let $y$ be a local coordinate on $\bP^1$, determined by $y( (a:1)) = a$, so that $Y_0 := \Spec k[y]$ and $Y_\infty := \Spec k[1/y]$ form a standard open covering of $\bP^1$ by two affine lines.
\begin{lemma} \label{lem: XiHiaffine} Let $X_i = \psi^{-1}(Y_i)$ and $\bR_i = \pr_1^{-1}(X_i)$ for $i = 0, \infty$. Then $X_i$ is a basic affine open in $\bA^1$, and $\bR_i$ is a basic affine open in $\bR$.
\end{lemma}
\begin{proof} We have $X_0 = \{a \in \bA^1 : (a^2 + 1 : a) \in Y_0\} = \{a \in \bA^1 :  a \neq 0\}$, which is the basic affine open $D(x) \subset \bA^1$. Similarly, $X_\infty = \{a \in \bA^1 : (a^2+1:a) \in Y_\infty\} = \{a \in \bA^1 : a^2+1\neq 0\}$, which is the basic affine open $D(x^2+1) \subset \bA^1$. The preimage of any basic affine open $D(f) \subset \Spec k[x]$ under the morphism of affine schemes $\pr_1: \bR \to \bA^1$ is $D(\pr_1^\sharp(f))$, which is a basic affine open in $\bR$. Therefore $\bR_0$ and $\bR_\infty$ are basic affine opens in $\bR$.
\end{proof}

\begin{proposition} \label{prop: exactYiXiHi} $\xymatrix{ 0 \ar[r]& \cO(Y_i) \ar[r]^{\psi^\sharp}& \cO(X_i) \ar@<0.6ex>[r]^{\pr_1^\sharp} \ar@<-0.6ex>[r]_{\pr_2^\sharp} & \cO(\bR_i)}$ is an equaliser diagram of commutative rings if $i = 0$ or $i = \infty$.
\end{proposition}
\begin{proof} Let $\bH_i = \bR_i \cap \bH$ for $i=0,\infty$. By postcomposing $\pr_1^\sharp, \pr_2^\sharp$ with the restriction maps $\cO(\bR_i) \twoheadrightarrow \cO(\bH_i)$, we can replace $\bR_i$ by $\bH_i$ in this proof.

First we consider the case $i = 0$, where the diagram becomes
\[\xymatrix{ 0 \ar[r]& k[y] \ar[r]^{\psi^\sharp}& k[x,x^{-1}] \ar@<0.6ex>[r]^{\pr_1^\sharp} \ar@<-0.6ex>[r]_{\pr_2^\sharp} & k[x,x^{-1}]}\]
and the maps are given by $\psi^\sharp(y) = x + x^{-1}$, $\pr_1^\sharp(x) = x$, $\pr_2^\sharp(x) = x^{-1}$. Clearly $\psi^\sharp$ is injective, and $\pr_1^\sharp \psi^\sharp = \pr_2^\sharp \psi^\sharp$ by Lemma \ref{lem: psipri}.

Suppose that $a = \sum\limits_{i=-n}^n a_ix^i \in \ker(\pr_1^\sharp - \pr_2^\sharp)$ for some $a_i \in k$. Then $a_i = a_{-i}$ for all $i$, so $a$ lies in the $k$-linear span of $\{x^n + x^{-n} : n \geq 0\}$. Hence $a \in F_m := \sum\limits_{n=0}^m k(x^n + x^{-n})$ for some $m \geq 0$. We will show by induction that $F_r \subseteq k[x + x^{-1}]$ for all $r \geq 0$. The base case $r = 0$ is clear. Assuming inductively that $F_{r-1} \subseteq k[x+x^{-1}]$ for some $r \geq 1$, we see that $x^r + x^{-r} \equiv (x+x^{-1})^r \mod F_{r-1}$, so $x^r + x^{-r} \in k[x+x^{-1}]$ as well. Hence
\begin{equation}\label{eq: kerpr1pr2} \ker(\pr_1^\sharp - \pr_2^\sharp) = k[x + x^{-1}],\end{equation}
and the sequence is exact in the middle as required.

Next, we consider the case where $i = \infty$. The diagram becomes
\[\xymatrix{ 0 \ar[r]& k[y^{-1}] \ar[r]^{\psi^\sharp}& k[x,\frac{1}{x^2+1}] \ar@<0.6ex>[r]^(0.45){\pr_1^\sharp} \ar@<-0.6ex>[r]_(0.45){\pr_2^\sharp} & k\left[x,x^{-1}, \frac{1}{x^2+1}\right]}\]
where the maps are now given by $\psi^\sharp(y^{-1}) = \frac{1}{x + x^{-1}} = \frac{x}{x^2+1}$, $\pr_1^\sharp(x) = x$, and $\pr_2^\sharp(x) = x^{-1}$. It is again clear that $\psi^\sharp$ is injective, and that $\pr_1^\sharp \psi^\sharp = \pr_2^\sharp \psi^\sharp$ by Lemma \ref{lem: psipri}.

Suppose that $\frac{a(x)}{(x^2+1)^n} \in \ker(\pr_1^\sharp - \pr_2^\sharp)$ for some $a(x) \in k[x]$ and some $n \geq 0$. Then
\[ \frac{a(x)}{(x^2+1)^n} = \frac{a(x^{-1})}{(x^{-2} + 1)^n}.\]
Multiplying through by $(x^2+1)^n$ shows that $a(x) = a(x^{-1})x^{2n}$. Hence $\deg(a) \leq 2n$, and dividing through by $x^n$ shows that $a(x)x^{-n} \in k[x,x^{-1}]$ is invariant under the substitution $x \mapsto x^{-1}$. Using $(\ref{eq: kerpr1pr2})$, we see that $a(x)x^{-n} = b(x + x^{-1})$ for some $b(x) \in k[x]$ with $\deg b \leq n$. Write $b(x) = b_0 + b_1 x + \cdots + b_nx^n$ for some $b_0,\cdots,b_n \in k$; then
\[ \frac{a(x)}{(x^2+1)^n} = \frac{x^n}{(x^2+1)^n}\sum\limits_{i=0}^n b_i \left(\frac{x^2+1}{x}\right)^i = \sum_{i=0}^n b_i \left(\frac{x}{x^2+1}\right)^{n-i} \in k\left[\frac{x}{x^2+1}\right].\]
Hence $\frac{a(x)}{(x^2+1)^n} \in k\left[\frac{x}{x^2+1}\right]$ and the sequence is exact in the middle as required.
\end{proof}
\begin{proposition}\label{prop: RAPcoeq} $\xymatrix{\bR \ar@<0.6ex>[r]^{\pr_1}\ar@<-0.6ex>[r]_{\pr_2} & \bA^1 \ar[r]^\psi & \bP^1}$ is a coequaliser diagram in $\LRS$.
 \end{proposition}
\begin{proof} Noting that $\psi \circ \pr_1 = \psi \circ \pr_2$ by Lemma \ref{lem: psipri}, this follows from Lemma \ref{lem: catquot}, applied with the standard covering $\{Y_0 = \Spec k[y], Y_\infty = \Spec k[y^{-1}]\}$ of $Y = \bP^1$, once we have checked its three conditions. Condition b) follows from Lemma \ref{lem: XiHiaffine} and condition c) follows from Prop. \ref{prop: exactYiXiHi}, so it remains to check condition a). This amounts to checking the following:
\begin{itemize}
\item[(i)] $|\psi| : |\bA^1| \to |\bP^1|$ is surjective,
\item[(ii)] the equivalence classes of the equivalence relation $\sim$ on $|\bA^1|$ defined by $\xymatrix{|\bR| \ar@<0.6ex>[r]^{|\pr_1|}\ar@<-0.6ex>[r]_{|\pr_2|} & |\bA^1|}$ are equal to the fibres of $|\psi|$,
\item[(iii)] the induced map $|\bA^1|/\sim \quad \longrightarrow \quad |\bP^1|$ is a homeomorphism.
\end{itemize}

For (i) and (ii), for any $k$-variety of finite type $X$, we use \cite{MO}, Chapter IV, Thm. 2.3 to identify $|X|$ with the set of $\cG$-orbits in $|X_{\ok}|$, where $\cG = \Aut(\ok/k)$ and $X_{\ok} = X \times_k \ok$ is the base-change of $X$ to $\ok$. Then we have the following commutative diagram
\[\xymatrix{|\bR_{\ok}| \ar@<0.6ex>[r]\ar@<-0.6ex>[r]\ar[d] & |\bA^1_{\ok}| \ar[r]^{|\psi_{\ok}|}\ar[d] & |\bP^1_{\ok}|\ar[d] \\
|\bR| \ar@<0.6ex>[r]\ar@<-0.6ex>[r] & |\bA^1| \ar[r]_{|\psi|} & |\bP^1|}\]
where the vertical arrows are surjective. Chasing this diagram reduces us to the case $k = \overline{k}$.

(i) The map $|\psi|$ sends the generic point in $|\bA^1|$ to the generic point in $|\bP^1|$, and $0 \in \bA^1$ to the point at infinity $(1:0) \in \bP^1$. Any other point in $|\bP^1|$ is of the form $(c : 1)$ for some $c \in k$; since $k$ is algebraically closed, the equation $a + a^{-1} = c$ has a solution, so $(c:1) \in \Im(|\psi|)$.

(ii) Let $a, b \in |\bA^1|$ be such that $\psi(a) = \psi(b)$. If $\psi(a)$ is the generic point of $\bP^1$, then $a$ and $b$ must both be equal to the generic point of $\bA^1$, so assume otherwise. Then necessarily $a,b$ are closed points in $\bA^1$. Since $k = \overline{k}$ we may assume that $a,b \in \bA^1(k) = k$. Now $\psi(a) = \psi(b)$ implies that $(a^2+1:a) = (b^2+1:b)$, hence $(a^2+1)b = a(b^2+1)$, so $(a-b)(ab-1) = 0$ and $(a,b) \in \bR(k)$. Setting $u = (a,b) \in |\bR|$, we see that $a = \pr_1(u)$ and $b = \pr_2(u)$ as required.


(iii) The map $|\psi| : |\bA^1| \to |\bP^1|$ is surjective and has finite fibres. Hence the quotient topology on $|\bP^1|$ induced by this map from the Zariski topology on $|\bA^1|$ is the cofinite topology, and therefore coincides with the Zariski topology on $|\bP^1|$.
\end{proof}

\begin{theorem}\label{thm: selfgluing} Suppose that $\gamma = \{\alpha\}$, where $\alpha = \mu/\alpha$. Then $\Xi'_\gamma/\cR'_\gamma$ is also isomorphic to the projective line $\bP^1$.
 \end{theorem}
\begin{proof}
Since $\gamma = \{\alpha\}$, we have $\Xi'_\gamma = \Xi'_\alpha = \Spec k[t_\alpha]$ and $\cR'_\gamma = \cR'_{\alpha,\alpha}$ in view of equation (\ref{eq: Xi'R'}). Since $\alpha = \mu/\alpha$, Prop. \ref{prop: RdashAlphaBeta} tells us that $\cR'_{\alpha,\alpha} = V((x-y)(xy-1))$, where $x = t_\alpha \otimes 1$ and $y = 1 \otimes t_\alpha$. Hence we have a commutative diagram of schemes
\begin{equation}\label{eq: RAPdiagram2}
\xymatrix{
\cR'_\gamma \ar@<0.6ex>[rr]^{f_1^\gamma}\ar@<-0.6ex>[rr]_{f_2^\gamma}\ar[d]_\cong && \Xi'_\gamma \ar[r]\ar[d]^\cong & \Xi'_\gamma / \cR'_\gamma \\
\bR \ar@<0.6ex>[rr]^{\pr_1}\ar@<-0.6ex>[rr]_{\pr_2} && \bA^1 \ar[r]^\psi & \bP^1
}
\end{equation}
where the top line is the diagram (\ref{eq: Xi'ModR'}), and the vertical arrows are the isomorphisms of schemes whose corresponding comorphisms are
\[  \cO(\bR) = \frac{k[x,y]}{\langle (x-y)(xy-1) \rangle} \stackrel{\cong}{\longrightarrow} \cO(\cR'_{\alpha,\alpha}) \qmb{and} \cO(\bA^1) = k[x] \stackrel{\cong}{\longrightarrow} k[t_\alpha] = \cO(\Xi'_\alpha), \]
given by $x \mapsto t_\alpha \otimes 1, y \mapsto 1 \otimes t_\beta$ and $x \mapsto t_\alpha$, respectively. Since $\psi$ is a coequaliser of $\pr_1,\pr_2$ in $\LRS$ by Prop. \ref{prop: RAPcoeq}, this gives the required isomorphism of schemes $\Xi'_\gamma / \cR'_\gamma \stackrel{\cong}{\longrightarrow} \bP^1$.
\end{proof}

\begin{corollary} \label{cor: sizeHatOm} $\Xi'/\cR'$ is isomorphic to the disjoint union of $\frac{p+1}{2}$ projective lines.
\end{corollary}
\begin{proof} After Cor. \ref{cor: Xi'R'}, Thm. \ref{thm: gluingpair} and Thm. \ref{thm: selfgluing}, $\Xi'/\cR'$ is isomorphic to the disjoint union of $|\hatOm/\!\!\!\sim\!\!\!|$ projective lines. There are exactly $2$ elements $\alpha \in \hatOm$ that satisfy $\alpha = \mu/\alpha$: if $\alpha = \id^j$ for some $j = 0,\cdots, p-2$, then $\alpha^2 = \id^{2j} = \mu = \id^2$ if and only if $2j \equiv 2 \mod p-1$ which is equivalent to $j = 1$ or $j = \frac{p+1}{2}$. These correspond to the singleton equivalence classes, and all other classes have size two. Therefore $|\hatOm/\!\!\sim\!\!| = 2 + \frac{p-3}{2} = \frac{p+1}{2}$.
\end{proof}

\subsubsection{Re-gluing the projective lines to form $\Xi/\cR$}

At this point we have to introduce more notation.

\begin{definition} \label{def: NotnP1comps} For each $1 \leq r \leq \frac{p+1}{2}$, let $\gamma_r := \{ \id^r, \id^{2-r} \}$ and define
\[Z_r := \Xi'_{\gamma_r} / \cR'_{\gamma_r}\]
\end{definition}
Note that $\Xi'/\cR' \cong Z_1 \coprod Z_2 \coprod \cdots \coprod Z_{\frac{p+1}{2}}$ by Cor. \ref{cor: Xi'R'} and Cor. \ref{cor: sizeHatOm}. After Thm. \ref{thm: gluingpair} and Thm. \ref{thm: selfgluing}, we know that each $Z_r$ is isomorphic to the projective line $\bP^1$.  Note also that if $\alpha = \id^r$, then $\alpha = \mu/\alpha$ if and only if $r = 1$ or $r = \frac{p+1}{2}$.

\begin{definition} \label{def: LabelsOnLines} Let $1 \leq r \leq \frac{p+1}{2}$ and write $\alpha = \id^r$.
\begin{itemize}
\item If $r = 1$, let $z_r$ be the local coordinate on $Z_r$ that pulls back to $\frac{t_{\alpha}}{t_{\alpha}^2+1}$ under the morphism $\Xi'_{\alpha} \twoheadrightarrow \bP^1$ from the proof of Thm. \ref{thm: selfgluing}.
\item If $r \neq 1$ and $r \neq \frac{p+1}{2}$, let $z_r$ be the local coordinate on $Z_r$ that pulls back to $t_{\alpha}$ under the morphism $\Xi'_\alpha \coprod \Xi'_{\id^2/\alpha} \twoheadrightarrow \bP^1$ from the proof of Thm. \ref{thm: gluingpair}.
\item If $r =  \frac{p+1}{2}$, let $z_r$ be the local coordinate on $Z_r$ that pulls back to $t_{\alpha} + t_{\alpha}^{-1}$ under the morphism $\Xi'_\alpha \twoheadrightarrow \bP^1$ from the proof of Thm. \ref{thm: selfgluing}.
\item Let $O_r, \infty_r \in Z_r$ be the closed points defined by $O_r := \{z_r = 0\}$ and $\infty_r := \{z_r = \infty\}$.
\end{itemize}
\end{definition}
Recall that for any $\alpha \in \hatOm$, the origin on the affine line $\Xi'_{\alpha}$ is denoted by $O_\alpha$, and that $q' : \Xi' \twoheadrightarrow \Xi'/\cR'$ denotes the quotient map. After revisiting the proofs of Thm. \ref{thm: gluingpair} and Thm. \ref{thm: selfgluing}, we have the following
\begin{lemma} \label{lem: markedpoints} Let $1 \leq r \leq \frac{p+1}{2}$.
\begin{itemize}
\item If $r = 1$, then $q'(O_{\id^1}) = O_1$.
\item If $r \neq 1$ and $r \neq \frac{p+1}{2}$, then $q'(O_{\id^r}) = O_r$ and $q'(O_{\id^{2-r}}) = \infty_r$.
\item If $r = \frac{p+1}{2}$, then $q'(O_{\id^{\frac{p+1}{2}}}) = \infty_{\frac{p+1}{2}}$.
\end{itemize}
\end{lemma}
Recall the coequaliser diagram $\xymatrix{\Xi_{\sing} \ar@<0.6ex>[r]^{a}\ar@<-0.6ex>[r]_{b} & \Xi' \ar[r]^\theta & \Xi}$ from Prop. \ref{prop: NormalisationAsQuotient}. After Prop. \ref{prop: XiRXi'R'}, we wish to better understand the coequaliser diagram
\[\xymatrix{\Xi_{\sing} \ar@<0.6ex>[r]^{q'a}\ar@<-0.6ex>[r]_{q'b} & \Xi'/\cR' \ar[r]^(0.4)s & (\Xi'/\cR')/\Xi_{\sing}}.\]
\begin{lemma} \label{lem: explicitgluing} The pairs $\{\left(q'a(s_r), q'b(s_r)\right) : r = 1, \cdots, \frac{p-3}{2}\}$ are explicitly given as follows:
\[ (q'a(s_r), q'b(s_r)) = (O_r, \infty_{r+2}) \qmb{for all} r = 1,\cdots, \frac{p-3}{2}.\]
\end{lemma}
\begin{proof} Suppose first that $r \neq \frac{p-3}{2}$. Then using Def. \ref{def: XiSingab}.b and Lemma \ref{lem: markedpoints}, we have
\[ q'(a(s_r)) = q'(O_{\id^r}) = O_r \qmb{and} q'(b(s_r)) = q'(O_{\id^{-r}}) = q'( O_{\id^{2 - (r+2)}} ) =  \infty_{r+2},\]
where the last equality holds because $1 \leq r < \frac{p-3}{2}$ implies that $1 < r + 2 < \frac{p+1}{2}$.

Suppose now that $r = \frac{p-3}{2}$. Then we still have $q'(a(s_r)) = q'(O_{\id^r}) = O_r$, but now the last case in Lemma \ref{lem: markedpoints} gives
\[q'(b(s_{\frac{p-3}{2}})) = q'(O_{\id^{-\left(\frac{p-3}{2}\right)}}) = q'(O_{\id^{\frac{p+1}{2}}}) = \infty_{\frac{p+1}{2}}. \qedhere\]
\end{proof}

\begin{theorem} \label{thm: XiRscheme} The locally ringed space $\Xi/\cR$ is a scheme.
\end{theorem}
\begin{proof} We know that $\Xi'_\gamma/\cR'_\gamma$ is a scheme for each $\gamma \in \hatOm /\!\!\sim\!$ by Thm. \ref{thm: gluingpair} and Thm. \ref{thm: selfgluing}. Hence $\Xi'/\cR'$ is a scheme by Cor. \ref{cor: Xi'R'}.  By Prop. \ref{prop: XiRXi'R'}, it remains to show that the locally ringed space $(\Xi'/\cR') / \Xi_{\sing}$ is a scheme. Now, $\Xi'/\cR' \cong \coprod\limits_{r=1}^{\frac{p+1}{2}} Z_r$ and we define
\[ X_1 := \coprod\limits_{r\hsp \equiv \hsp 1 \hsp \mbox{\tiny or} \hsp 2 \!\!\! \mod 4} Z_r  \qmb{and} X_2 := \coprod\limits_{r\hsp \equiv \hsp 3 \hsp \mbox{\tiny or} \hsp 0 \!\!\! \mod 4}Z_r\]
so that $X_1 \coprod X_2 = \Xi'/\cR'$. We also define maps $\theta_1,\theta_2 : \Xi_{\sing} \to \Xi'/\cR'$ as follows:
\[ \theta_1(s_r) = \left\{ \begin{array}{lll} O_r & \mbox{ if } &r \equiv 1 \hsp \mbox{or} \hsp 2 \mod 4 \\ \infty_{r+2} & \mbox{ if } & r \equiv 3 \hsp \mbox{or} \hsp 0 \mod 4\end{array} \right. \quad \mbox{and}\quad \theta_2(s_r) = \left\{ \begin{array}{lll} \infty_{r+2} & \mbox{ if } &r \equiv 1 \hsp \mbox{or} \hsp 2 \mod 4 \\ O_r & \mbox{ if } & r \equiv 3 \hsp \mbox{or} \hsp 0 \mod 4\end{array} \right.\]
so that $\theta_1(\Xi_{\sing}) \subseteq X_1$ and $\theta_2(\Xi_{\sing}) \subseteq X_2$. Using Lemma \ref{lem: explicitgluing}, for all $s_r \in \Xi_{\sing}$ we have
\[ \left(\theta_1(s_r), \theta_2(s_r)\right) = \left(q'a(s_r), q'b(s_r)\right) \qmb{or} \left(q'b(s_r), q'a(s_r)\right). \]
Let $s : \Xi'/\cR' \to (\Xi'/\cR') / \Xi_{\sing}$ be a coequaliser of $\xymatrix{\Xi_{\sing} \ar@<0.6ex>[r]^(0.4){q'a}\ar@<-0.6ex>[r]_(0.4){q'b} & X_1 \coprod X_2}$; it follows that
\[\xymatrix{\Xi_{\sing} \ar@<0.6ex>[r]^(0.4){\theta_1}\ar@<-0.6ex>[r]_(0.4){\theta_2} & X_1 \coprod X_2 \ar[r]^(0.43)s & (\Xi'/\cR')/\Xi_{\sing}}\]
is a coequaliser diagram in $\LRS$. Hence $(\Xi'/\cR') / \Xi_{\sing}$ is isomorphic to the gluing of $X_1$ and $X_2$ along $\theta_1 : \Xi_{\sing} \to X_1$ and $\theta_2 : \Xi_{\sing} \to X_2$. Since $\theta_1$ and $\theta_2$ are both closed embeddings, this gluing is a scheme by Prop. 1.1.1 \cite{Ana}.\end{proof}
Using Lemma \ref{lem: explicitgluing}, we can also deduce the following
\begin{corollary} \label{cor: ConnCompsOfQuotSpace}Write $P_r := s(Z_r) \subset (\Xi'/\cR')/\Xi_{\sing}$ for $r = 1,\cdots, \frac{p+1}{2}$. \begin{itemize}
\item[a)] Suppose that $p \equiv 1 \mod 4$, so that $\frac{p+1}{2}$ is odd and $\frac{p-1}{2}$ is even. Then the connected component of $P_1$ in $(\Xi'/\cR')/\Xi_{\sing}$ is given by
\[ P_1 \cup P_3 \cup P_5 \cup \cdots \cup P_{\frac{p+1}{2}}\]
and the connected component of $P_2$ in $(\Xi'/\cR')/\Xi_{\sing}$ is given by
\[ P_2 \cup P_4 \cup P_6 \cup \cdots \cup P_{\frac{p-1}{2}}.\]
\item[b)] Suppose that $p \equiv 3 \mod 4$, so that $\frac{p-1}{2}$ is odd and $\frac{p+1}{2}$ is even. Then the connected component of $P_1$ in $(\Xi'/\cR')/\Xi_{\sing}$ is given by
\[ P_1 \cup P_3 \cup P_5 \cup \cdots \cup P_{\frac{p-1}{2}}\]
and the connected component of $P_2$ in $(\Xi'/\cR')/\Xi_{\sing}$ is given by
\[ P_2 \cup P_4 \cup P_6 \cup \cdots \cup P_{\frac{p+1}{2}}.\]
\end{itemize}
\end{corollary}

\begin{example}\label{ex:p13}When $p=13$, the schemes $\Xi$, $\Xi'$, $\Xi'/\cR'$ and $\Xi/\cR$ look as follows:

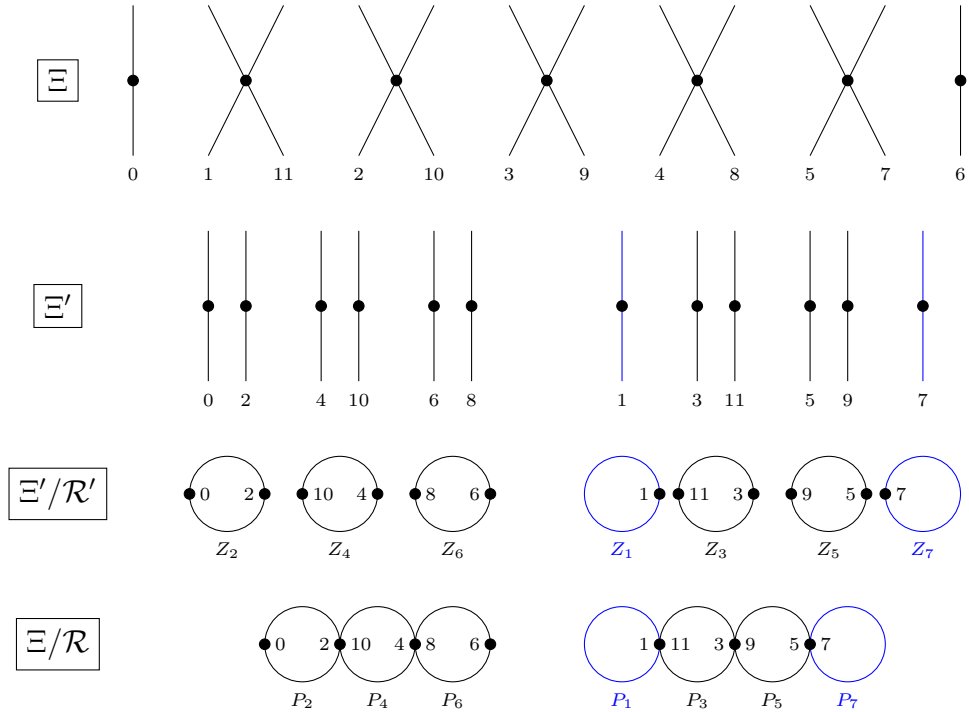
\begin{center}
    \begin{tikzpicture}
    \tikzstyle{every node}=[font=\large]
    \node [draw] at (0,1) {$\Xi$};
    \node [draw] at (0,-2) {$\Xi'$};
    \node [draw] at (0,-4.5) {$\Xi'/\cR'$};
    \node [draw] at (0,-6.5) {$\Xi/\cR$};

    \tikzstyle{every node}=[font=\tiny]

    \draw (1,0) -- (1,2);
    \draw (2,0) -- (3,2);
    \draw (3,0) -- (2,2);
    \draw (4,0) -- (5,2);
    \draw (5,0) -- (4,2);
    \draw (6,0) -- (7,2);
    \draw (7,0) -- (6,2);
    \draw (8,0) -- (9,2);
    \draw (9,0) -- (8,2);
    \draw (10,0) -- (11,2);
    \draw (11,0) -- (10,2);
    \draw (12,0) -- (12,2);

    \filldraw[black] (1, 1) circle (2pt);
    \filldraw[black] (2.5, 1) circle (2pt);
    \filldraw[black] (4.5, 1) circle (2pt);
    \filldraw[black] (6.5, 1) circle (2pt);
    \filldraw[black] (8.5, 1) circle (2pt);
    \filldraw[black] (10.5, 1) circle (2pt);
    \filldraw[black] (12, 1) circle (2pt);

    \path (1,-0.25) node {$0$};
    \path (2,-0.25) node {$1$};
    \path (3,-0.25) node {$11$};
    \path (4,-0.25) node {$2$};
    \path (5,-0.25) node {$10$};
    \path (6,-0.25) node {$3$};
    \path (7,-0.25) node {$9$};
    \path (8,-0.25) node {$4$};
    \path (9,-0.25) node {$8$};
    \path (10,-0.25) node {$5$};
    \path (11,-0.25) node {$7$};
    \path (12,-0.25) node {$6$};

    \draw (2,-3) -- (2,-1);
    \draw (2.5,-3) -- (2.5,-1);
    \draw (3.5,-3) -- (3.5,-1);
    \draw (4,-3) -- (4,-1);
    \draw (5,-3) -- (5,-1);
    \draw (5.5,-3) -- (5.5,-1);

    \draw[blue] (7.5,-3) -- (7.5,-1);
    \draw (8.5,-3) -- (8.5,-1);
    \draw (9,-3) -- (9,-1);
    \draw (10,-3) -- (10,-1);
    \draw (10.5,-3) -- (10.5,-1);
    \draw[blue] (11.5,-3) -- (11.5,-1);

    \filldraw[black] (2, -2) circle (2pt);
    \filldraw[black] (2.5, -2) circle (2pt);
    \filldraw[black] (3.5, -2) circle (2pt);
    \filldraw[black] (4, -2) circle (2pt);
    \filldraw[black] (5, -2) circle (2pt);
    \filldraw[black] (5.5, -2) circle (2pt);

    \filldraw[black] (7.5, -2) circle (2pt);
    \filldraw[black] (8.5, -2) circle (2pt);
    \filldraw[black] (9, -2) circle (2pt);
    \filldraw[black] (10, -2) circle (2pt);
    \filldraw[black] (10.5, -2) circle (2pt);
    \filldraw[black] (11.5, -2) circle (2pt);

    \path (2,-3.25) node {$0$};
    \path (2.5,-3.25) node {$2$};
    \path (3.5,-3.25) node {$4$};
    \path (4,-3.25) node {$10$};
    \path (5,-3.25) node {$6$};
    \path (5.5,-3.25) node {$8$};

    \path (7.5,-3.25) node {$1$};
    \path (8.5,-3.25) node {$3$};
    \path (9,-3.25) node {$11$};
    \path (10,-3.25) node {$5$};
    \path (10.5,-3.25) node {$9$};
    \path (11.5,-3.25) node {$7$};

    \draw (2.25,-4.5) circle (0.5cm) node at (2.25,-5.25) {$Z_2$};
    \draw (3.75,-4.5) circle (0.5cm) node at (3.75,-5.25) {$Z_4$};
    \draw (5.25,-4.5) circle (0.5cm) node at (5.25,-5.25) {$Z_6$};

    \draw[blue] (7.5,-4.5) circle (0.5cm) node at (7.5,-5.25) {$Z_1$};
    \draw (8.75,-4.5) circle (0.5cm) node at (8.75,-5.25) {$Z_3$};
    \draw (10.25,-4.5) circle (0.5cm) node at (10.25,-5.25) {$Z_5$};
    \draw[blue] (11.5,-4.5) circle (0.5cm) node at (11.5,-5.25) {$Z_7$};

    \filldraw[black] (1.75, -4.5) circle (2pt) node[right]{$0$};
    \filldraw[black] (2.75, -4.5) circle (2pt) node[left]{$2$};
    \filldraw[black] (3.25, -4.5) circle (2pt) node[right]{$10$};
    \filldraw[black] (4.25, -4.5) circle (2pt) node[left]{$4$};
    \filldraw[black] (4.75, -4.5) circle (2pt) node[right]{$8$};
    \filldraw[black] (5.75, -4.5) circle (2pt) node[left]{$6$};

    \filldraw[black] (8, -4.5) circle (2pt) node[left]{$1$};
    \filldraw[black] (8.25, -4.5) circle (2pt) node[right]{$11$};
    \filldraw[black] (9.25, -4.5) circle (2pt) node[left]{$3$};
    \filldraw[black] (9.75, -4.5) circle (2pt) node[right]{$9$};
    \filldraw[black] (10.75, -4.5) circle (2pt) node[left]{$5$};
    \filldraw[black] (11, -4.5) circle (2pt) node[right]{$7$};

    \draw (3.25,-6.5) circle (0.5cm) node at (3.25,-7.25) {$P_2$};
    \draw (4.25,-6.5) circle (0.5cm) node at (4.25,-7.25) {$P_4$};
    \draw (5.25,-6.5) circle (0.5cm) node at (5.25,-7.25) {$P_6$};

    \draw[blue] (7.5,-6.5) circle (0.5cm) node at (7.5,-7.25) {$P_1$};
    \draw (8.5,-6.5) circle (0.5cm) node at (8.5,-7.25) {$P_3$};
    \draw (9.5,-6.5) circle (0.5cm) node at (9.5,-7.25) {$P_5$};
    \draw[blue] (10.5,-6.5) circle (0.5cm) node at (10.5,-7.25) {$P_7$};

    \filldraw[black] (2.75, -6.5) circle (2pt) node[right]{$0$};
    \filldraw[black] (3.75, -6.5) circle (2pt) node[left]{$2$} node[right]{$10$};
    \filldraw[black] (4.75, -6.5) circle (2pt) node[left]{$4$} node[right]{$8$};
    \filldraw[black] (5.75, -6.5) circle (2pt) node[left]{$6$};

    \filldraw[black] (8, -6.5) circle (2pt) node[left]{$1$} node[right]{$11$};
    \filldraw[black] (9, -6.5) circle (2pt) node[left]{$3$} node[right]{$9$};
    \filldraw[black] (10, -6.5) circle (2pt) node[left]{$5$} node[right]{$7$};

    \end{tikzpicture}
\end{center}
\end{example}

\section*{Appendix}
\setcounter{section}{1}
\setcounter{subsection}{0}
\renewcommand{\thesection}{\Alph{section}}


\subsection{Some categorical results about colimits and coequalisers}

We omit the proof of the following standard result.
\begin{lemma}\label{lem:CoeqEpi} Suppose that $q : \cY \to \cZ$ is a coequaliser of the arrows $f,g : \cX \to \cY$ in a category $\cC$. Then $q$ is an epimorphism.
\end{lemma}


\begin{lemma}\label{lem: TwoCoequalisers} Let $\cC$ be a category, containing the following diagram:
\begin{equation} \label{eq: TwoCoequalisers}\begin{split} \xymatrix{ & & \cS \ar@<0.6ex>[dd]^b\ar@<-0.6ex>[dd]_a  \ar@<0.6ex>[ddrr]^{q'b}\ar@<-0.6ex>[ddrr]_{q'a} & & \\
 & & & & & \\
\cR' \ar@<0.6ex>[rr]^{f'}\ar@<-0.6ex>[rr]_{g'} && \cX' \ar[rr]_{q'} \ar[d]_\theta & & \cY' \ar[dr]^s & \\
\cR \ar@<0.6ex>[rr]^{f}\ar@<-0.6ex>[rr]_{g} && \cX \ar[rr]_q  & & \cY & \cZ. } \end{split} \end{equation}
Suppose that in this diagram, we have
\begin{itemize}
\item[a)] $q$ is a coequaliser of $\xymatrix{\cR \ar@<0.6ex>[r]^{f}\ar@<-0.6ex>[r]_{g} & \cX}$,
\item[b)] $q'$ is a coequaliser of $\xymatrix{\cR' \ar@<0.6ex>[r]^{f'}\ar@<-0.6ex>[r]_{g'} & \cX'}$,
\item[c)] $\theta$ is a coequaliser of $\xymatrix{\cS \ar@<0.6ex>[r]^{a}\ar@<-0.6ex>[r]_{b} & \cX'}$, and
\item[d)] $s$ is a coequaliser of $\xymatrix{\cS \ar@<0.6ex>[r]^{q'a}\ar@<-0.6ex>[r]_{q'b} & \cY'}$.
\end{itemize}
Suppose further that there exists a morphism $\theta' : \cR' \to \cR$ such that:
\begin{itemize}\setcounter{enumi}{4}
\item[e)] $\theta'$ is an epimorphism, and
\item[f)] $\theta f' = f \theta'$ and $\theta g' = g \theta'$.
\end{itemize}
Then there exists an isomorphism $\varphi : \cY \stackrel{\cong}{\longrightarrow} \cZ$ such that $sq' = \varphi q \theta$.
\end{lemma}
\begin{proof}  By d), we have $sq'a = sq'b$. Hence by c), there is a unique morphism $\alpha : \cX \to \cZ$ such that $\boxed{sq' = \alpha \theta}$. Using this, we have $\alpha f \theta' \stackrel{f)}{=} \alpha \theta f' = sq' f' \stackrel{b)}{=} sq'g' = \alpha \theta g' \stackrel{f)}{=} \alpha g \theta'.$ Then $\alpha f = \alpha g$ by e), so by a), there is a unique morphism $\varphi : \cY \to \cZ$ such that $\boxed{\alpha = \varphi q}$.

Next, we have $q \theta f' \stackrel{f)}{=} q f \theta' \stackrel{a)}{=} q g \theta' \stackrel{f)}{=} q \theta g'$, so by b), there is a unique morphism $\tau : \cY' \to \cY$ such that $\boxed{q \theta = \tau q'}$. Then $\tau q' a = q \theta a \stackrel{c)}{=} q \theta b = \tau q' b$, so by d), there is a unique morphism $\psi : \cZ \to \cY$ such that $\boxed{\psi s = \tau}$. We will show that $\varphi$ and $\psi$ are mutually inverse.

Firstly, $\varphi \psi s q' = \varphi \tau q' = \varphi q \theta = \alpha \theta = s q'$. But $q'$ and $s$ are coequalisers by b) and d), hence they are epimorphisms by Lemma \ref{lem:CoeqEpi}. Therefore $\varphi \psi = 1_{\cZ}$. Secondly, $\psi \varphi q \theta = \psi \alpha \theta = \psi sq' = \tau q' = q\theta$. Since $\theta$ and $q$ are coequalisers by c) and a), they are epimorphisms by Lemma \ref{lem:CoeqEpi}. Hence $\psi \varphi = 1_{\cY}$. Finally, $\varphi q \theta = \alpha \theta = sq'$.\end{proof}

We omit the proof of the following standard result.

\begin{lemma} \label{lem: CoprodOfCoeq} Let $\cI$ be a set and let $\cC$ be a category with coproducts. Suppose that
\[ \left\{\xymatrix{ \cR_i \ar@<0.6ex>[r]^{a_i} \ar@<-0.6ex>[r]_{b_i} & X_i \ar[r]^{q_i} & Y_i }, i \in I \right\}\]
is a family of coequaliser diagrams in $\cC$. Then
\[ \xymatrix{ \coprod\limits_{i \in I} \cR_i \ar@<0.6ex>[rr]^{\coprod\limits_{i \in I}a_i} \ar@<-0.6ex>[rr]_{\coprod\limits_{i \in I}b_i}  && \coprod\limits_{i \in I} X_i \ar[rr]^{\coprod\limits_{i \in I} q_i} && \coprod\limits_{i \in I} Y_i }\]
is also a coequaliser diagram in $\cC$.
\end{lemma}

\subsection{An alternative stability proof} By Prop. \ref{prop:Krull-G}, $\Mod_k(G)$ has Krull dimension $1$, so the only non-trivial term in the Krull-dimension filtration of $\Mod_k(G)$ is $\Mod_k(G)_0$. Its stability then follows from our general result, Prop. \ref{prop:KdimStable}. Here we give an alternative, more direct, argument for the stability of $\Mod_k(G)_0$. Our argument is inspired by Pa\v{s}k\={u}nas' proof of the corresponding result for the group $GL_2(\mathbb{Q}_p)$ and representations with a fixed central character (\cite{Pas} Prop.\ 5.16).

An $H$-module is called \emph{locally finite} if each of its elements is contained in a submodule of finite length. By Remark \ref{rem:locnoeth-has-Krull} the locally finite $H$-modules form the localising subcategory $\Mod(\cH)_0$ of objects of Krull dimension $0$.

\begin{lemma}\label{H-stable}
   $\Mod(H)_0$ is stable.
\end{lemma}
\begin{proof}
By \cite{MCR} Cor. 13.1.13(iii) and Thm. 13.10.3(i), any simple module over the affine PI $k$-algebra $H$ is finite dimensional over $k$. Hence a locally finite module is the same as a module in which every element is contained in a finite dimensional submodule. It easily follows that $\Mod(H)_0$ satisfies the criterion \cite{Gab} Prop.\ V.6.12.
\end{proof}

Correspondingly, a $G$-representation in $\Mod_k(G)$ is locally finite if each of its elements is contained in a subrepresentation of finite length.

\begin{lemma}\label{lem:adm}
  If the representation $V$ in $\Mod_k(SL_2(\mathbb{Z}_p))$ is admissible then also its injective hull in $\Mod_k(SL_2(\mathbb{Z}_p))$ is admissible.
\end{lemma}
\begin{proof}
Put $K := SL_2(\mathbb{Z}_p)$. It is well know that a representation $V$ in $\Mod_k(K)$ is admissible if and only if its Pontrjagin dual $V^\vee$ is finitely generated as a module over the completed group ring $k[[K]]$. The inclusion $V \hookrightarrow E(V)$ into an injective hull dualizes to a projective cover $E(V)^\vee \twoheadrightarrow V^\vee$. Since the ring $k[[H]]$ is noetherian (cf.\ the explanations in the proof of \cite{Sch} Prop.\ 5) this cover $E(V)^\vee$ must be finitely generated as a $k[[H]]$-module. It follows that $E(V)$ is admissible.
\end{proof}

\begin{lemma}\label{lem:fl-adm}
  For a representation $V$ in $\Mod_k(G)$ we have:
\begin{itemize}
  \item[a)] $V$ is of finite length if and only if $V$ is finitely generated and admissible.
  \item[b)] $V$ is locally finite if and only if it is locally admissible (in the sense of \cite{EP}).
\end{itemize}
\end{lemma}
\begin{proof}
a) Suppose that $V$ is of finite length. Then it obviously is finitely generated. In order to see that $V$ is admissible we may assume that it is irreducible. But then, by Thm.\ \ref{equivalence} the $H$-module $V^I$ is finite dimensional, which means that $V$ is admissible. Now assume, vice versa, that $V$ is finitely generated and admissible. Using a filtration of $V$ as in Lemma \ref{fin-filt} and the fact that admissibility is preserved by passing to subquotients (\cite{Eme} Prop.\ 2.2.13) we may assume that $V$ lies in $\Mod^I(G)$. Since $V^I$ is finite dimensional the equivalence in Thm.\ \ref{equivalence} tells us that $V$ is of finite length.

b) If $V$ is locally finite then, by a), it is locally admissible. Suppose therefore that $V$ is locally admissible. But then it is the union of admissible and finitely generated subrepresentations. Again by a) it follows that $V$ is the union of subrepresentations of finite length.
\end{proof}

\begin{lemma}\label{lem:inj}
   If $V$ is an injective object in $\Mod_k(G)_0$ then $V$ is also an injective object in $\Mod_k(SL_2(\mathbb{Z}_p))$.
\end{lemma}
\begin{proof}
The proof is almost literally the same as for \cite{EP} Cor.\ 3.8. For the convenience of the reader we note:
\begin{itemize}
  \item[--] Use Lemma \ref{lem:fl-adm}.b to replace the assumption that $V$ is locally finite by the assumption that $V$ is locally admissible, which loc.\ cit.\ uses.
  \item[--] The group $SL_2(\mathbb{Q}_p)$ is the amalgam of $SL_2(\mathbb{Z}_p)$ and $(\begin{smallmatrix}
            0 & 1 \\ p & 0 \end{smallmatrix}) SL_2(\mathbb{Z}_p) (\begin{smallmatrix}
            0 & p^{-1} \\ 1 & 0 \end{smallmatrix})$ along the Iwahori subgroup (cf.\ \cite{Ser} II.4.1 Cor.\ 1).
\end{itemize}
We also point out that in loc.\ cit.\ a version of Lemma \ref{lem:adm} is used without mentioning it.
\end{proof}

\begin{proposition}\label{G-stable}
   $\Mod_k(G)_0$ is stable.
\end{proposition}
\begin{proof}
By Gabriel's criterion (cf.\ \cite{Ste} Prop.\ VI.7.1) we have to check the following: Let $V_0 \subseteq V$ be any essential extension in $\Mod_k(G)$ (in particular $V_0 \neq \{0\}$) such that $V_0$ is locally finite. We then have to show that $V$ is locally finite as well. First of all, by replacing $V$ by an injective hull, we may assume that $V$ is an injective object. Secondly since $\Mod_k(G)_0$ is localising, by possibly enlarging $V_0$, we may assume that $V_0$ is the maximal locally finite subrepresentation of $V$. It is then straightforward to see that $V_0$ is an injective object in $\Mod_k(G)_0$. Hence using Frobenius reciprocity we deduce from Lemma \ref{lem:inj} that
\begin{equation}\label{f:zero}
  \Ext_{\Mod_k(G)}^1(\ind_K^G(\sigma), V_0) = \Ext_{\Mod_k(K)}^1(\sigma,V_0) = 0
\end{equation}
for any smooth representation $\sigma$ of the subgroup $K := SL_2(\mathbb{Z}_p)$.

Reasoning by contradiction we assume that $V_0 \neq V$. It suffices to construct a nonzero locally finite subrepresentation of $V/V_0$. For this we pick an irreducible $K$-subrepresentation $\sigma$ in the $K$-socle of $V/V_0$. We also pick a vector $v \in (V/V_0)^I$ which generates $\sigma$ as a $K$-representation. We let $\bar{X} \subseteq V/V_0$ denote the $G$-subrepresentation generated by $v$ and by $X$ its preimage in $V$. We obtain:
\begin{itemize}
  \item[a)] The short exact sequence $0 \rightarrow V_0 \rightarrow X \rightarrow \bar{X} \rightarrow 0$ in $\Mod_k(G)$ does not split since $V_0 \hookrightarrow X$ is an essential extension.
  \item[b)] By the Frobenius reciprocity equality $\Hom_{k[K]}(\sigma,\bar{X}) = \Hom_{k[G]}(\ind_K^G(\sigma),\bar{X})$ the inclusion $\sigma \subseteq \bar{X}$ corresponds to a $G$-homomorphism $\bar{\psi} : \ind_K^G(\sigma) \twoheadrightarrow \bar{X}$, which is surjective since $v$ lies in its image.
\end{itemize}
Applying \eqref{f:zero} to a) shows that the map $\Hom_{k[G]}(\ind_K^G(\sigma),X) \twoheadrightarrow \Hom_{k[G]}(\ind_K^G(\sigma),\bar{X})$ is surjective, so that $\bar{\psi}$ has a preimage $\psi : \ind_K^G(\sigma) \rightarrow X$. If $\bar{\psi}$ would also be injective then $\psi \circ \bar{\psi}^{-1} : \bar{X} \rightarrow X$ would split the sequence in a). It follows that $\bar{\psi}$ is a proper quotient map. Since $\bar{\psi}$ is a map in $\Mod_k^I(G)$ the equivalence of categories in Thm.\ \ref{equivalence} implies that the sequence
$\bar{\psi}^I : \ind_K^G(\sigma)^I \twoheadrightarrow \bar{X}^I$ is a proper quotient map of $H$-modules. We claim that any proper $H$-quotient of $\ind_K^G(\sigma)^I$ has finite $k$-dimension. Again Thm.\ \ref{equivalence} then implies that $\bar{X}$ has finite length, which would contradict the maximality of $V_0$.

In order to determine the $H$-module structure of $\ind_K^G(\sigma)^I$ we introduce the finite dimensional subalgebra $H_0 := \End_{k[K]}(k[K/I])^{\op}$ of $H$. It is well known that the $H_0$-module of invariants $\sigma^I$ is one dimensional and therefore provides a character $\chi : H_0 \rightarrow k$. By \cite{Oll} Lemma 3.6 there is an isomorphism of $H$-modules $\ind_K^G(\sigma)^I \cong H \otimes_{H_0} \chi$.\footnote{In \cite{Oll} the field $k$ is assumed to be an algebraic closure of $\mathbb{F}_p$. But in our situation all irreducible mod $p$ representations of $SL_2(\mathbb{F}_p)$ are defined over $\mathbb{F}_p$, which makes this assumption unnecessary.}

On the other hand recall the notations introduced in $\S \ref{sec: propIHA}$. Note $\tau_0$ and $e_1$ lie in $H_0$ and that $Z(H)$ contains the polynomial ring $k[\zeta]$ by Lemma \ref{lem: zeta}. As a consequence of \cite{OS18}\ Cor.\ 3.4 we have the isomorphism of $k[\zeta]$-modules
\begin{equation*}
  k[\zeta] \oplus k[\zeta] \xrightarrow{\;\cong\;} H \otimes_{H_0} \chi
\end{equation*}
which sends $(1,0)$ to $1 \otimes 1$ and $(0,1)$ to $\tau_1 \otimes 1$.

Let $Q := k(\zeta)$ be the field of fractions of $k[\zeta]$ and let $V := Q \otimes_{k[\zeta]} (H \otimes_{H_0} \chi)$.
Then $V$ is a 2-dimensional vector space over $Q$ with basis ${v,w}$ where $v := 1 \otimes 1 \otimes 1$ and $w := 1 \otimes \tau_1 \otimes 1$. Suppose now that we have a nonzero $H$-submodule of infinite codimension in $H \otimes_{H_0} \chi$. It gives rise to a line in the vector space $V$ which is respected by the $Q$-linear action of $\tau_0$ and $\tau_1$. Let $a := \chi(e_1)$ and $b := \chi(\tau_0)$. So, $a$ is $0$ or $1$ and $b$ is some element of $k$. We calculate that the matrix $A$ of the action of $\tau_0$ with respect to the basis ${v,w}$ is
$(\begin{smallmatrix}
b & \zeta - a(b+1)  \\
0 &    -(a+b)
\end{smallmatrix})$
and the matrix $B$ of the action of $\tau_1$ with respect to the same basis is
$(\begin{smallmatrix}
0 &  0  \\
1 & -a
\end{smallmatrix})$.
The minimum polynomial of $B$ is visibly $X^2 + a X$. If $a = 0$ then $B$ is nilpotent and the only $\tau_1$-stable line in $V$ is spanned by $w$. If $a = 1$ then $B$ is diagonalizable, and there is an additional $\tau_1$-stable line spanned by $v + w$. So we just need to check that these two lines are not $\tau_0$-stable. Passing to the matrices we need to check that the column vectors $(\begin{smallmatrix}0\\1\end{smallmatrix})$ and $(\begin{smallmatrix}a\\1\end{smallmatrix})$ are not $A$-eigenvectors:
\begin{enumerate}
  \item[--] $A(\begin{smallmatrix}0\\1\end{smallmatrix}) = (\begin{smallmatrix}\zeta - a(b+1)\\ -(a+b)\end{smallmatrix})$ is not in $Q(\begin{smallmatrix}0 \\ 1 \end{smallmatrix})$ since $\zeta - a(b+1)$ is not zero in $Q$;
  \item[--] $A(\begin{smallmatrix}a\\1\end{smallmatrix}) = (\begin{smallmatrix}\zeta-a\\-(a+b)\end{smallmatrix}) = c (\begin{smallmatrix}a\\1\end{smallmatrix})$ for some $c \in Q$ implies $c = -(a+b) \in k$, but also $\zeta - a = ca \in k$, which is impossible. \qedhere
\end{enumerate}
\end{proof}


\end{document}